\documentclass{amsart}
\usepackage{amssymb,
enumerate,
mathrsfs,
hyperref
}
\hypersetup{colorlinks=true}
\numberwithin{equation}{section}
\usepackage[T1]{fontenc}

\renewcommand{\L}{\mathcal{L}}
\newcommand{\U}{\mathcal{U}}
\newcommand{\RR}{\mathbb{R}}
\newcommand{\QQ}{\mathbb{Q}}
\newcommand{\pre}[2]{{}^{#1} #2}
\newcommand{\set}[2]{\{ #1 \mid #2 \}}

\newcommand{\id}{\operatorname{id}}
\newcommand{\pow}{\mathscr{P}}
\newcommand{\Mod}{\operatorname{Mod}}
\newcommand{\leng}{\operatorname{length}}
\newcommand{\pred}{\operatorname{pred}}
\newcommand{\cf}{{\operatorname{cf}}}
\newcommand{\inj}{{\operatorname{inj}}}
\newcommand{\rk}{{\operatorname{rk}}}

\newcommand{\ZF}{{\sf ZF}}
\newcommand{\AC}{{\sf AC}}

\newcommand{\I}[2]{I_{#1}^{\vec{#2}}}
\newcommand{\II}[1]{I^{\vec{#1}}}
\newcommand{\isom}{\cong}
\newcommand{\RCT}{\mathrm{RCT}}

\newtheorem{theorem}{Theorem}[section]
\newtheorem{lemma}[theorem]{Lemma}
\newtheorem{corollary}[theorem]{Corollary}
\newtheorem{proposition}[theorem]{Proposition}
\newtheorem{question}[theorem]{Question}
\newtheorem{fact}[theorem]{Fact}
\newtheorem{claim}{Claim}[theorem]
\theoremstyle{definition}
\newtheorem{defin}[theorem]{Definition}
\newtheorem{notation}[theorem]{Notation}
\theoremstyle{remark}
\newtheorem{remark}[theorem]{Remark}

\begin{document}

\title[On isometry and isometric embeddability]{On isometry and isometric embeddability\\ between ultrametric Polish spaces}
\date{October 8, 2017}
\author[R.~Camerlo]{Riccardo Camerlo}
\address{Dipartimento di scienze matematiche \guillemotleft{Joseph-Louis Lagrange}\guillemotright, Politecnico di Torino, Corso Duca degli Abruzzi 24, 10129 Torino --- Italy}
\email{riccardo.camerlo@polito.it}
\author[A.~Marcone]{Alberto Marcone}
\address{Dipartimento di scienze matematiche, informatiche e fisiche, Universit\`a di Udine, Via delle Scienze 208, 33100 Udine --- Italy}
\email{alberto.marcone@uniud.it}
\author[L.~Motto Ros]{Luca Motto Ros}
\address{Dipartimento di matematica \guillemotleft{Giuseppe Peano}\guillemotright, Universit\`a di Torino, Via Carlo Alberto 10, 10121 Torino --- Italy}
\email{luca.mottoros@unito.it}
\thanks{Marcone's research was supported by PRIN 2009 Grant
\lq\lq Modelli e Insiemi\rq\rq\ and PRIN 2012 Grant \lq\lq Logica, Modelli e
Insiemi\rq\rq. Until September 2014 Motto Ros was a member of the Logic
Department of the Albert-Ludwigs-Universit\"at Freiburg, which supported him
at early stages of this research. After that, he was supported by the Young
Researchers Program ``Rita Levi Montalcini'' 2012 through the project ``New
advances in Descriptive Set Theory''.}

\begin{abstract}
We study the complexity with respect to Borel reducibility of the relations
of isometry and isometric embeddability between ultrametric Polish spaces
for which a set $D$ of possible distances is fixed in advance. These are,
respectively, an analytic equivalence relation and an analytic quasi-order
and we show that their complexity depends only on the order type of $D$.
When $D$ contains a decreasing sequence, isometry is Borel bireducible with countable
graph isomorphism and isometric embeddability has maximal complexity among
analytic quasi-orders. If $D$ is well-ordered the situation is more
complex: for isometry we have an increasing sequence of Borel equivalence
relations of length $\omega_1$ which are cofinal among Borel equivalence
relations classifiable by countable structures, while for isometric
embeddability we have an increasing sequence of analytic quasi-orders of
length at least $\omega+3$.

We then apply our results to solve various open problems in the literature.
For instance, we answer a long-standing question of Gao and Kechris by
showing that the relation of isometry on locally compact ultrametric Polish
spaces is Borel bireducible with countable graph isomorphism.
\end{abstract}

\maketitle

\tableofcontents

\section{Introduction}

A common problem in mathematics is to classify interesting objects up to some
natural notion of equivalence. More precisely, one considers a class of
objects \( X \) and an equivalence relation \( E\) on \( X \), and tries to
find a set of complete invariants \( I \) for \( (X,E) \). To be of any use,
such an assignment of invariants should be as simple as possible. In most
cases, both \( X \) and \( I \) carry some intrinsic Borel structures, so
that it is natural to ask the assignment to be a Borel measurable map.

A classical example  is the problem  of classifying separable complete metric
spaces, called \emph{Polish metric spaces}, up to isometry. In
\cite{gromov1999} Gromov showed for instance that one can classify compact
Polish metric spaces using (essentially) elements of \( \RR \) as complete
invariants; in modern terminology, we say that the corresponding
classification problem is smooth. However, as pointed out by Vershik in
\cite{vershik1998} the problem of classifying arbitrary Polish metric spaces
is \guillemotleft an enormous task\guillemotright, in particular it is far
from being smooth. Thus it is natural to ask how complicated is such a
classification problem.

A natural tool for studying the complexity of classification problems is the
notion of Borel reducibility introduced in \cite{Friedman1989} and in
\cite{HKL}: we say that a classification problem \( (X,E) \) is \emph{Borel
reducible} to another classification problem \( (Y,F) \) (in symbols, \( E
\leq_B F \)) if there exists a Borel measurable function \( f \colon X \to Y
\) such that \( x \mathrel{E} x' \iff f(x) \mathrel{F} f(x') \) for all \(
x,x' \in X \). Intuitively, this means that the classification problem \(
(X,E ) \) is not more complicated than \( (Y,F) \): in fact, any assignment
of complete invariants for \( (Y,F) \) may be turned into an assignment for
\( (X,E) \) by composing  with \( f \). A comprehensive reference for the
theory of Borel reducibility is \cite{gaobook}.

In \cite{Gao2003} (see also~\cite{Clemens2001,clemensisometry}), Gao and
Kechris were able to determine the exact complexity of the classification
problem for isometry on  arbitrary Polish metric spaces with respect to Borel
reducibility: it is Borel bireducible with the most complex orbit equivalence
relation (i.e.\ every equivalence relation induced by a Borel action of a
Polish group on a Polish space Borel reduces to it). Then, extending the work
of Gromov on compact Polish metric spaces, they turned their attention to
some other natural subclasses of spaces. Among these, the cases of locally
compact Polish metric spaces and of ultrametric%
\footnote{Recall that a metric \( d \) on a space \( X \) is called an
\emph{ultrametric} if it satisfies the following strengthening of the
triangular inequality: \( d(x,y) \leq \max \{ d(x,z),d(z,y) \} \) for all \(
x,y,z \in X \).} Polish spaces are particularly important. Notice that
ultrametric Polish spaces naturally occur in various parts of mathematics and
computer science. For example, the space \( \QQ_p \) of \( p \)-adic numbers
with the metric induced by the evaluation map \( | \cdot |_p \) (for any
prime number \( p \)) and, more generally, the completion of any countable
valued field are ultrametric Polish spaces. The same is true for the space \(
\pre\omega\Sigma \) of infinite words over a finite alphabet \( \Sigma \)
appearing  e.g.\ in automaton theory, which is usually equipped with the
metric \( d \) measuring how much two given words \( x,y \in \pre\omega\Sigma
\) are close to each other, i.e.\ \( d(x,y) = 2^{-n} \) with $n$ least such
that $x(n) \neq y(n)$. Indeed the space \( \pre\omega\Sigma \) is ultrametric
complete for \emph{any} (finite or infinite) alphabet \( \Sigma \), and is
Polish if and only if \( \Sigma \) is at most countable.

While the complexity of the isometry relation on arbitrary locally compact
Polish metric spaces has not yet been determined,  \cite[Theorem
4.4]{Gao2003} shows that for the case of ultrametric Polish spaces the
corresponding classification problem is Borel bireducible with  isomorphism
on countable graphs, the most complex isomorphism relation for classes of
countable structures.
Countable graph isomorphism is far from being smooth but it is
strictly simpler, with respect to Borel reducibility, than isometry on
arbitrary Polish metric spaces. For our purposes, it is important to notice
that the proof of the result on ultrametric Polish spaces crucially uses
spaces whose set of distances has \( 0 \) as a limit point and that are far
from being locally compact. Motivated by these observations, Gao and Kechris
devote a whole chapter of their monograph to the study of the isometry
relation on Polish metric spaces that are both locally compact and
ultrametric, obtaining some interesting partial results which seem to
transfer this problem to the study of the isometry relation on discrete
ultrametric Polish spaces. However, the latter problem remained unsolved and
hence they asked:

\begin{question} \cite[\S 8C and Chapter 10]{Gao2003} \label{mainquest}
What is the exact complexity of the isometry problem for discrete ultrametric
Polish spaces and for arbitrary locally compact ultrametric Polish spaces?
\end{question}

As a partial result, they isolated two lower bounds for these isometry
relations, namely isomorphism of countable well-founded trees and isomorphism
of countable trees with only countably many infinite branches.

Question~\ref{mainquest} was raised again (and another lower bound was
proposed) in \cite{ClemensPreprint}, where Clemens studies the complexity of
isometry on the collection of Polish metric spaces using only distances in a
set  \( A \subseteq \RR^+ \) fixed in advance. A related question is also
raised in \cite{GaoShao2011}, where Gao and Shao consider Clemens' problem
restricted to ultrametric Polish spaces:

\begin{question} \cite[\S 8]{GaoShao2011} \label{questgaoshao}
Given a countable%
\footnote{The countability of \( D \) is a necessary
requirement when dealing with ultrametric Polish spaces, see
Lemma~\ref{countablymanydistances}.} set of distances \( D \subseteq \RR^+ \)
such that \( 0 \) is not a limit point of \( D \), what is the complexity of
the isometry relation between ultrametric Polish spaces using only distances
from \( D \)?
\end{question}

Since all spaces as in Question~\ref{questgaoshao} are necessarily discrete
(and hence locally compact), it is clear that the isometry relations
considered there constitute lower bounds for the isometry relations of
Question~\ref{mainquest}.

In this paper, we address Question~\ref{questgaoshao} and show that our
solution of this problem allows us to answer also Question~\ref{mainquest}.
Moreover, we also consider the analogous problem concerning the complexity of
the quasi-order of isometric embeddability between ultrametric Polish spaces
using only distances from a given countable \( D \subseteq \RR^+ \). (The
formal setup for these problems is described in Section \ref{sect:setup}.)
Concerning isometric embeddability on arbitrary ultrametric Polish spaces, it
is already known that the relation is as complicated as possible: Louveau and
Rosendal (\cite{louros}) showed that it is a complete analytic quasi-order,
and we strengthened this in \cite{cammarmot} by showing that it is in fact
invariantly universal (see definitions below). However, the proofs of the
results in \cite{louros,cammarmot} again use in an essential way ultrametric
spaces with distances converging to zero.  This naturally raises the
following question, which is somewhat implicit in~\cite{GaoShao2011}:

\begin{question} \label{questemb}
Given a countable set of distances \( D \subseteq \RR^+ \) such that \( 0 \)
is not a limit point of \( D \), what is the complexity of the isometric
embeddability relation between ultrametric Polish spaces using only distances
from \( D \)?
\end{question}

We will show that actually the answers to both Question~\ref{questgaoshao}
and~\ref{questemb} depend only on whether $D$ contains a decreasing sequence
and, when it does not, on the order type (a countable ordinal) of $D$.
Thus isometry between ultrametric Polish spaces using only distances in $D$ is actually the same for any ill-founded $D$, regardless of the limits of the decreasing sequences in $D$, and the same is true for isometric embeddability.

In Sections \ref{isomif} and \ref{embif} we prove that if $D$ contains a
decreasing sequence of real numbers then our relations attain the maximal
possible complexity: isometry is Borel bireducible with countable graph isomorphism,
and isometric embeddability is invariantly universal, and therefore complete
for analytic quasi-orders. This implies that the isometry relation on every
class of ultrametric Polish spaces containing all spaces using only distances
from some specific $D$ with a decreasing sequence is also Borel bireducible
with countable graph isomorphism. By choosing such a $D$ so that $0$ is not one of its
limit points we obtain only discrete spaces. Thus our result in particular
answers Question~\ref{mainquest} (see
Corollary~\ref{cor:mainquest}): the relations of isometry on discrete
ultrametric Polish spaces and on locally compact ultrametric Polish spaces
are both Borel bireducible with countable graph isomorphism. This also shows that
isometry on the classes of countable or $\sigma$-compact ultrametric Polish
spaces is also Borel bireducible with countable graph isomorphism. We will also observe
that many of the lower bounds previously isolated in the literature are in
fact not sharp.

The situation is more complex when $D$ is a well-ordered set of real numbers
of order type $\alpha$. In Section \ref{isomwf} we show that when $\alpha $
ranges over non-null countable ordinals, the corresponding isometry relations
form a strictly increasing sequence of Borel equivalence relations, cofinal
among Borel equivalence relations classifiable by countable structures. In
Section \ref{sect:wf} we prove that if $\alpha \leq \omega+1$ then the
corresponding quasi-order of isometric embeddability is not complete
analytic. In fact for $1 \leq \alpha \leq \omega+2$ we have a strictly
increasing chain of quasi-orders; these are Borel exactly when $\alpha \leq
\omega$. We do not know whether for some $\alpha \geq \omega+2$ the
corresponding quasi-order of isometric embeddability is already complete for
analytic quasi-orders.

The main tool we use to deal with well-ordered sets of distances is a jump
operator $S \mapsto S^\inj$ defined on quasi-orders, which seems to be of
independent interest. This is defined and studied in Section \ref{sect:jump},
which starts with a review of Rosendal's jump operator $S \mapsto S^\cf$ in
Subsection \ref{subsect:cf} before introducing our variant in Subsection
\ref{subsect:inj}. In particular, when $S$ is Borel, $S^\cf$ is always Borel
while we will show that $S^\inj$ can be proper analytic (Subsection
\ref{subsec:injnotB}) but cannot be complete for analytic quasi-orders
(Subsection \ref{subsec:eqrelinj}).

\section{Terminology and notation}

We now describe the basic terminology and notation used in the paper. Recall
that a quasi-order is a reflexive and transitive binary relation. Any
quasi-order $S$ on a set $X$ induces an equivalence relation on \( X \), that
we denote by $E_S$, defined by $x \mathrel{E_S} x'$ if and only if
$x\mathrel{S}x'$ and $x'\mathrel{S}x$. A quasi-order is a \emph{well
quasi-order} (wqo for short) if it is well-founded and contains no infinite
antichains. In the 1960's Nash-Williams introduced the notion of \emph{better
quasi-order} (bqo for short): the definition of bqo is quite involved, but to
understand the references to this notion in this paper it suffices to know
that, as the terminology suggests, every bqo is indeed a wqo, and to use as
black boxes some closure properties of bqo's.

If $\mathcal{A}$ is a countably generated $\sigma$-algebra of subsets of $X$
that separates points we refer to the members of $\mathcal{A}$ as Borel sets
(indeed, as shown e.g.\ in \cite[Proposition 12.1]{Kechris1995}, in this case
$\mathcal{A}$ is the collection of Borel sets of some separable metrizable
topology on $X$). A map between two sets equipped with a collection of Borel
sets is Borel (measurable) if the preimages of Borel sets of the target space
are Borel sets of the domain. The space $(X, \mathcal{A})$ is standard Borel
if $\mathcal{A}$ is the collection of Borel sets of some Polish (i.e.\
separable and completely metrizable) topology on $X$. Except where explicitly
noted, in this paper we will always deal with standard Borel spaces.

Let $R$ and $S$ be binary relations on spaces equipped with a collection of
Borel sets $X$ and $Y$, respectively. We say that $R$ is \emph{Borel
reducible} to $S$, and we write $R\leq_BS$, if there is a Borel function $f
\colon X\to Y$ such that $x \mathrel{R} x'$ if and only if $f(x) \mathrel{S}
f(x')$ for all $x,x'\in X$. If $R\leq_BS$ and $S\leq_BR$ we say that $R$ and
$S$ are \emph{Borel bireducible} and we write $R\sim_BS$. If on the other
hand we have $R\leq_BS$ and $S\nleq_BR$ we write $R <_B S$.

If $\Gamma $ is a class of binary relations on standard Borel spaces and
$S\in\Gamma$, we say that $S$ is \emph{complete for $\Gamma$} if $R\leq_BS$
for all $R\in\Gamma$. Some classes $\Gamma$ we consider in this paper are the
collection of all analytic equivalence relations and the collection of all
analytic quasi-orders. If $\Gamma$ is the class of equivalence relations
classifiable by countable structures, the canonical example of an equivalence
relation complete for $\Gamma$ is countable graph isomorphism. Thus an equivalence
relation on a standard Borel space which is Borel bireducible with countable graph
isomorphism is in fact complete for equivalence relations classifiable by
countable structures.

Following the main result of \cite{frimot}, in \cite{cammarmot} we introduced
the following notion.

\begin{defin}\label{def:invariantlyuniversal}
Let the pair \( (S,E) \) consist of an analytic quasi-order \( S \) and an
analytic equivalence relation \( E \subseteq S \), with both relations
defined on the same standard Borel space $X$. Then $(S,E)$ is
\emph{invariantly universal} (for analytic quasi-orders) if for any analytic
quasi-order $R$ there is a Borel $B\subseteq X$ invariant under $E$ such that
$R\sim_BS\restriction B$.

Whenever the equivalence relation $E$ is clear from the context (in this
paper this usually means that $E$ is isometry between metric spaces in the
class under consideration) we just say that $S$ is invariantly universal.
\end{defin}

Notice that if $(S,E)$ is invariantly universal, then $S$ is, in particular,
complete for analytic quasi-orders.

We now extend the notions of classwise Borel embeddability and isomorphism
between equivalence relations introduced in \cite{MR12} to pairs consisting
of a quasi-order and an equivalence relation (\( E \simeq_{cB} F \),
respectively $E \sqsubseteq_{cB} F$, in the notation of \cite{MR12} is the
same as $(E,E)$ is classwise Borel isomorphic to, respectively embeddable in,
$(F,F)$ in our terminology). Given a pair \( (S,E) \) consisting of a
quasi-order \( S \) and an equivalence relation \( E \subseteq S \) on a set
$X$, we denote by \( S/E \) the \( E \)-quotient of \( S \), i.e.\ the
quasi-order on $X/E$ induced by $S$. If $F$ and $E$ are equivalence relations
on sets $X$ and $Y$ and $f \colon X/F\to Y/E$, then a \emph{lifting} of $f$
is a function $ \hat f \colon X\to Y$ such that $[ \hat f (x)]_E=f([x]_F)$
for every $x\in X$.

\begin{defin}\label{def:cB}
Let $(R,F)$ and $(S,E)$ be pairs consisting of a quasi-order and an
equivalence relation on some standard Borel spaces, with $F \subseteq
R$ and $E \subseteq S$.

We say that $(R,F)$ is \emph{classwise Borel isomorphic} to $(S,E)$, in
symbols $(R,F) \simeq_{cB} (S,E)$, if there is an isomorphism of quasi-orders
\( f \) between \( R/F \) and \( S/E \) such that both \( f \) and \( f^{-1}
\) admit Borel liftings.

We say that $(R,F)$ is \emph{classwise Borel embeddable} in $(S,E)$, in
symbols $(R,F) \sqsubseteq_{cB} (S,E)$, if there is an \( E \)-invariant
Borel subset \( B \) of the domain of \( S \) such that \( (R,F) \simeq_{cB}
(S \restriction B, E \restriction B) \).

Again, when the equivalence relations $F$ and $E$ are clear from the context
we just say that $R$ is classwise Borel isomorphic to or embeddable into $S$,
and write \( R \simeq_{cB} S \) or \( R \sqsubseteq_{cB} S \).
\end{defin}

The relevance of these definitions lies in the observation that if $(R,F)$ is
invariantly universal and $(R,F) \sqsubseteq_{cB} (S,E)$ then $(S,E)$ is
invariantly universal as well. This will be used e.g.\ for proving
Theorem~\ref{theorconvergingto0} and Theorem~\ref{theorillfounded2}.\medskip

We end this section by looking at Polish ultrametric preserving functions,
which will be used later. Pongsriiam and Termwuttipong (\cite{PT}) studied
\emph{ultrametric preserving functions} (i.e.\ functions $f$ such that for
every ultrametric $d$ on a space $X$, $f \circ d$ is still an ultrametric on
$X$) and showed that $f \colon \RR^+ \to \RR^+$ is ultrametric preserving if
and only if it is non-decreasing and such that $f^{-1}(0)=\{0\}$. For $f:
\RR^+ \to \RR^+$ to send Polish ultrametrics into Polish ultrametrics it is
necessary and sufficient that $f$ is an ultrametric preserving function
continuous at $0$ (notice that this implies that for every sequence $(x_n)_{n
\in \omega}$ such that $\lim_{n \to \infty} f(x_n)=0$ we have $\lim_{n \to
\infty} x_n=0$).

If we consider functions $f: A \to \RR^+$ for some $A \subseteq \RR^+$ (in
particular when $0$ is not an accumulation point of $A$) we need to
strengthen the continuity condition. Thus $f: A \to \RR^+$ is \emph{Polish
ultrametric preserving} (i.e.\ for every Polish ultrametric $d$ on a space
$X$ which uses only distances in $A$, $f \circ d$ is still a Polish
ultrametric on $X$) if and only if it is non-decreasing, such that
$f^{-1}(0)=\{0\}$, and satisfies $\lim_{n \to \infty} x_n=0$ if and only if
$\lim_{n \to \infty} f(x_n)=0$ for every sequence $(x_n)_{n \in \omega}$ of
elements of $A$.

\section{Jump operators}\label{sect:jump}
By a \emph{jump operator} we mean a mapping $S \mapsto S^J$ from a collection
of binary relations into itself satisfying $S \leq_B S^J$ and $R \leq_B S
\implies R^J \leq_B S^J$. Typically the jump is defined on equivalence
relations or quasi-orders. When we study jump operators we are
interested in finding conditions on $S$ that imply $S <_B S^J$ (notice that
if $S \leq_B R <_B S^J$ then $R <_B R^J$) and in using the jump to define
transfinite sequences of binary relations which are increasing with respect
to Borel reducibility.

\subsection{Rosendal's jump operator $S \mapsto S^{\cf}$}\label{subsect:cf}

A jump operator on equivalence relations \( E \mapsto E^+ \) was introduced
by H.~Friedman (see~\cite{Stanley1985}). If $E$ is defined on $X$, then $E^+$
is defined on $ \pre{\omega}{X} $ by letting
\[
(x_n)_{n \in \omega} \mathrel{E^+} ( y_n )_{ n \in \omega } \iff
(\forall n \exists m \, (x_n \mathrel{E} y_m)\wedge\forall n \exists m \, (y_n \mathrel{E} x_m)).
\]

In \cite{friedman2000} (where $E^+$ is denoted $ {\rm NCS}(X,E)$) it is
proved that this jump operator gives rise to a transfinite sequence of
equivalence relations that is strictly increasing with respect to $\leq_B$
and cofinal among Borel equivalence relations classifiable by countable
structures (see also \cite[\S12.2]{gaobook}, where a variant of this sequence
is called the Friedman-Stanley tower). This sequence is defined as follows.
Let $T(0)$ be the equivalence relation $(\omega , =)$. For $\alpha
<\omega_1$, let $T(\alpha +1)=(T({\alpha }))^+$. For $\lambda <\omega_1$
limit, let $T(\lambda )=\sum_{\beta <\lambda }T(\beta )$, where this sum is
the disjoint union of the equivalence relations $T(\beta )$.

Another well-known sequence of equivalence relations is obtained by
considering isomorphism on well-founded trees of bounded rank. Let $\mathcal
T_{\alpha }$ be the set of well-founded trees on $\omega $ of rank less than
$\alpha $: the equivalence relation of isomorphism on $\mathcal T_{\alpha }$
is quite easily seen to be Borel. H. Friedman and Stanley proved in
\cite{Friedman1989} that for $0<\alpha <\beta <\omega_1$, one has that
isomorphism on $ \mathcal T_{\alpha }$ is $<_B$ than isomorphism on $\mathcal
T_{\beta }$. The following fact is well-known and follows e.g.\ from
\cite[Theorem 13.2.5]{gaobook}.

\begin{proposition} \label{folklore}
For every Borel equivalence relation $E$ classifiable by countable structures
there exists $\alpha <\omega_1$ such that $E$ is Borel reducible to
isomorphism on $\mathcal T_{\alpha }$.
\end{proposition}

In~\cite{Rosend2005} Rosendal introduced the following jump operator  \( S
\mapsto S^{\cf} \) on quasi-orders which is an analogue in this wider context
of Friedman's operator on equivalence relations \( E \mapsto E^+ \) (in fact,
with our notation $E^+ = E_{E^{\cf}}$).

\begin{defin}[{\cite[Definition 4]{Rosend2005}}] \label{defin:rosendaljump}
Let $(X,S)$ be a quasi-order. We denote by \( (X, S)^{\cf} \) (or even just
by \( S^{\cf} \), when the space \( X \) is clear from the context) the
quasi-order on \( \pre{\omega}{X} \) defined by
\[
(x_n)_{n \in \omega} \mathrel{S^{\cf}} ( y_n )_{ n \in \omega } \iff \forall n \exists m (x_n \mathrel{S} y_m).
 \]
\end{defin}

It is immediate to check that \( S \mapsto S^{\cf} \) is indeed a jump
operator. Moreover \( S \) is a Borel (respectively, analytic) quasi-order if
and only if so is \( S^{\cf} \).

In~\cite{Rosend2005}, the author used the jump operator \( S \mapsto S^{\cf}
\) to define an  \( \omega_1 \)-sequence of Borel quasi-orders which is
cofinal among Borel quasi-orders.

\begin{defin}[\cite{Rosend2005}] \label{defin:rosendalsequence}
Let \( P_0 \) be the quasi-order \( (\omega,=) \). For \( \alpha < \omega_1
\) let \( P_{\alpha+1} = P_\alpha^{\cf} \). For \( \lambda < \omega_1 \)
limit let \( P_\lambda = \prod_{\beta < \lambda} P_\beta \), where the
product of relations is the relation defined componentwise on the Cartesian
product of the domains.
\end{defin}

\begin{theorem}[{\cite[Corollary 15]{Rosend2005}}] \label{th:rosendalsequence}
The sequence \( (P_\alpha)_{ \alpha < \omega_1} \) is strictly \( \leq_B
\)-increasing and cofinal among the Borel quasi-orders, i.e.\ for every Borel
quasi-order \( S \) there is \( \alpha < \omega_1 \) such that \( S \leq_B
P_\alpha \).
\end{theorem}

Notice that it immediately follows that the sequence \(
(E_{P_\alpha})_{\alpha < \omega_1 } \) of the associated equivalence
relations is cofinal among all Borel equivalence relations.

\begin{remark}
By monotonicity of \( S \mapsto S^{\cf} \) and the fact that Rosendal's
sequence is strictly increasing, if \( S \) is a Borel quasi-order such that
\( P_\alpha \leq_B S \leq_B P_{\alpha+1} \) for some \( \alpha < \omega_1 \),
then \( S <_B S^{\cf} \).
\end{remark}

To prove that the sequence \( (P_\alpha)_{\alpha < \omega_1} \) is \( \leq_B
\)-strictly increasing, one argues by induction on \( \alpha < \omega_1 \)
using in the successor step that if \( S \) is a Borel quasi-order on \( X \)
containing \( S \)-incompatible elements (i.e.\ $x$ and $y$ such that for no
$z$ we have $z \mathrel{S} x$ and $z \mathrel{S} y$), then \( S <_B S^{\cf}
\) (see~\cite[Proposition 6]{Rosend2005}). For some of the results of this
paper, we need to slightly improve this technical result as follows.

\begin{lemma}\label{lemmacf}
Let \( S \) be a Borel quasi-order on \( X \) and suppose that one of the
following conditions holds:
\begin{enumerate}[\quad(i)]
\item there exist \( S \)-incomparable elements \( x,y \) such that the
    restriction of \( S \) to \( \{ z \in X \mid z \mathrel{S}  x \land z
    \mathrel{S} y \} \) is well-founded;
\item the quotient $S/E_S$ has a well-founded infinite downward closed
    subset.
\end{enumerate}
Then \( S <_B S^{\cf} \).
\end{lemma}

\begin{proof}
To simplify the
notation, given \( x \in X \) denote by \( x^\infty \in \pre{\omega}{X} \)
the constant \(\omega\)-sequence with value \( x \).

We follow the ideas of the proof of~\cite[Proposition 6]{Rosend2005}. Suppose
$f \colon \pre{\omega}{X} \to X$ witnesses $S^{\cf}\leq_BS$. Recall
that~\cite[Proposition 5]{Rosend2005} asserts that for every \( \vec{x} = (
x_n )_{ n \in \omega } \in \pre{\omega}{X} \) there is \( k \in \omega \)
such that \( f(\vec{x}) \mathrel{S} x_k \).

Assume first that (i) holds and let $x,y$ be incomparable elements with a
well-founded set of common predecessors. Then $f(x^{\infty }) \mathrel{S} x$
and $f(y^{\infty }) \mathrel{S} y$; moreover, either $f(xy^{\infty })
\mathrel{S} x$ or $f(xy^{\infty }) \mathrel{S} y$. Since $x^{\infty
}\mathrel{S^{\cf}} xy^{\infty }$ and $y^{\infty }\mathrel{S^{\cf}} xy^{\infty
}$, one has $f(x^{\infty }) \mathrel{S} f(xy^{\infty })$ and $f(y^{\infty })
\mathrel{S} f(xy^{\infty })$. So at least one of $f(x^{\infty }),f(y^{\infty
})$ is a predecessor of both $x$ and $y$. Suppose, for instance, $f(x^{\infty
}) \mathrel{S} x$ and $f(x^{\infty }) \mathrel{S} y$ and let $z_0=f(x^{\infty
})$. Now notice that $z_0^{\infty }$ strictly precedes $x^{\infty }$ under
$S^{\cf}$ (because \( x \not \mathrel{S} y \) while \( z_0 \mathrel{S} y \)),
thus $z_1=f(z_0^{\infty })$ is a strict predecessor of $z_0 = f(x^\infty)$
under $S$. Arguing by induction in a similar way, we build a strictly
decreasing sequence $z_{n} = f(z_{n-1}^\infty)$ of common predecessors of $x$
and $y$.

If (ii) holds, let $Y \subseteq X/E_S$ be infinite, downward closed, and
well-founded, as in the case assumption. If \( Y \) contains two \( S/E_S
\)-incomparable elements, then we are done by case (i). Hence we can assume
without loss of generality that \( Y \) is a well-order and, since it is
infinite, that it contains an initial segment isomorphic to $\omega$. Let
$x_0 \mathrel{S} x_1 \mathrel{S} \dotsc$ be a strictly increasing sequence in
$X$ of representatives of such an initial segment, so that $x_0^{\infty},
x_1^{\infty},\dotsc, x_\omega = ( x_i )_{ i \in \omega }$ is a strictly
increasing sequence with respect to $S^{\cf}$.
Let $n\in\omega $ be such that $f(x_\omega) \mathrel{S}
x_n$. Then $(f(x_k^{\infty }))_{k \in \omega}$ would constitute a strictly
increasing sequence with respect to $S$ bounded by $x_n$, which is
impossible.
\end{proof}

\begin{remark}
\begin{enumerate}[(1)]
\item Conditions (i) and (ii) of Lemma~\ref{lemmacf} are sufficient but not
    necessary for having \( S <_B S^{\cf} \). To see this, consider a Borel
    quasi-order \( S \) whose quotient $S/E_S$ is a copy of \( \omega^* \) (the
    ordinal \( \omega \) equipped with the reverse order) together with two
    incomparable elements above it.
\item There are Borel  quasi-orders \( S \) such that \( S \sim_B S^{\cf}
    \): for example, let \( S \) be such that its quotient is isomorphic to
    a reverse well-ordering.
\end{enumerate}
\end{remark}

One may wonder whether the weaker conditions of Lemma~\ref{lemmacf} may be
used instead of~\cite[Proposition 6]{Rosend2005} to recursively build a \(
\leq_B \)-increasing and cofinal \( \omega_1 \)-sequence of Borel
quasi-orders similar to \( (P_\alpha)_{\alpha < \omega_1} \). Towards this
end, we first need to check that the disjunction of the conditions (i) and
(ii) of Lemma~\ref{lemmacf} is preserved by the operator \( S \mapsto S^{\cf}
\) and by countable products. To simplify the presentation, a quasi-order
satisfying such a disjunction will be called \emph{suitable}.

\begin{lemma} \label{lem:suitable}
Let \( S \), \( (S_\beta)_{\beta < \lambda} \) (for some limit \( \lambda <
\omega_1 \)) be suitable Borel quasi-orders. Then both \( S^{\cf} \) and \(
\prod_{\beta < \lambda} S_\beta \) are suitable as well.
\end{lemma}

\begin{proof}
Assume first that \( S \) satisfies condition (i) of Lemma~\ref{lemmacf}, let
$x,y \in X$ be $S$-incomparable and such that $A=\{ [z] \in X/E_S\mid z
\mathrel{S} x \land z \mathrel{S} y\}$ is well-founded. Without loss of
generality, we may assume that $A$ is actually linearly ordered: if not,
simply replace \( x,y \) with \( x',y' \in X \) such that
\begin{itemize}
\item \( [x'],[y'] \in A \);
\item \( x',y' \in X \) are \( S \)-incomparable;
\item the pair \( ( x',y' ) \) is \( S \times S \)-minimal among pairs with
    the above two properties (such a minimal pair exists because \( A \) is
    well-founded).
\end{itemize}
So $A$ is a well-order of order type some ordinal $\alpha$ and $x^{\infty
},y^{\infty }$ are $S^{ \cf }$-incomparable. Since $S$ is Borel, by the
boundedness theorem for analytic well founded relations (\cite[Theorem
31.1]{Kechris1995}), it follows that $\alpha <\omega_1$. Let $B=\{ [\vec z ]
\in {}^{\omega }X/E_{S^{ \cf }}\mid \vec z \mathrel{S^{ \cf }} x^{\infty
}\land \vec z \mathrel{S^{ \cf }} y^{\infty }\} $, so that, in particular,
for all \( [\vec z] \in B \) the \( E_S \)-classes of all coordinates of \(
\vec{z} \) belong to \( A \): we claim that $B$ is a well-order too, so that
\( S^{\cf} \) satisfies condition (i) of Lemma~\ref{lemmacf}. Indeed, if $[
\vec z ] \in B$, there are two possibilities: either there is a component in
$ \vec z $ that $S$-dominates all other components, in which case $ \vec z $
is $E_{S^{ \cf }}$-equivalent to the constant sequence with values that
coordinate; or there is no $S$-biggest component of $ \vec z $, so that $
\vec z $ is $S^{ \cf }$-above all constant sequences built using its
components and is $S$-below any sequence mentioning at least one component
$S$-bigger than all components of $ \vec z $. From these observations it
easily follows that \( B \) is a well-order, and in fact the order type of
$B$ can be obtained from $\alpha $ by adding a point $p_{\lambda }$, for any
limit ordinal $\lambda\leq\alpha $, on top of the subset
$\lambda\subseteq\alpha $ (this order type is consequently either $\alpha $
or $\alpha +1$).

Suppose now that \( S \) satisfies condition (ii) of Lemma~\ref{lemmacf}, and
let $A \subseteq X/E_S$ be the set satisfying the condition. If $A$ contains
incomparable elements, then actually condition (i) holds for $S$ and thus for
$S^{ \cf }$. Otherwise $A$ is a well-order. Repeating the argument of case
(i), $\{ [ \vec z ]\in {}^{\omega }X/E_{S^{ \cf }}\mid\forall n\in\omega\
[z_n]\in A\} $ witnesses that $S^{ \cf }$ fulfills condition (ii).

Let us now consider the case of countable products. Let $S_\beta$, \( \beta <
\lambda \), be suitable quasi-orders. If for some \( \gamma < \lambda \)
there are \( S_\gamma \)-incomparable elements $a_\gamma$ and $b_\gamma$
without a common \(S_\gamma \)-predecessor, then for every \( \alpha <
\lambda \) distinct from \( \gamma \) we fix an arbitrary element \( a_\alpha
= b_\alpha \in X_\alpha = \mathrm{dom}(S_\alpha) \) and we consider the
sequences \( \vec a = (a_\beta)_{\beta< \lambda} \) and \( \vec b  =
(b_\beta)_{\beta < \lambda} \): clearly \( \vec{a} \) and \( \vec{b} \) are
\( \prod_{\beta < \lambda}  S_\beta \)-incomparable and they do not have a
common \( \prod_{\beta < \lambda} S_\beta \)-predecessor, so that \(
\prod_{\beta < \lambda} S_\beta \) satisfies condition (i) of
Lemma~\ref{lemmacf}. Otherwise, for all \( \beta < \lambda \) there is an
$S_\beta$-minimal element $a_\beta$ with some \( b_\beta \) strictly \(
S_\beta \)-above \( a_\beta \). Then the sequences $(a_0,b_1,a_2,a_3, \dotsc
)$ and $(b_0,a_1,a_2,a_3, \dotsc )$ are \( \prod_{\beta < \lambda} S_\beta
\)-incomparable, and the set of their common predecessors consists just of
the equivalence class of $(a_0,a_1,a_2, \dotsc )$. Therefore \( \prod_{\beta
< \lambda} S_\beta \) satisfies condition (i) of Lemma~\ref{lemmacf} again.
\end{proof}

This allows us to show that in constructing Rosendal's cofinal sequence of
Borel quasi-orders, we could have started with any nontrivial Borel
quasi-order (instead of a Borel quasi-order satisfying the stronger
assumption of~\cite[Proposition 6]{Rosend2005}).

\begin{proposition}
Let \( S \) be a Borel quasi-order, and recursively define the sequence \(
(S_\alpha)_{\alpha < \omega_1} \) by setting \( S_0 = S \), \( S_{\alpha+1} =
S^{\cf}_\alpha \), and \( S_\lambda = \prod_{\beta < \lambda} S_\beta \) for
\( \lambda < \omega_1 \) limit. Then
\begin{enumerate}[(1)]
\item if \( E_S \) has at least two equivalence classes, then the sequence
    \( (S_\alpha)_{ \alpha < \omega_1} \) is cofinal among the Borel
    quasi-orders;
\item if moreover \( S \) is suitable, then \( (S_\alpha)_{ \alpha <
    \omega_1} \) is also strictly \( \leq_B \)-increasing.
\end{enumerate}
\end{proposition}

\begin{proof}
To prove (1), by Theorem~\ref{th:rosendalsequence} it clearly suffices to
show that \( P_0  \leq_B S_{\omega+\omega} \), that is that there is an
infinite family \( (\vec{z}^{\, i})_{i \in \omega} \) of pairwise \(
S_{\omega+\omega} \)-incomparable elements.

With an argument as the one used in the proof of the second part of Lemma~\ref{lem:suitable}, we get that \(
S_\omega \) contains at least two incomparable elements \( a_\omega, b_\omega
\). Moreover, since if \( x,y \) are incomparable with respect to a
quasi-order \( R \) then \( x^\infty, y^\infty \) are \( R^{\cf}
\)-incomparable, we also get that for each \( n \in \omega \) there are \(
S_{\omega+n} \)-incomparable elements \( a_{\omega+n}\) and \( b_{\omega+n}
\). Finally, for \( n < \omega \) we pick an arbitrary element \( a_n \in X_n
= \mathrm{dom}(S_n) \). For \( i \in \omega \) and \( \beta < \omega+\omega
\) set \( z^i_\beta = a_\beta \) if \( \beta \neq \omega+i \) and \(
z^i_{\omega +i} = b_{\omega +i} \). Then the sequence obtained setting \(
\vec{z}^{\, i} = (z^i_\beta)_{\beta < \omega+ \omega} \) is as required.

Part (2) follows from Lemmas~\ref{lemmacf} and~\ref{lem:suitable}.
\end{proof}

\subsection{The jump operator $S \mapsto S^{\inj}$}\label{subsect:inj}
Let us now consider the following variant of Rosendal's jump operator.

\begin{defin}
Let \( (X,S) \) be a quasi-order. We denote by \( (X,S)^{\inj} \) (or just by
\( S^{\inj} \), if the space \( X \) is clear from the context) the
quasi-order on \( \pre{\omega}{X} \) defined by
\[
( x_n )_{ n \in \omega } \mathrel{S^{\inj}} ( y_n )_{ n \in \omega}  \iff
\exists f \colon \omega \to \omega \text{ injective such that } \forall n (x_n \mathrel{S} y_{f(n)}).
 \]
\end{defin}

Similarly to what happened for Rosendal's jump \( S \mapsto S^{\cf} \), also
the operator \( S \mapsto S^{\inj} \) has already been implicitly considered
in the literature. For example, in~\cite[\S1.2.2]{Friedman1989} H.~Friedman
and Stanley used a jump operator for equivalence relations \( E \mapsto
E^\omega \) such that \( E^\omega =E_{E^{\inj}} \), as the following lemma
shows.

\begin{lemma} \label{ebij}
If $E$ is an equivalence relation on $X$, then
\[
( x_n )_{ n \in \omega } \mathrel{E_{E^{\inj}}} ( y_n )_{ n \in \omega}  \iff
\exists f \colon \omega \to \omega \text{ bijective } \forall n (x_n \mathrel{E} y_{f(n)}).
\]
\end{lemma}

\begin{proof}
The implication from right to left is immediate. Conversely, suppose there
are injections $\omega\to\omega $ witnessing $( x_n )_{ n \in \omega }
\mathrel{E^{\inj}} ( y_n )_{ n \in \omega}$ and $( y_n )_{ n \in \omega }
\mathrel{E^{\inj}} (x_n )_{ n \in \omega}$. This entails that for each
$E$-equivalence class $C$,
\[
|\{ n\in\omega\mid x_n\in C\} |=|\{ n\in\omega\mid y_n\in C\} |
\]
implying the existence of $f$.
\end{proof}

The following result (\cite[\S1.2.2]{Friedman1989}), reformulated with our
terminology using the observation above, will be used in the sequel.

\begin{proposition}\label{friedmaninj}
If $E$ is a Borel equivalence relation with at least two classes, then
$E_{E^{ \inj }}$ is Borel and $E<_BE_{E^{ \inj }}$.
\end{proposition}

Coming back to our operator for quasi-orders \( S \mapsto S^\inj \), it is
immediate to check that \( S \leq_B S^{\cf} \leq_B S^{\inj} \), that  \( S \)
is an analytic quasi-order if and only if so is \( S^{\inj} \) (but the
analogous statement with analytic replaced by Borel is not true, see
Propositions~\ref{prop:Q} and~\ref{prop:P(omega)}), and that \( S \mapsto
S^{\inj} \) is monotone with respect to \( \leq_B \), i.e.\ it is really a
jump operator. From these observations and Lemma~\ref{lemmacf} it immediately
follows that in many cases the new jump operator strictly increases the
complexity of the quasi-order which it is applied to.

\begin{corollary}\label{inj_jumps}
If \( S \) is a suitable Borel quasi-order (i.e.\ \( S \) satisfies the
hypothesis of Lemma~\ref{lemmacf}), then \( S <_B S^{\inj} \).
\end{corollary}

We now start a detailed analysis of the new jump operator \( S \mapsto
S^{\inj} \), and compare it to the one introduced by Rosendal. In particular,
we show the following.
\begin{enumerate}[(A)]
\item For extremely simple quasi-orders \( S \), like equality on an
    arbitrary standard Borel space, we get \( S^{\cf} \sim_B S^{\inj} \)
    (Proposition~\ref{prop:simple}).
\item If a Borel quasi-order \( S \) is combinatorially simple, like an
    equivalence relation or a wqo, then \( S^{\inj} \) remains Borel
    (Corollaries~\ref{cor:equivalencerelation} and~\ref{cor:injonwqo}).
\item There are not too complicated Borel quasi-orders \( S \) (e.g.\ any
    linear order on $\omega$ which is isomorphic to \( \QQ \), or the
    inclusion relation \( \subseteq \) on \( \pow(\omega) \)) such that \(
    S^{\inj} \) is analytic non-Borel. In fact, we can have both upper and
    lower \( S^{\inj} \)-cones which are \( \boldsymbol{\Sigma}^1_1
    \)-complete, and also the associated equivalence relation \(
    E_{S^{\inj}} \) can be analytic non-Borel (Propositions~\ref{prop:Q}
    and~\ref{prop:P(omega)}).
\item However, if \( S \) is Borel then all \( E_{S^{\inj}} \)-equivalence
    classes are Borel (Theorem~\ref{th:S^inj not complete}). In particular,
    \( S^{\inj} \) is not complete for analytic quasi-orders.
\item Moreover, there are also Borel quasi-orders \( S \) such that \(
    S^{\inj} \) is analytic non-Borel but their associated equivalence
    relation \( E_{S^{\inj}} \) still remains Borel
    (Corollary~\ref{cor:restrictedP(omega)}).
\item There exist Borel quasi-orders \( S \) such that \( S^{\cf} <_B
    S^{\inj} \) but \( E_{S^{\cf}} \sim_B E_{S^{\inj}} \)
    (Proposition~\ref{prop:associatedeqrel}).
\end{enumerate}

Most of these results will be used to analyze the isometric embeddability
relation between ultrametric Polish spaces with well-ordered set of
distances. However, some of them are also of independent interest. For
instance, Corollary~\ref{cor:restrictedP(omega)} provides simple
combinatorial examples of an analytic quasi-order \( S \) such that its
associated equivalence relation \( E_S \) is Borel without \( S \) being
Borel itself. Proposition~\ref{prop:P(omega)} and Theorem~\ref{th:S^inj not
complete} also allow us to construct new examples of analytic quasi-orders
\( S \) which are not Borel but still have the property that \( S <_B S^{\cf}
\) (Corollary~\ref{cor:newexamplesofjump}): these examples are of a different
type with respect to those considered in~\cite[Section 4]{Camerlo2007}, as
they do not satisfy the hypothesis of~\cite[Corollary 4.3]{Camerlo2007}.

\subsection{Borel quasi-orders $S$ such that $S^{\inj}$ is Borel}
\begin{lemma} \label{prop:equivalencerelation}
If $E$ is a Borel equivalence relation then $E^{\inj} \leq_B (E \times
(\omega, {=}))^{\cf}$.
\end{lemma}
\begin{proof}
Suppose $E$ is defined on the space $X$. Given $\vec{x} = (x_i)_{ i \in
\omega } \in {}^{\omega }X$, set $\varphi (\vec{x} )=(x_i, n_i)_{i \in
\omega} \in \pre{\omega}{(X \times \omega)}$, where $n_i$ is the number of
$j<i$ with $x_j \mathrel{E} x_i$. As $E$ is Borel the function $\varphi $ is
Borel.

Assume first that the function $f \colon \omega \to \omega $ witnesses
$\vec{x} = (x_i )_{i \in \omega} \mathrel{E^{\inj}} (y_j )_{j \in \omega} =
\vec{y}$. Let $\varphi (\vec{x} )=(x_i,n_i)_{i \in \omega}$ and $\varphi
(\vec{y})=(y_j,m_j)_{j \in \omega}$. Since $f$ is injective, given $l \in
\omega $ there are at least $n_l+1$ indices $j$ such that $y_j \mathrel{E}
y_{f(l)}$, for exactly one of which $m_j=n_l$. So $(x_l,n_l)$ and \(
(y_j,m_j) \) are \((E \times {=}) \)-related for this $j$, which implies that
$\varphi (\vec{x} ) \mathrel{(E \times {=})^{\cf}} \varphi (\vec{y} )$.

Conversely, suppose for some \( \vec{x} = (x_i)_{i \in \omega} , \vec{y} =
(y_j)_{j \in \omega} \in \pre{\omega}{X} \) we have that $\varphi (\vec{x}
)=(x_i,n_i)_{i \in \omega}$ and $\varphi (\vec{y} ) = (y_j,m_j)_{j \in
\omega}$ are \( (E \times {=})^{\cf} \)-related, so that in particular
\[
\forall l \in \omega \, \exists k_l \in \omega \, (x_l \mathrel{E} y_{k_l} \land n_l = m_{k_l}).
\]
Then the map $f \colon \omega \to \omega \colon l \mapsto k_l$ is injective
and witnesses $\vec{x} \mathrel{E}^{\inj} \vec{y}$.
\end{proof}

\begin{corollary}\label{cor:equivalencerelation}
If \( E \) is a Borel equivalence relation, then \( E^{\inj} \) is Borel as
well.
\end{corollary}

Another consequence of Lemma~\ref{prop:equivalencerelation} is that there are
some simple cases in which $S^{ \cf }\sim_BS^{ \inj }$.

\begin{proposition} \label{prop:simple}
\begin{enumerate}[(1)]
\item \( (\omega,=)^{\inj} \sim_B (\omega,=)^{\cf} \sim_B (\pow(\omega),
    \subseteq) \);
\item \( (\RR, = )^{\inj} \sim_B (\RR, = )^{\cf} \).
\end{enumerate}
\end{proposition}

\begin{proof}
All Borel reducibilities are trivial except perhaps for those of the form \(
E^{\inj} \leq_B E^{\cf} \), where \( E \) is one of \( (\omega,= ) \) or \(
(\RR,=) \): these can be proved using Lemma~\ref{prop:equivalencerelation}
together with the fact that in this specific cases \( E \times (\omega,=)
\sim_B E \).
\end{proof}

More generally, by Lemma~\ref{prop:equivalencerelation} we have that \(
E^{\inj} \sim_B E^{\cf} \) whenever \( E \) is a Borel equivalence relation
satisfying \( {E \times ( \omega,=)} \leq_B E \). This includes e.g.\ the
well-known equivalence relations \( E_0 \), \( E_1 \), and so on.

Corollary~\ref{cor:equivalencerelation} provides a first example illustrating
the fact that when applied to combinatorially simple quasi-orders the
operator \( S \mapsto S^{\inj} \) preserves Borelness. Another example of
this kind is when \( S \) is a wqo. To see this, we first need to prove some
easy facts concerning definable wqo's.

\begin{lemma} \label{coranalyticwqo}
If \( S \) is a well-founded analytic quasi-order without uncountable
antichains, then $( \RR ,=) \nleq_B E_S$. Hence \( E_S \) has at most \(
\aleph_1 \)-many classes, and if  \( S \) is Borel it has countably many
classes.
\end{lemma}

\begin{proof}
Towards a contradiction, let \( f \colon \RR \to X \) be a Borel function
such that either \( f(x) \not \mathrel{S} f(y) \) or \( f(y) \not \mathrel{S}
f(x) \) for all distinct \( x, y \in \RR \). Let \( \preceq \) be defined on \(
\RR \) by
\[
x \preceq y \iff f(x) \mathrel{S} f(y).
\]
Then $\preceq$ is an analytic well-founded partial order. Using the
boundedness theorem for analytic well founded relations (\cite[Theorem
31.1]{Kechris1995}), $\preceq $ has countable rank. So there are uncountably
many reals having the same rank with respect to $\preceq $. But this is
impossible as it would yield an uncountable antichain for $\preceq $ and then
for $S$.

The additional part follows from Burgess' trichotomy theorem for analytic
equivalence relations (\cite[Corollary 2]{Burgess1979}) and Silver's
dichotomy theorem for coanalytic equivalence relations (\cite{Silver1980}).
\end{proof}

It is not hard to check that if \( S \) is an analytic quasi-order on a
standard Borel space \( X \) such that \( E_S \) has at most countably many
classes, then \( S \) is actually Borel. This fact combined with
Lemma~\ref{coranalyticwqo} gives the following result.

\begin{corollary} \label{corwqo=countable}
Let \( S \) be a well-founded analytic quasi-order on \( X \) without
uncountable antichains. Then \( S \) is Borel if and only if \(E_S \) has
countably many classes.
\end{corollary}

In particular, both Lemma~\ref{coranalyticwqo} and
Corollary~\ref{corwqo=countable} apply to analytic wqo's.

The proof of the next proposition is quite similar to that of~\cite[Lemma
10]{NashWil1965}.

\begin{proposition} \label{prop:wqojump}
If \( S \) is a wqo on \( \omega \), then \( S^{\inj} \) is Borel.
\end{proposition}

\begin{proof}
For every \( \vec{y} = (y_n)_{n \in \omega} \in \pre{\omega}{\omega} \) and
\( n \in \omega \), let \( \vec{y}^{\; \uparrow n} \) denote the \( S \)-upper
cone determined by \( n \) in \( \vec{y} \), i.e.\ \( \vec{y}^{\; \uparrow  n} =
\{ m \in \omega \mid y_n \mathrel{S} y_m \} \). Then the set \( F_{\vec{y}} =
\{ n \in \omega \mid \vec{y}^{\; \uparrow  n} \text{ is finite} \} \) is finite
by~\cite[Lemma 3]{NashWil1965}. Let \( K_{\vec{y}} = \omega \setminus
F_{\vec{y}} \).

\begin{claim} \label{claiminfinite}
For every \( n \in K_{\vec{y}} \) the set \(  \vec{y}^{\; \uparrow
n}_{K_{\vec{y}}} = \{ m \in K_{\vec{y}} \mid y_n \mathrel{S} y_m \} \) is
infinite.
\end{claim}
\begin{proof}
Let \( n \in \omega \) be such that \( \vec{y}^{\; \uparrow n}_{K_{\vec{y}}} \)
is finite. Then \( \vec{y}^{\; \uparrow  n} \subseteq \vec{y}^{\; \uparrow
n}_{K_{\vec{y}}} \cup F_{\vec{y}}\) is finite as well, i.e.\ \( n \in
F_{\vec{y}} \).
\end{proof}

Given \( \vec{x}, \vec{y} \in \pre{\omega}{\omega} \), let \(
K_{\vec{x},\vec{y}} = \{ n \in \omega \mid \exists m \in K_{\vec{y}} \, (x_n
\mathrel{S} y_m) \} \).

\begin{claim} \label{claimK}
There is an injective function \( f \colon K_{\vec{x},\vec{y}} \to
K_{\vec{y}} \) such that \( x_n \mathrel{S} y_{f(n)} \) for every \( n \in
K_{\vec{x},\vec{y}} \).
\end{claim}
\begin{proof}
We define the function \( f \) by recursion on \( n \in K_{\vec{x},\vec{y}}
\). Notice that it is enough to verify that for every \( n \in \omega \) the
restriction \( f \restriction (n \cap K_{\vec{x},\vec{y}}) \) of \( f \) is
injective and such that \( x_k \mathrel{S} y_{f(k)} \) for every \( k \in n
\cap K_{\vec{x},\vec{y}} \). Suppose that \( n \in K_{\vec{x},\vec{y}} \) and
that \( f \restriction (n \cap K_{\vec{x},\vec{y}}) \) is as above. Let \( m
\in K_{\vec{y}} \) be such that \( x_n \mathrel{S} y_m \) (which exists
because \( n \in K_{\vec{x},\vec{y}} \)). By Claim~\ref{claiminfinite} there
are infinitely many \( m' \in K_{\vec{y}} \) such that \( y_m \mathrel{S}
y_{m'} \): letting \( f(n) \) be the least of these \( m' \) such that \( m'
\notin f(n \cap K_{\vec{x},\vec{y}}) \), we get \( x_n \mathrel{S} y_{f(n)}
\) (by transitivity of \( S \)), and that \( f \restriction (n+1 \cap
K_{\vec{x},\vec{y}}) \) is as required.
\end{proof}

The next claim shows that for having \( \vec{x} \mathrel{S}^{\inj} \vec{y} \)
it is necessary and sufficient that there is a partial witness of this fact
defined on \( \omega \setminus K_{\vec{x},\vec{y}} \).

\begin{claim} \label{claimnonK}
\( \vec{x} \mathrel{S}^{\inj} \vec{y} \) if and only if there is an injective
function \(g \colon \omega \setminus K_{\vec{x},\vec{y}} \to F_{\vec{y}} \)
such that \( x_n \mathrel{S} y_{g(n)} \) for every \( n \in \omega \setminus
K_{\vec{x},\vec{y}} \).
\end{claim}
 \begin{proof}
For the forward direction, let \( f \) be a witness of \( \vec{x}
\mathrel{S}^{\inj} \vec{y}\), and let \( g  = f \restriction (\omega
\setminus K_{\vec{x},\vec{y}}) \): it suffices to show that \( g(n) = f(n)
\in F_{\vec{y}} \) for every \(n \in \omega \setminus K_{\vec{x},\vec{y}} \).
Since \( x_n \mathrel{S} y_{f(n)} \), this follows from the definition of \(
K_{\vec{x},\vec{y}} \).

Conversely, let \( g \) be as in the statement of the claim, and let \(f\) be
the map obtained in Claim~\ref{claimK}: then \( f \cup g \) witnesses \(
\vec{x} \mathrel{S}^{\inj} \vec{y} \). To see this, it suffices to show that
if \( n \in K_{\vec{x},\vec{y}} \) and \( m \in \omega \setminus
K_{\vec{x},\vec{y}} \) then \( f(n) \neq g(m) \), and this follows from the
fact that the range of \( f \) is contained in \( K_{\vec{y}} = \omega
\setminus F_{\vec{y}} \), while the range of \( g \) is contained in \(
F_{\vec{y}} \).
\end{proof}

By definitions, Claim~\ref{claimnonK}, and the fact that \(
F_{\vec{y}} \) is finite, $\vec{x} \mathrel{S}^{\inj} \vec{y}$ is equivalent
to
\[
\omega \setminus K_{\vec{x},\vec{y}} \text{ is finite} \land
\exists g \text{ injective} \colon \omega \setminus K_{\vec{x},\vec{y}} \to F_{\vec{y}}\,
\forall n \in \omega \setminus K_{\vec{x},\vec{y}} \, ( x_n \mathrel{S} y_{g(n)}).
\]
Since the maps \( \vec{y} \mapsto F_{\vec{y}} \) and \( (\vec{x},\vec{y})
\mapsto K_{\vec{x},\vec{y}} \) are Borel, this equivalence shows that \(
S^{\inj} \) is Borel.
\end{proof}

\begin{remark} \label{rmk:wellordersonomega}
\begin{enumerate}[(1)]
\item If \( S \) is further assumed to be a bqo on $\omega$, then the fact
    that \( S^{\inj} \) is Borel follows also from
    Corollary~\ref{corwqo=countable} and a result of Laver (\cite[Theorem
    4.11]{Laver1971}). In fact in this case \( S^{\inj} \) is also a bqo by
    Nash-Williams' theorem on transfinite sequences (\cite{NW67}).
\item If moreover $S$ is linear (i.e.\ a well-order on \(\omega\)) then:
\begin{enumerate}[(a)]
\item \( S^{\inj} <_B (\pow(\omega), \subseteq) \) because $f \colon
    \pre\omega \omega \to \pow (\omega)$ defined by $f(\vec{x}) =
    \{\langle n,k \rangle \mid |\{i \mid n \mathrel{S} x_i\}| \geq k \}$, where $(n,k) \mapsto \langle n,k \rangle$ is a pairing function,
    witnesses \( S^{\inj} \leq_B (\pow(\omega), \subseteq) \), while \(
    (\pow(\omega), \subseteq) \nleq_B S^{\inj}\) follows from the
    previous point because $S$ and thus \( S^{\inj} \) are bqo's;
\item if \( S' \) is another linear wqo on \( \omega \) with order type
    strictly larger than that of \( S \), then \( S^\inj <_B (S')^\inj \)
    (this can be proved using the fact that every wqo contains a chain of
    maximal order type \cite{Wolk}).
\end{enumerate}
In particular \( (\omega,\leq)^{\inj} <_B (\omega+1,\leq)^{\inj} <_B \dotsc
<_B (\alpha,\leq)^{\inj} <_B \dotsc <_B (\pow(\omega), \subseteq) \) is a
strictly increasing chain of Borel quasi-orders of length $\omega_1+1$.
    \end{enumerate}
\end{remark}

\begin{corollary}\label{cor:injonwqo}
Let \( S \) be a Borel quasi-order. If \( S \) is a wqo, then \( S^{\inj}\)
is Borel as well.
\end{corollary}

\begin{proof}
Let \( S \) be a Borel wqo on \( X \). By Corollary~\ref{corwqo=countable},
\( E_S\) has at most countably many classes. Let \( (A_i)_{i < I} \) (for
some \( I \leq \omega \)) be an enumeration without repetitions of the \( E_S
\)-equivalence classes, and consider the wqo \( \tilde{S} \) on \( \omega \)
defined by letting $n \mathrel{\tilde{S}} m $ if and only if one of the
following conditions holds:
\begin{itemize}
\item $n,m < I \land \forall x \in A_n \forall y \in A_m \, (x \mathrel{S}
    y)$; or
\item $n\in\omega\setminus I\land m<I\land\forall x\in A_0\forall y\in
    A_m\, (x \mathrel{S} y)$; or
\item $n<I\land m\in\omega\setminus I\land\forall x\in A_n\forall y\in
    A_0\, (x \mathrel{S} y)$; or
\item $n,m \in (\omega \setminus I) \cup \{ 0 \} $.
\end{itemize}
Since clearly \( S \sim_B \tilde{S} \), it suffices to show that \(
\tilde{S}^{\inj}\) is Borel: but this follows from
Proposition~\ref{prop:wqojump}, hence we are done.
\end{proof}

\subsection{Borel quasi-orders $S$ such that $S^{\inj}$ is proper
analytic}\label{subsec:injnotB}

We now start considering some examples of Borel quasi-orders \( (X,S) \) such
that \( (X,S)^{\inj}\) is analytic non-Borel: as customary in the subject,
this is usually obtained by either directly showing that the binary relation
\( S^{\inj} \) is a \( \boldsymbol{\Sigma}^1_1 \)-complete subset of the
square \( \pre{\omega}{X} \times \pre{\omega}{X} \), or else by showing that
the \emph{upper cone generated by \( \vec{p} \)}
\[
C^{S^{\inj}}(\vec{p}) =  \{ \vec{x} \in \pre{\omega}{X} \mid \vec{p} \mathrel{S}^{\inj} \vec{x} \}
\]
or the \emph{lower cone generated by \( \vec{q} \)}
\[
C_{S^{\inj}}(\vec{q}) = \{ \vec{x} \in \pre{\omega}{X} \mid \vec{x} \mathrel{S}^{\inj} \vec{q} \}
 \]
are \( \boldsymbol{\Sigma}^1_1 \)-complete subsets of \( \pre{\omega}{X}\)
(for suitable \( \vec{p},\vec{q} \in \pre{\omega}{X} \)).

Our first example shows that the jump operator \( S \mapsto S^{\inj}\) can
produce analytic non-Borel quasi-orders even when applied to countable linear
orders (this should be contrasted with the case of countable well-orders, see
Remark~\ref{rmk:wellordersonomega}(2)).

\begin{proposition} \label{prop:Q}
Consider the linear order $(\QQ, \leq)$. Then $\leq^{\inj}$, considered as a
subset of $\pre{\omega}\QQ \times \pre{\omega}\QQ$ (where the latter is
endowed with the product of the discrete topology on \( \QQ \)), is
$\boldsymbol{\Sigma}^1_1$-complete and hence non-Borel. In fact the
equivalence relation $E_{(\QQ, \leq)^{\inj}}$ is a \( \boldsymbol{\Sigma}^1_1
\)-complete subset of \( \pre{\omega}{\QQ} \times \pre{\omega}{\QQ} \).

Therefore, for every non-scattered countable linear order \( L \) the
quasi-order \( L^{\inj}\) is analytic non-Borel.
\end{proposition}

\begin{proof}
Since the map \( \QQ \to \QQ \colon q \mapsto - q \) witnesses \( (\QQ, \leq)
\sim_B (\QQ,\geq)\), it is clearly enough to prove the result for \( (\QQ,
\geq)^{\inj} \). To this end we define a continuous reduction of
$\mathrm{NWO}$ to ${\geq^{\inj}}$, where $\mathrm{NWO}$ is the set of
non-well-orders (viewed as a subset of the Polish space of countable linear
orders, see~\cite[Section 27.C]{Kechris1995}), a well-known
$\boldsymbol{\Sigma}^1_1$-complete set.

We associate to every linear order $L$ a pair of sequences $(\vec{\alpha}_L,
\vec{\beta}_L)$ of rationals in the interval $[0,1]$. First we map
continuously and in an order preserving way $L$ into $\QQ \cap [0,1)$, so
that we can identify $L$ with its image. Let $\vec{\alpha}_L$ and
$\vec{\beta}_L$ be injective enumerations of $L \cup \{1\}$ and $L$,
respectively. It is straightforward to check that $L \in \mathrm{NWO}$ if and
only if $\vec{\alpha}_L \geq^{\inj} \vec{\beta}_L$.

Since $\vec{\beta}_L \geq^{\inj} \vec{\alpha}_L$ for every $L$, this argument
shows in fact that $E_{(\QQ, \geq)^{\inj}}$ is a $ \boldsymbol
\Sigma^1_1$-complete subset of \( \pre{\omega}{\QQ} \times \pre{\omega}{\QQ}
\).
\end{proof}

In the second example, we consider the jump of the quasi-order \(
(\pow(\omega), \subseteq)\), which by Proposition~\ref{prop:simple}(1) is
Borel-bireducible with \( (\omega,=)^{\inj} \).

\begin{proposition}\label{prop:P(omega)}
\begin{enumerate}[(1)]
\item The upper cone in $(\pow(\omega), {\subseteq})^{\inj}$ generated by
    \(\vec{p} = (A_n)_{n \in \omega} \in \pre{\omega}{(\pow(\omega))}\)
    with \( A_n = \{ n \} \) is \( \boldsymbol{\Sigma}^1_1 \)-complete. In
    particular, $\subseteq^{\inj}$ is a
    $\boldsymbol{\Sigma}^1_1$-complete (hence non-Borel) subset of
    $\pre{\omega}{(\pow(\omega))} \times \pre{\omega}{(\pow(\omega))}$.
\item The lower cone in $(\pow(\omega), {\subseteq})^{\inj}$ generated by
    \(\vec{q} = (B_n)_{n \in \omega} \in \pre{\omega}{(\pow(\omega))} \)
    with \( B_n = \omega \setminus \{ n \} \) is \( \boldsymbol{\Sigma}^1_1
    \)-complete.
\item The equivalence relation \( E_{ (\pow(\omega),\subseteq)^{\inj}} \)
    is a \(\boldsymbol{\Sigma}^1_1 \)-complete subset of
    $\pre{\omega}{(\pow(\omega))} \times \pre{\omega}{(\pow(\omega))}$.
\end{enumerate}
\end{proposition}

\begin{proof}
Fix a bijection \( \omega \to \pre{<\omega}{\omega} \colon n \mapsto s_n \)
such that \( s_0 = \emptyset \), and for every \( n \in \omega \) such that
\( n>0 \) let \( n^\star \) be the unique natural number such that \(
s_{n^\star} = s_n \restriction (\leng(s_n)-1) \).

(1) We show that there is a continuous function mapping each nonempty tree \(
T \subseteq \pre{< \omega}{\omega}\) to some \( \vec{q}_T \in
\pre{\omega}{(\pow(\omega))} \)  which reduces $\mathrm{IF}$ to the upper
cone generated by $ \vec p $, where \( \mathrm{IF} \subseteq \pow(\pre{<
\omega}{\omega}) \) is the set of trees with at least one infinite branch (a
well-known \( \boldsymbol{\Sigma}^1_1 \)-complete set, see~\cite[Section
27.A]{Kechris1995}). Given a nonempty tree \( T \subseteq
\pre{<\omega}{\omega} \) let \( \vec{q}_T = (B_{T,n})_{n \in \omega} \in
\pre{\omega}{(\pow(\omega))} \) be defined by
\[
B_{T,n} =
\begin{cases}
\emptyset & \text{if } n = 0 \\
\{ n \} & \text{if } s_n \notin T \\
\{ n, n^\star \} & \text{if } n \neq 0 \text{ and } s_n \in T.
\end{cases}
 \]
Notice that if there is \( m \neq n \) such that \( m \in B_{T,n} \), then \(
\emptyset \neq s_n \in T \) and \( m = n^\star \). The map \( T \mapsto
\vec{q}_{T} \) is clearly continuous, and we claim that \( T \in \mathrm{IF}
\iff {\vec{p} \subseteq^{\inj} \vec{q}_T} \).

Let first \( T \in \mathrm{IF}\), and let \( (n_k)_{k \in \omega} \) be a
sequence of natural numbers such that \( s_{n_k} \in T \), \( \leng(s_{n_k})
= k \), and \( s_{n_k} \subseteq s_{n_{k'}} \) for \( k \leq k' \in \omega
\), so that, in particular, \( n_0 = 0 \). Define \( \varphi \colon  \omega
\to \omega \) by setting
\[
\varphi(n) =
\begin{cases}
n_{k+1} & \text{if } n = n_k \text{ for some } k \in \omega \\
n & \text{if } n \neq n_k \text{ for all } k \in \omega.
\end{cases}
 \]
It is then easy to check that \( \varphi \) witnesses \( \vec{p}
\subseteq^{\inj} \vec{q}_T\).

Conversely, let \( \varphi \colon \omega \to \omega \) be a witness of \(
\vec{p} \subseteq^{\inj} \vec{q}_T \), and recursively set \( n_0 = 0 \) and
\( n_{k+1} = \varphi(n_k) \). Using the injectivity of \( \varphi \) and the
fact that \( A_i \not\subseteq \emptyset =  B_{T,0} \) for every $i$, one can
easily check by induction that \( n_k \neq n_{k+1} \) and \( n_{k+1} \neq 0
\) for all \( k \in \omega \). These facts, together with \( \{ n_k \} =
A_{n_k} \subseteq B_{T, \varphi(n_k)} = B_{T, n_{k+1}} \), imply that \(
s_{n_{k+1}} \in T \) and  \( n_k = n_{k+1}^\star \) (whence \( s_{n_k}
\subsetneq s_{n_{k+1}} \)) by the observation following the definition of \(
\vec{q}_T \). Thus \( (s_{n_k})_{k \in \omega } \) is an infinite
branch through \( T \) and \( T \in \mathrm{IF} \).

(2) The proof is similar to that of (1). Given a nonempty tree \( T \subseteq
\pre{<\omega}{\omega}\), let \( \vec{p}_T = (A_{T,n})_{n \in \omega} \in
\pre{\omega}{(\pow(\omega))} \) be defined by
\[
A_{T,n} =
\begin{cases}
\omega \setminus \{ n \} & \text{if } s_n \notin T \\
\omega \setminus \{ m \in \omega \mid s_m \in T \land m^\star =0\} & \text{if } n = 0 \\
\omega \setminus (\{ m \in \omega \mid s_m \in T \land m^\star = n \} \cup \{ n \}) & \text{if } s_n \in T \text{ and } n \neq 0.
\end{cases}
\]
It is straightforward to check that exactly
the same argument used in (1) shows \( T \in \mathrm{IF} \iff \vec{p}_T
\subseteq^{\inj} \vec{q} \).

(3) This follows again from a minor modification of the construction given in
part (1). Given \( n \in \omega\), let \( \pred(n) = \{ m \in \omega \mid s_m
\subseteq s_n \} \). To
a given infinite tree $T \subseteq \pre{<\omega}{\omega}$ associate the pair
$(\vec{p}_T, \vec{q}_T)$, where $\vec{p}_T \in \pre{\omega}{(\pow(\omega))}$
lists all sets in $\{ \pred(n) \mid s_n \in T \}$, and $\vec{q}_T \in
\pre{\omega}{(\pow(\omega))}$ is the same as $\vec{p}_T$ except for omitting
$\{ 0 \} $, that is \( \vec{q}_T \) lists the elements of \( \{ \pred(n) \mid
s_n \in T \land n \neq 0 \} \). Then $\vec{q}_T \subseteq^{\inj} \vec{p}_T$
for any infinite tree $T$, while arguing as in (1) one easily sees that
$T \in \mathrm{IF} \iff \vec{p}_T \subseteq^{\inj} \vec{q}_T $.
\end{proof}

\begin{remark} \label{rmk:restrictedP(omega)}
The proof of Proposition~\ref{prop:P(omega)}(1) actually shows that the upper
cone determined by \( \vec{p}\) remains \( \boldsymbol{\Sigma}^1_1
\)-complete even if we replace \( \subseteq^{\inj} \) with its restriction to
subsets of \( \omega \) of size at most \( 2 \). Therefore we further have
that for every \( 2 < N \leq \omega \), the analytic quasi-order \( (\pow_{<
N}(\omega), \subseteq)^{\inj} \) still contains \( \boldsymbol{\Sigma}^1_1
\)-complete upper cones.

Moreover, the proof of Proposition~\ref{prop:P(omega)}(3) shows that \(
E_{(\pow_{< \omega}(\omega), {\subseteq})^{\inj}} \) is already a \(
\boldsymbol{\Sigma}^1_1 \)-complete subset of \( \pre{\omega}{(\pow_{<
\omega}(\omega))} \times \pre{\omega}{(\pow_{< \omega}(\omega))} \).
\end{remark}

\subsection{The complexity of the equivalence relation associated to
$S^{\inj}$}\label{subsec:eqrelinj}

Throughout this subsection we fix a quasi-order $S$ on a set $X$.

\begin{defin}
Let $\vec{a} = (a_n)_{n \in \omega} \in \pre\omega X$. By transfinite
recursion we define for every $\alpha < \omega_1$ a set $\I\alpha a \subseteq
\omega$ as follows:
\begin{itemize}
  \item $\I0 a = \omega$;
  \item $\I{\alpha+1}a = \set{n \in \I\alpha a}{\exists^\infty m \in
      \I\alpha a\, (a_n \mathrel{S} a_m)}$;
  \item $\I\lambda a = \bigcap_{\alpha<\lambda} \I\alpha a$ when $\lambda$
      is limit.
\end{itemize}
\end{defin}

It is obvious that $\alpha < \beta$ implies $\I\alpha a \supseteq \I\beta a$,
and hence the sequence must stabilize at some countable ordinal.

\begin{defin}
For $\vec{a} \in \pre\omega X$ let $\rho (\vec{a})$ be the least $\alpha <
\omega_1$ such that $\I{\alpha+1}a = \I\alpha a$. We write $\II a$ in place
of $\I{\rho (\vec{a})}a$, so that \( \II a = \bigcap_{\alpha < \omega_1} \I
\alpha a \).
\end{defin}

The following lemma follows immediately from the definitions.

\begin{lemma}\label{Ibasic} Let $\vec{a} \in \pre\omega X$. Then
\begin{enumerate}[(1)]
\item for all $\alpha < \omega_1$ and $n \in \I\alpha a$, if $a_m
    \mathrel{S} a_n$ then $m \in \I\alpha a$;
\item if \( n \notin \II a \) and \( \alpha < \omega_1 \) is such that \( n
    \in \I \alpha a \setminus \I {\alpha+1} a \), then
\[
\set{m \in \omega}{a_n \mathrel{E_S} a_m} = \set{m \in \I \alpha a}{a_n \mathrel{E_S} a_m },
 \]
and this set is finite;
\item if $n \in \II a$ then $\exists^\infty m \in \II a\, (a_n \mathrel{S}
    a_m)$.
\end{enumerate}
\end{lemma}

In particular, $\II a$ is either empty or infinite, and each \( \I \alpha a
\) is invariant under \( E_S \), that is: if \( n,m \in \omega \) are such
that \( a_n \mathrel{E_S} a_m \), then \( n \in \I \alpha a \iff m \in \I
\alpha a \).

\begin{lemma}\label{IEinj}
Let $\vec{a}, \vec{b} \in \pre\omega X$ and suppose $f,g \colon  \omega \to
\omega$ are injective functions witnessing respectively $\vec{a}
\mathrel{S^{\inj}} \vec{b}$ and $\vec{b} \mathrel{S^{\inj}} \vec{a}$. Then
\begin{enumerate}[(1)]
\item for every $\alpha < \omega_1$ we have $\forall n (n \in \I\alpha a
    \iff f(n) \in \I\alpha b)$ and $\forall n (n \in \I\alpha b \iff g(n)
    \in \I\alpha a)$ (so that, in particular, \( \rho(\vec{a}) =
    \rho(\vec{b}) \));
\item $\forall n \in\omega\setminus \II a (a_n \mathrel{E_S} b_{f(n)})$ and
    $\forall n \in\omega\setminus \II b (b_n \mathrel{E_S} a_{g(n)})$.
\end{enumerate}
\end{lemma}
\begin{proof}
To prove (1) we first notice that the right to left direction of each
equivalence follows from the left to right implication of the other
equivalence. Indeed if $f(n) \in \I\alpha b$ then $g(f(n)) \in \I\alpha a$.
Since by definition of $S^{\inj}$ we have $a_n \mathrel{S} b_{f(n)}
\mathrel{S} a_{g(f(n))}$, Lemma \ref{Ibasic}(1) implies $n \in \I\alpha a$.
Similarly one shows that if $g(n) \in \I\alpha a$ then $n \in \I\alpha
b$.

The proof of the forward implication of both equivalences is by induction on
$\alpha$. If $\alpha=0$ the statement is obvious, and if $\alpha$ is a limit
ordinal it suffices to apply the induction hypothesis. Thus it remains to
derive the two implications for $\alpha+1$ assuming they hold for $\alpha$.

Fix $n \in \I{\alpha+1}a$. Since $n \in \I\alpha a$ the induction hypothesis
implies that $f(n) \in \I\alpha b$ and to show that $f(n) \in \I{\alpha+1}b$
we need to find infinitely many $m \in \I\alpha b$ such that $b_{f(n)}
\mathrel{S} b_m$.
We have $a_n \mathrel{S}
b_{f(n)}$. The proof splits into two cases.

If $b_{f(n)} \mathrel{S} a_n$ then for every $m \in \I\alpha a$ with $a_n
\mathrel{S} a_m$ we have $b_{f(n)} \mathrel{S} a_n \mathrel{S} a_m
\mathrel{S} b_{f(m)}$. Since there are infinitely many of these $m$ (because
$n \in \I{\alpha+1}a$) and for each of them $f(m) \in \I\alpha b$ (by induction
hypothesis), by injectivity of \( f \) we have reached our goal.

Now suppose $b_{f(n)} \not \mathrel{S} a_n$. Notice that $(f \circ g)^k(f(n))
\in \I\alpha b$ for every $k$ (here we are using the induction hypothesis for
both $f$ and $g$). Moreover, $b_{f(n)}
\mathrel{S} b_{(f \circ g)^k(f(n))}$ for every $k$. Thus it suffices to show
that the map $k \mapsto (f \circ g)^k (f(n))$ is injective. Since $f \circ g$
is injective, it is enough to show that $(f \circ g)^k(f(n)) \neq f(n)$ for
every $k>0$ and, by injectivity of $f$, this is equivalent to $(g \circ
f)^k(n) \neq n$ for every $k>0$. But if $k>0$ we have $b_{f(n)} \mathrel{S}
a_{(g \circ f) (n)} \mathrel{S} \dots \mathrel{S} a_{(g \circ f)^k(n)}$. Thus
$(g \circ f)^k(n) = n$ implies $b_{f(n)} \mathrel{S} a_n$ against our
hypothesis.

We have thus shown the first implication. The other implication is proved by
switching the roles of $f$ and $g$.

To prove (2) assume that $n \in \I\alpha a \setminus \I{\alpha +1} a $ and
$a_n \mathrel{E_S} b_{f(n)}$ fails, i.e.\ $b_{f(n)} \not \mathrel{S} a_n$. We
showed in the second case of the preceding proof that under this hypothesis
the map $k \mapsto (g \circ f)^k(n)$ is injective. Since by (1) each $(g
\circ f)^k(n) \in \I\alpha a$ and $a_n
\mathrel{S} a_{(g \circ f)^k(n)}$, this shows $n \in \I{\alpha+1}a$. Again, the
second statement is proved symmetrically.
\end{proof}

Lemma~\ref{IEinj}(1) immediately implies the following corollary.

\begin{corollary}\label{cEinj}
Let \( \vec{a}, \vec{b} \in \pre{\omega}{X} \) be such that \( \vec{a}
\mathrel{E_{S^{\inj}}} \vec{b} \), and let \( f \colon \omega \to \omega \)
witness \( \vec{a} \mathrel{S^{\inj}} \vec{b} \). Then
\begin{enumerate}[(1)]
\item \( f (\I \alpha a \setminus \I {\alpha+1} a ) \subseteq \I \alpha b
    \setminus \I {\alpha+1} b \) for every \( \alpha < \rho(\vec{a}) \), so
    that, in particular, \( f (\omega \setminus \II a) \subseteq \omega
    \setminus \II b \);
\item \( f(\II a) \subseteq \II b \).
\end{enumerate}
\end{corollary}
\begin{proof}
(2) follows from \( \II a = \bigcap_{\alpha < \omega_1} \I \alpha a \) and \(
\II b = \bigcap_{\alpha < \omega_1} \I \alpha b \).
\end{proof}

The following lemma provides a combinatorial characterization of the relation
\( E_{S^\inj} \).

\begin{lemma}\label{Einj:char}
For every $\vec{a}, \vec{b} \in \pre\omega X$ we have $\vec{a}
\mathrel{E_{S^{\inj}} }\vec{b}$ if and only if the following conditions hold:
\begin{enumerate}[\quad(i)]
  \item $\rho (\vec{a}) = \rho (\vec{b})$;
  \item $|\set{m \in \omega}{a_n \mathrel{E_S} a_m}| = |\set{m \in
      \omega}{a_n \mathrel{E_S} b_m}|$ for all \( n \in \omega \setminus
      \II a \);
  \item $|\set{m \in \omega}{b_n \mathrel{E_S} b_m}| = |\set{m \in
      \omega}{b_n \mathrel{E_S} a_m}|$ for all \( n \in \omega \setminus
      \II b \);
  \item $\forall n\in \II a \ \exists m\in \II b\, ( a_n\mathrel{S} b_m)$
      and $\forall n\in \II b \ \exists m\in \II a \, (b_n\mathrel{S}
      a_m)$.
\end{enumerate}
\end{lemma}

\begin{proof}
We first assume $\vec{a} \mathrel{E_{S^{\inj}} }\vec{b}$ and fix the
injective functions $f,g \colon \omega \to \omega$ witnessing respectively
$\vec{a} \mathrel{S^{\inj}} \vec{b}$ and $\vec{b} \mathrel{S^{\inj}}
\vec{a}$, so that we have the same notation of Lemma~\ref{IEinj} and
Corollary~\ref{cEinj}. Then (i) follows directly from Lemma~\ref{IEinj}(1).
Now assume \( n \in \omega \setminus \II a \). By Lemma~\ref{Ibasic}(2), we
have \( \set{m \in \omega}{a_m \mathrel{E_S} a_n} \subseteq \omega \setminus
\II a\), so that $f$ maps injectively $\set{m \in \omega}{a_n \mathrel{E_S}
a_m}$ into $\set{m \in \omega}{a_n \mathrel{E_S} b_m} \cap (\omega \setminus
\II b)$ by Lemma~\ref{IEinj}(2) and Corollary~\ref{cEinj}(1). Since this last
set is nonempty, by Lemma~\ref{Ibasic}(2) again we also have \( \set{m \in
\omega}{a_n \mathrel{E_S} b_m} \subseteq \omega \setminus \II b\), and by
Lemma~\ref{IEinj}(2) $g$ maps injectively $\set{m \in \omega}{a_n
\mathrel{E_S} b_m}$ into $\set{m \in \omega}{a_n \mathrel{E_S} a_m}$. This
shows (ii), and (iii) is obtained similarly. Finally, $f$ maps $\II a$ into
$\II b$ and $g$ maps $\II b$ into $\II a$ by Corollary~\ref{cEinj}(2), which
in particular implies (iv).

For the other direction assume (i)--(iv) hold. To define $f$ witnessing
$\vec{a} \mathrel{S^{\inj}} \vec{b}$, we start pasting together bijections
between each $\set{m \in \omega}{a_n \mathrel{E_S} a_m}$ and $\set{m \in
\omega}{a_n \mathrel{E_S} b_m}$ whenever $n \in \omega \setminus \II a$ is
such that $\forall m<n\, \neg (a_n \mathrel{E_S} a_m)$. Notice that for each
such \( n \), the set $\set{m \in \omega}{a_n \mathrel{E_S} b_m}$ is always
contained in \( \omega \setminus \II b \): if not, there would be \( m \in
\II b \) with \( a_n \mathrel{E_{S}} b_m \); but then by (iv) there would be
\( k \in \II a \) with \( b_m \mathrel{S} a_k \), so that also \( n \in \II a
\) by \( a_n \mathrel{S} b_m \mathrel{S} a_k \) and Lemma~\ref{Ibasic}(1).
Therefore we have constructed \( f
\restriction (\omega \setminus \II a) \) with range in \( \omega \setminus
\II b \): this leaves us with the task of defining $f$ on $\II a$, with range
in $\II b$. This can be done recursively as follows. Assume $n \in \II a$ and
we already defined $f(\ell)$ for all $\ell \in \II a$ with $\ell<n$. By (iv)
there exists $m \in \II b$ such that $a_n \mathrel{S} b_m$. By Lemma
\ref{Ibasic}(3) there are infinitely many $m' \in \II b$ with $b_m
\mathrel{S} b_{m'}$ and hence $a_n \mathrel{S} b_{m'}$. If $m'$ is the least
such which is not yet in the image of $f$ we can set $f(n)=m'$ preserving
injectivity. The definition of $g$ witnessing $\vec{b} \mathrel{S^{\inj}}
\vec{a}$ is symmetric.
\end{proof}

The following theorem shows in particular that if $S$ is Borel then
$S^{\inj}$ is far from being a complete analytic quasi-order.

\begin{theorem} \label{th:S^inj not complete}
If $S$ is a Borel quasi-order on the standard Borel space $X$ then for every
$\vec{a} \in \pre\omega X$ the equivalence class $\set{\vec{b} \in \pre\omega
X}{\vec{a} \mathrel{E_{S^{\inj}} }\vec{b}}$ is Borel.
\end{theorem}

\begin{proof}
Fix $\vec{a} \in \pre\omega X$ and let $\rho=\rho(\vec{a})$. We describe a
Borel procedure, based on the characterization of Lemma \ref{Einj:char}, for
checking, given any $\vec{b} \in \pre\omega X$, whether $\vec{a}
\mathrel{E_{S^{\inj}} }\vec{b}$.

The procedure starts by computing $\I\alpha b$ for every $\alpha \leq \rho+1$
(since $\rho$ is fixed, this is Borel). If $\I\alpha b = \I{\alpha+1}b$ for
some $\alpha<\rho$, or if $\I\rho b \neq \I{\rho+1}b$ then the procedure
gives a negative answer (because (i) of Lemma \ref{Einj:char} fails).

Otherwise the procedure looks at every $\alpha < \rho$ and $n \in \I\alpha a
\setminus \I{\alpha+1}a$, computes the (finite) cardinality of $\set{m \in
\omega}{a_n \mathrel{E_S} a_m} = \set{m \in \I\alpha a}{a_n \mathrel{E_S}
a_m}$ (see Lemma~\ref{Ibasic}(2)), and checks whether it coincides with the
cardinality of $\set{m \in \omega}{a_n \mathrel{E_S} b_m}$. If this fails (so
that (ii) of Lemma \ref{Einj:char} is not true) the procedure gives a
negative answer. Otherwise it performs a similar operation reversing the
roles of $\vec{a}$ and $\vec{b}$, and checks whether (iii) holds.

If everything works out, the procedure checks whether
\[
\forall n\in \II a \ \exists m\in \II b \, (a_n\mathrel{S} b_m) \land
\forall n\in \II b \exists m\in \II a \, (b_n\mathrel{S} a_m)
\]
(a Borel condition) and gives the final answer.
\end{proof}

In spite of Theorem \ref{th:S^inj not complete}, $S$ can be a Borel
quasi-order without $\mathrel{E_{S^{\inj}}}$ being Borel, as Propositions
\ref{prop:Q} and \ref{prop:P(omega)} show. However if we restrict \( S^{\inj}
\) to the collection \( (\pre{\omega}{X})_\alpha \) of those \( \vec{a} \in
\pre{\omega}{X} \) such that \( \rho(\vec{a}) < \alpha \) for some fixed \(
\alpha < \omega_1 \), then the proof Theorem \ref{th:S^inj not complete}
shows that the equivalence relation associated to \( {S^{\inj}_\alpha} =
{S^{\inj} \restriction (\pre{\omega}{X})_\alpha} \) is Borel (notice that
$(\pre{\omega}{X})_\alpha$ is Borel and $E_{S^\inj}$-invariant by condition
(i) of Lemma~\ref{Einj:char}). However even more is true, as we can show that
these Borel equivalence relations have a Borel upper bound.

\begin{proposition}\label{prop:restrictedpreceqinj}
Let \( S \) be a Borel quasi-order on a standard Borel space $X$. There
exists a Borel quasi-order $R$ on a standard Borel space $Y$ such that for
every \( \alpha < \omega_1 \), \( E_{S^{\inj}_\alpha} \leq_B E_{R^\cf}\). In
particular $E_{S^{\inj}_\alpha}$ is Borel.
\end{proposition}
\begin{proof}
Let $(Y,R)$ be the direct sum
\begin{equation} \label{eq:(Y,R)}
(\omega,{=}) \oplus (X,{S}) \oplus (X \times \omega, E_{S} \times {=}).
\end{equation}
Given $\alpha < \omega_1$, let $i \colon \alpha \to \omega$ be an injection.

For any $\vec{a} = (a_n)_{n \in \omega} \in (\pre{\omega}{X})_\alpha$ we
define $f(\vec{a}) = (a'_n)_{n \in \omega} \in \pre{\omega}{Y}$ as follows.
We let $a'_0 = i(\rho(\vec{a})) \in \omega$ and $a'_{n+1}$ is defined
according to whether $n \in I^{\vec{a}}$. If $n \in I^{\vec{a}}$ we let
$a'_{n+1} = a_n \in X$. If $n \notin I^{\vec{a}}$ we let $a'_{n+1} = (a_n,k)
\in X \times \omega$ where $k=|\{m \in \omega \mid a_n \mathrel{E_S} a_m\}|$
(this cardinality is finite by Lemma~\ref{Ibasic}(2)).

Since $\alpha$ is fixed, we can recover in a Borel-in-$\vec{a}$ way
$i(\rho(\vec{a}))$ and $I^{\vec{a}}$ and hence $f \colon
(\pre{\omega}{X})_\alpha \to \pre{\omega}{Y}$ is Borel.

Applying Lemma~\ref{Einj:char}, it is straightforward to check that $f$
witnesses \( E_{S^{\inj}_\alpha} \leq_B E_{R^\cf}\).
\end{proof}

\begin{corollary} \label{cor:restrictedP(omega)}
For every positive natural number \( N \), the equivalence relation \(
E_{(\pow_{<N}(\omega), \subseteq)^{\inj}} \) is Borel. Therefore, for $N>2$,
\( (\pow_{<N}(\omega), \subseteq)^{\inj} \) is an analytic non-Borel
quasi-order whose associated equivalence relation is Borel.
\end{corollary}

\begin{proof}
This follows from \( \pre{\omega}{(\pow_{<N}(\omega))} \subseteq (
\pre{\omega}{(\pow(\omega))})_{N+1} \).
\end{proof}

Corollary~\ref{cor:restrictedP(omega)} should be contrasted with the fact
that  \( E_{(\pow_{< \omega}(\omega), {\subseteq})^{\inj}} \) is already a \(
\boldsymbol{\Sigma}^1_1 \)-complete subset of \( \pre{\omega}{(\pow_{<
\omega}(\omega))} \times \pre{\omega}{(\pow_{< \omega}(\omega))} \) by
Remark~\ref{rmk:restrictedP(omega)}.

It may also be interesting to notice that there exist Borel quasi-orders \( S
\) such that \( E_{S^{\cf}} \sim_B E_{S^{\inj}} \) even though \( S^{\cf} <_B
S^{\inj} \) (see the subsequent Proposition~\ref{prop:associatedeqrel}(2)).
In what follows $\subseteq^{\inj}_\alpha$ and $\subseteq^{\cf}_\alpha$ stand
for $((\pre{\omega} \pow(\omega))_\alpha, {\subseteq^{\inj}})$ and
$((\pre{\omega} \pow(\omega))_\alpha, {\subseteq^{\cf}})$ respectively. If
$\alpha \geq 4$, the quasi-order \( \subseteq_\alpha^{\inj} \) is proper
analytic by Remark~\ref{rmk:restrictedP(omega)} and the observation in the
proof of Corollary \ref{cor:restrictedP(omega)}, and hence
${\subseteq_\alpha^\cf} <_B {\subseteq_\alpha^\inj}$. Similarly,
Remark~\ref{rmk:restrictedP(omega)} implies that \( (\pow_{< N}(\omega),
{\subseteq})^\cf <_B (\pow_{< N}(\omega), {\subseteq})^\inj \) for all \( 2 <
N < \omega \). To show that $E_{\subseteq^{\cf}_\alpha} \sim_B
E_{\subseteq^{\inj}_\alpha}$ we use the following general fact.

\begin{fact}\label{fact1}
There is a Borel function $f \colon  \pre{\omega}X \to (\pre{\omega}X)_1$ such that
$\vec{a} \mathrel{E_{S^\cf}} f( \vec a )$ for every $\vec{a} \in
\pre{\omega}X$.
\end{fact}
\begin{proof}
Define $f$ to be a Borel function such that $f( \vec a )$ is a sequence
repeating every element of $\vec{a}$ infinitely many times.
\end{proof}

\begin{proposition}\label{prop:associatedeqrel}
\begin{enumerate}[(1)]
\item For every $0<\alpha < \omega_1$ we have
\[
E_{\subseteq^{\cf}_\alpha} \sim_B E_{\subseteq^{\inj}_\alpha} \sim_B E_{\subseteq^{\cf}};
\]
\item for every \( 2 \leq N < \omega \)
\[
E_{(\pow_{<N}(\omega), \subseteq)^{\cf}} \sim_B E_{(\pow_{<N}(\omega), \subseteq)^{\inj}}.
\]
\end{enumerate}
\end{proposition}
\begin{proof}
To prove part (1), it suffices to show that $E_{\subseteq^{\inj}_\alpha}
\leq_B E_{\subseteq^{\cf}}$ and $E_{\subseteq^{\cf}} \leq_B
E_{\subseteq^{\cf}_1}$. The latter follows from Fact \ref{fact1}. To prove
the former we use Proposition \ref{prop:restrictedpreceqinj} and its proof.
Let $(Y,R)$ be as in~\eqref{eq:(Y,R)} with $(X,{S}) =
(\pow(\omega),{\subseteq})$: by (the proof of)
Proposition~\ref{prop:restrictedpreceqinj}, it is clearly enough to show that
$R \leq_B {\subseteq}$.
We define $f \colon Y \to
\pow(\omega)$ by
\[
f(a) = \left\{
   \begin{array}{ll}
   \{\langle a,0 \rangle\}, & \hbox{if $a \in \omega$;} \\
     \{\langle n,1 \rangle \mid n \in a\}, & \hbox{if $a \in \pow(\omega)$;} \\
\begin{split}\{\langle 2n,k+2 \rangle & \mid n \in b\} \mathrel{\cup} \\
\{\langle &2n+1, k+2 \rangle \mid n \notin b\}, \end{split} & \hbox{if $a = (b,k) \in \pow(\omega) \times \omega$.}
   \end{array}
 \right.
\]
Since $E_{\subseteq}$ is equality on $ \pow (\omega )$, it is easy to check that $f$
is the desired Borel reduction of $R$ to $\subseteq$.

To prove the nontrivial reduction in part (2) one can use a similar argument.
Let \( (Y,R) \) be as in~\eqref{eq:(Y,R)} with \( (X,{S}) =
(\pow_{<N}(\omega), {\subseteq}) \). Since \(
\pre{\omega}{(\pow_{<N}(\omega))} =
(\pre{\omega}{(\pow_{<N}(\omega))})_{N+1} \), \(  E_{S^\inj} =
E_{S^\inj_{N+1}} \leq_B E_{R^\cf} \) by
Proposition~\ref{prop:restrictedpreceqinj}, it is enough to show that \(
(Y,R) \) is Borel reducible to \( (\pow_{<N}(\omega), {\subseteq}) \). This
is witnessed by the map
\[
f(a) =
   \begin{cases}
   \{\langle a,0 \rangle\}, & \hbox{if $a \in \omega$;} \\
     \{\langle n,1 \rangle \mid n \in a\}, & \hbox{if $a \in \pow_{<N}(\omega)$;} \\
\{\langle \langle n , |b| \rangle, k+2  \rangle \mid n \in b\}  & \hbox{if $a = (b,k) \in \pow_{<N}(\omega) \times \omega$.}\qedhere
   \end{cases}
\]
\end{proof}

Proposition~\ref{prop:associatedeqrel}(2) should be contrasted with the fact
that \( E_{(\pow_{<\omega}(\omega), \subseteq)^{\cf}} <_B
E_{(\pow_{<\omega}(\omega), \subseteq)^{\inj}} \) because \(
E_{(\pow_{<\omega}(\omega), \subseteq)^{\cf}} \) is Borel while \(
E_{(\pow_{<\omega}(\omega), \subseteq)^{\inj}} \) is proper analytic by
Remark~\ref{rmk:restrictedP(omega)}.

\subsection{A family of proper analytic quasi-orders $S$ with $S <_B
S^{\inj}$}

Combining the next simple lemma together with Proposition~\ref{prop:P(omega)}
and Theorem \ref{th:S^inj not complete}, we will construct a large class of
proper analytic quasi-orders \( S \) which are not stable under the jump
operators \( S \mapsto S^{\cf} \) and \( S \mapsto S^{\inj} \). Examples of
this kind were already provided in~\cite[Section 4]{Camerlo2007}, but those
considered here are different, as shown in Remark~\ref{rmk:CM07}.

\begin{lemma} \label{lem:upperlowercones}
Let \( (X,S) \) be an analytic quasi-order. Then
\begin{enumerate}[(1)]
\item if there are \( \boldsymbol{\Sigma}^1_1 \)-complete upper cones in \(
    S \), then \( E_{S^{\cf}} \) is \( \boldsymbol{\Sigma}^1_1
    \)-complete as a subset of \( \pre{\omega}{X} \times \pre{\omega}{X}
    \);
\item if there are \( \boldsymbol{\Sigma}^1_1 \)-complete lower cones in \(
    S \), then \( E_{S^{\cf}} \) contains a \( \boldsymbol{\Sigma}^1_1
    \)-complete equivalence class.
\end{enumerate}
Similar results hold when \( S^{\cf} \) is replaced by \( S^{\inj} \).
\end{lemma}
\begin{proof}
For part (1), let \( p \in X \) be such that the upper cone \( C^S(p) \) is
\( \boldsymbol{\Sigma}^1_1 \)-complete, and define the map \( f \colon X \to
\pre{\omega}{X} \times \pre{\omega}{X}\) by setting \( f(x) = (p
{}^\smallfrown{} x^{\infty}, x^{\infty}) \). Since clearly \( x^\infty
S^{\cf} p {}^\smallfrown{} x^{\infty} \) for all \( x \in X \) and
\[
p {}^\smallfrown{} x^{\infty} S^{\cf} x^\infty \iff x \in C^S(p),
\]
the function \( f \) continuously reduces \( C^S(p) \) to \( E_{S^{\cf}} \).

For part (2), let \( q \in X \) be such that the lower cone \( C_S(q) \)  is
\( \boldsymbol{\Sigma}^1_1 \)-complete, and define the map \( g \colon X \to
\pre{\omega}{X} \) by setting \( g(x) = x {}^\smallfrown{} q^{\infty} \).
Arguing as above, it is easy to see that \( g \) continuously reduces \(
C_S(q) \) to the \( E_{S^{\cf}} \)-equivalence class of \( q^{\infty} \).

The same arguments work when \( S^{\cf} \) is replaced by \( S^{\inj} \).
\end{proof}

\begin{corollary} \label{cor:newexamplesofjump}
Let \( R \) be a Borel quasi-order such that \( (\pow(\omega ), \subseteq)
\leq_B R \). Then \( S = R^{\inj} \) is a proper analytic quasi-order such
that \( S <_B S^{\cf} \) (hence also \( S <_B S^{\inj} \)).
\end{corollary}
\begin{proof}
By Proposition~\ref{prop:P(omega)}(2), \( S \) contains \(
\boldsymbol{\Sigma}^1_1 \)-complete lower cones and in particular it is a
proper analytic quasi-order.
By Lemma~\ref{lem:upperlowercones}(2)
there is an \( E_{S^{\cf}} \)-equivalence class which is proper analytic.
Therefore \( S^{\cf} \nleq_B S \) because all \( E_S \)-equivalence classes
are Borel by Theorem~\ref{th:S^inj not complete}.
\end{proof}

\begin{remark}\label{rmk:CM07}
If a quasi-order of the form \( R^{\inj} \) is proper analytic, it cannot
satisfy the hypothesis of~\cite[Corollary 4.3]{Camerlo2007}. To see this,
observe that the proof of~\cite[Theorem 4.2]{Camerlo2007} shows that if a
directed (i.e.\ such that every pair of elements has un upper bound)
quasi-order is not Borel then it does not satisfy that hypothesis. Since
\(R^{\inj}\) is always directed (and in fact even countably-directed) our
examples are different as claimed.
\end{remark}

\section{Ultrametric Polish spaces with a fixed set of distances}\label{sect:setup}

All metric spaces we consider are always assumed to be nonempty. Let \( \RR^+
= \{ r \in \RR \mid r \geq 0 \} \). Let \( U \) be an ultrametric space with
distance \( d_U \). Notice that \( d_U \) is an ultrametric if and only if  for every \( x,y,z \in U \) at
least two of the distances \( d_U(x,y), d_U(x,z), d_U(z,y) \) equal \( \max
\{ d_U(x,y), d_U(x,z), d_U(z,y) \} \) (i.e.\ all triangles are isosceles with
legs not shorter than the base). Thus we have
\begin{equation} \label{eq:basicultrametric}
  d_U(x,z) < d_U(z,y) \Rightarrow
d_U(x,y) =   d_U(z,y)  .
\end{equation}
Recall also that in an ultrametric space every open ball is also closed.

We say that \( U \) is an \emph{ultrametric Polish space} if it is separable
and the ultrametric \( d_U \) is complete.\footnote{Let us remark that if a
Polish space \( X \) admits a compatible ultrametric, then it also admits a
compatible ultrametric \( d \) which is also complete. This is because under
our assumptions \( X \) must be zero-dimensional, and hence homeomorphic to a
closed subset \( F \) of \( \pre{\omega}{\omega} \) by~\cite[Theorem
7.8]{Kechris1995}. Transferring back on \( X \) the complete ultrametric on
\( F \) induced by the standard metric on \( \pre{\omega}{\omega} \) we
obtain \( d \) as required.}

We denote by \( D(U) \) the set of distances that are realized by points in
\( U \), i.e.
\[
D(U) = \{ r \in \RR^+ \mid \exists x,y \in U (d_U(x,y) = r) \}.
\]
Let \( \mathcal{D} \) denote the set of all countable \( D \subseteq \RR^+ \)
with \( 0 \in D \).

\begin{lemma} \label{countablymanydistances}
\( D(U) \in \mathcal{D} \) for every separable ultrametric space \( U \).
\end{lemma}

\begin{proof}
Let $Q$ be a countable dense subset of $U$: it suffices to show that $D(U)=\{
d_U(p,q)\mid p,q\in Q\} $. Clearly \( d_U(p,p) = 0 \), so let \( x,y \in U \)
be such that \( d_U(x,y) = r \neq 0 \), and let \( p,q \in Q \) be such that
\( d_U(x,p), d_U(y,q) < r \). Then using~\eqref{eq:basicultrametric} we get
\( d_U(p,y) = d_U(x,y)  =r \), whence
\[
d_U(p,q) =  d_U(p,y) =r . \qedhere
\]
\end{proof}

Conversely, given \( D \in \mathcal{D } \) one can construct a canonical
ultrametric Polish space \( U(D) \) with \( D(U(D)) = D \).

\begin{defin} \label{def:U(D)}
Let \( D \in \mathcal{D} \). Then \( U(D) \) is the ultrametric Polish space
with domain \( D \) and distance function defined by \( d_{U(D)}(r,r') = \max
\{ r,r' \} \) for \( r\neq r' \) and \( d_{U(D)}(r,r) = 0 \).
\end{defin}

Given \( D \in \mathcal{D} \), we denote by \( \U_D \) the set of all
ultrametric Polish spaces $U$ with \( D(U) \subseteq D \), and by \( \U^\star_D
\) the set of all ultrametric Polish spaces \( U \) such that \( D(U) = D \).

Recalling that a (nonempty) metric space is Polish if and only if it is
isometric to an element of \( F(\mathbb{U}) \), the collection of all
nonempty closed subsets of the Urysohn space \( \mathbb{U} \) (this notation
differs slightly from the one used in \cite{Kechris1995}, where \(
F(\mathbb{U}) \) includes the empty set), we can use the following
formalizations of \( \U_D \) and \( \U^\star_D \).

\begin{defin}\label{not:UR}
Let
\[
\U_D = \{ U \in F(\mathbb{U})  \mid d_U = d_{\mathbb{U}} \restriction U^2 \text{ is an ultrametric and } D(U) \subseteq D \}
\]
and
\[
\U^\star_D = \{ U \in F(\mathbb{U})  \mid d_U = d_{\mathbb{U}} \restriction U^2 \text{ is an ultrametric and } D(U) = D \} .
\]
\end{defin}

\begin{notation} \label{notation:psi_n}
Using~\cite[Theorem 12.13]{Kechris1995}, we fix once and for all a sequence
of Borel functions $\psi_n \colon F(\mathbb{U}) \to \mathbb{U}$ such that for
every $F \in F(\mathbb{U})$ the sequence $( \psi_n(F) )_{ n \in \omega }$ is
an enumeration (which can be assumed without repetitions if $F$ is infinite)
of a dense subset of $F$. Notice that if $F \in F(\mathbb{U})$ is discrete
then $( \psi_n(F) )_{ n \in \omega }$ is an enumeration of $F$.
\end{notation}

\begin{proposition} \label{prop:standardBorel}
The sets \( \U_D \) and \( \U^\star_D \) are both Borel subsets of the
standard Borel space $F( \mathbb U )$, hence they are standard Borel spaces
as well.
\end{proposition}
\begin{proof}
With the same argument as in the proof of Lemma \ref{countablymanydistances},
$U \in F(\mathbb{U})$ is in $ \U_D$ if and only if  \( d_U \) is an
ultrametric\footnote{This is a Borel condition because it is enough to check
that \( d_U \) satisfies the definition of ultrametric on the dense set \( \{
\psi_n(U) \mid n \in \omega \} \subseteq U \).} and $d_U(\psi_n(U),\psi_m(U))
\in D$ for all $n,m\in\omega$. To deal with $\U^\star_D$, add the following
condition: for all $r\in D$ there exist $n,m\in\omega$ such that
$d_U(\psi_n(U),\psi_m(U))=r$.
\end{proof}

Another possible formalization of \( \U_D \) and \( \U^\star_D \) uses,
instead of $\mathbb U$, the Polish $D$-ultrametric Urysohn space
$\mathbb{U}^\U_D$. This is the unique (up to isometry) ultrametric Polish
space with $D(\mathbb{U}^\U_D) = D$ which is ultrahomogeneous and universal
for \( \U_D \). The space $\mathbb{U}^\U_D$ can also be characterized as the
unique (up to isometry) ultrametric Polish space with the $D$-ultrametric
Urysohn property: for every finite ultrametric space $B$ with $D(B) \subseteq
D$, every $A \subseteq B$, and every isometric embedding $\varphi \colon A \to
\mathbb{U}^\U_D$, there is an isometric embedding $\varphi^* \colon  B \to
\mathbb{U}^\U_D$ such that $\varphi^* \restriction A = \varphi$. More about
$\mathbb{U}^\U_D$, including several constructions of the space, can be found
in \cite{GaoShao2011}.

Now \( \U_D \) may be identified with $F(\mathbb{U}^\U_D)$, while \(
\U^\star_D \) becomes a Borel subset of $F(\mathbb{U}^\U_D)$, which we denote
by $F^\star (\mathbb{U}^\U_D)$. Although our official definition of $\U_D$
and $\U^\star_D$ is as in Definition \ref{not:UR}, it will be sometimes
convenient to work instead with their counterparts $F(\mathbb{U}^\U_D)$ and
$F^\star (\mathbb{U}^\U_D)$. The next lemma shows that the two formalizations
are essentially equivalent.

\begin{lemma}\label{lemma:2form}
There exist Borel maps $\Phi \colon  F(\mathbb{U}^\U_D) \to \U_D$ and $\Psi
\colon  \U_D \to F(\mathbb{U}^\U_D)$ such that $\Phi(U)$ and $U$ are
isometric for any $U \in F(\mathbb{U}^\U_D)$ and $\Psi(V)$ and $V$ are
isometric for any $V \in \U_D$. In particular, the range of $\Phi
\restriction F^\star (\mathbb{U}^\U_D)$ is included in $\U^\star_D$ and the
range of $\Psi \restriction \U^\star_D$ is included in $F^\star
(\mathbb{U}^\U_D)$.
\end{lemma}
\begin{proof}
Any isometric embedding of $\mathbb{U}^\U_D$ in $\mathbb{U}$ induces a map
$\Phi$ with the desired properties.

To define $\Psi$ fix a countable dense subset $\set{q_k}{k \in \omega}$ in
$\mathbb{U}^\U_D$. Given $V \in \U_D$ consider $\{ \psi_n (V) \mid n\in\omega
\}$, where the functions \( \psi_n \) are as in
Notation~\ref{notation:psi_n}. In a Borel way we recursively define a
sequence $\set{k_n}{n \in \omega}$ such that the map $\psi_n (V) \mapsto
q_{k_n}$ is an isometry. Once this is done we can define $\Psi(V)$ to be the
closure of $\set{q_{k_n}}{n \in \omega}$ in $\mathbb{U}^\U_D$.

Assuming we defined $k_n$ for $n<m$, let $k_m$ be the least $k$ such that
$d(\psi_m(V), \psi_n(V)) = d(q_k, q_{k_n})$ for every $n<m$. To see that such
a $k$ exists, notice that by the $D$-ultrametric Urysohn property there
exists $y \in \mathbb{U}^\U_D$ such that $d(\psi_m(V), \psi_n(V)) = d(y,
q_{k_n})$ for every $n<m$. If $d(q_k,y)< \min \set{d(q_{k_n}, y)}{n<m}$, then
$k$ has the required property by~\eqref{eq:basicultrametric}.
\end{proof}

Let \( \isom \) and \( \sqsubseteq \) denote respectively the relations of
isometry and isometric embeddability between metric spaces. When restricted
to ultrametric spaces, these relations can be described as follows.

\begin{notation} \label{not:C^x_r}
Given an ultrametric space \( U \), \( r \in \RR^+ \), and \(x \in U \), let
\( C^x_r(U) = \{ y \in U \mid d_U(x,y) = r \} \).
\end{notation}

\begin{lemma} \label{lem:limit}
Let \( U,U' \) be two ultrametric spaces (not necessarily separable nor
complete).
Then the following are equivalent:
\begin{enumerate}[\quad(i)]
\item
\( U \isom U' \);
\item for every \( x \in U \) there exists \( x' \in U' \) such that for
    every \( r \in D(U) \cup D(U') \), either $C_r^x(U)$ and $C_r^{x'}(U')$
    are both empty, or else \( C^x_r(U) \isom C^{x'}_r(U') \);
\item there exist \( x \in U \) and \( x' \in U' \) such that for every \(
    r \in D(U) \cup D(U') \), either $C_r^x(U)$ and $C_r^{x'}(U')$ are both
    empty, or else \( C^x_r(U) \isom C^{x'}_r(U') \).
\end{enumerate}
Similarly, the following are equivalent:
\begin{enumerate}
\item[\quad(iv)]
\( U \sqsubseteq U' \);
\item[\quad(v)] for every \( x \in U \) there exists \( x' \in U' \) such
    that \( C^x_r(U) \sqsubseteq C^{x'}_r(U') \) for every distance \( r
    \in D(U) \) realized by $x$;
\item[\quad(vi)] there exist \( x \in U \) and \( x' \in U' \) such that \(
    C^x_r(U) \sqsubseteq C^{x'}_r(U') \) for every distance \( r \in D(U)
    \) realized by $x$.
\end{enumerate}
\end{lemma}

\begin{proof}
We first consider the part concerning isometric embeddability.
To prove (iv) implies (v), let \( \varphi \colon U \to U' \) be an isometric
embedding and pick \( x \in U \).  Then setting \( x' = \varphi(x) \) we get
that \( \varphi \restriction C^x_r(U) \) witnesses \( C^x_r(U) \sqsubseteq
C^{x'}_r(U') \) for every distance \( r \in D(U) \) realized by $x$.

(v) implies (vi) is trivial, while to prove (vi) implies (iv) fix \( x \in U
\) and \( x' \in U' \) such that for every distance \( r \in D(U) \) realized
by $x$ there is an isometric embedding \( \varphi_r \colon C^x_r(U) \to
C^{x'}_r(U') \). We claim that \( \varphi = \bigcup_r\varphi_r \) is an
isometric embedding of \( U \) into \( U' \). To see this,  let \( y,z \in U
\). If there is \( r \in D(U) \) such that \( d_U(y,x) = d_U(z,x) = r \) then
\( y,z \in C^x_r(U) \): it follows that \( \varphi(y) = \varphi_r(y) \) and
\( \varphi(z) = \varphi_r(z) \), whence \( d_U(y,z) = d_{U'}(\varphi(y),
\varphi(z)) \) since \( \varphi_r \) is distance preserving. If instead \(
d_U(y,x) = r_y < r_z = d_U(z,x) \) then \( d_U(y,z) = r_z  \)
by~\eqref{eq:basicultrametric}. Since \( \varphi(y)  = \varphi_{r_y}(y) \in
C^{x'}_{r_y}(U') \) and \( \varphi(z) = \varphi_{r_z}(z) \in C^{x'}_{r_z}(U')
\), it follows that \( d_{U'}(\varphi(y),x') = r_y \) and \(
d_{U'}(\varphi(z),x') = r_z \). Therefore \( d_{U'}(\varphi(y), \varphi(z)) =
r_z  = d_U(y,z) \) by~\eqref{eq:basicultrametric} again.

The same proof works for isometry as well: it is enough to notice that in
this case all the isometric embeddings involved are automatically surjective.
\end{proof}

\begin{notation}
Let $\isom_D$ and $\isom^\star_D$ denote the restrictions of $\isom$ to
spaces in \( \U_D \) and \( \U^\star_D \), respectively. Similarly, let
$\sqsubseteq_D$ and $\sqsubseteq^\star_D$ be the restrictions of
$\sqsubseteq$ to \( \U_D \) and \( \U^\star_D \).
\end{notation}

Using Proposition~\ref{prop:standardBorel}, it is straightforward to see that
all the relations \( (\U_D, \isom_D) \), \( (\U^\star_D, \isom^\star_D) \),
\( (\U_D, \sqsubseteq_D) \), and \( (\U^\star_D, \sqsubseteq^\star_D) \) are
analytic.

\begin{remark}\label{rem:2form}
By Lemma~\ref{lemma:2form} the equivalence relations $\isom_D$ and
$\isom^\star_D$ are classwise Borel isomorphic to $\isom \restriction F
(\mathbb{U}^\U_D)$ and $\isom \restriction F^\star (\mathbb{U}^\U_D)$
respectively. Similarly for the quasi-orders $\sqsubseteq_D$ and
$\sqsubseteq^\star_D$.
\end{remark}

Our aim is to study the complexity of the above relations for various \( D
\in \mathcal{D} \). Notice that since \( \U^\star_D \subseteq \U_D \), the
identity function on \( \U^\star_D \) witnesses that \( {\isom^\star_D}
\leq_B {\isom_D} \) and \( {\sqsubseteq^\star_D} \leq_B {\sqsubseteq_D} \)
for every \( D \in \mathcal{D} \). We will show later
(Corollaries~\ref{cor:isomequivalent} and~\ref{cor:equivalent}), using some
nontrivial results, that the converse reductions hold as well.

The following lemma extends \cite[Theorem 8.2]{GaoShao2011}. Another
strengthening, without the hypotheses of continuity in $0$, will be given in
Corollary \ref{improvedembedding}.

\begin{lemma} \label{lemma:basic}
Let \( D,D' \in \mathcal{D} \).
\begin{enumerate}[(1)]
\item If there exists a Polish ultrametric preserving injection $f \colon D
    \to D'$ then \( {\isom_{D}} \) classwise Borel embeds into \(
    {\isom_{D'}} \) and \( {\sqsubseteq_{D}} \) classwise Borel embeds into
    \( {\sqsubseteq_{D'}} \);
\item If there exists a Polish ultrametric preserving bijection $f \colon D
    \to D'$ then the relations \( {\isom_D} \), \( \isom^\star_{D} \), \(
    {\sqsubseteq_D} \), and \( {\sqsubseteq^\star_D} \) are classwise Borel
    isomorphic\footnote{In fact, they are Borel isomorphic.} to \(
    {\isom_{D'}} \), \( \isom^\star_{D'} \), \( {\sqsubseteq_{D'}} \), and
    \( {\sqsubseteq_{D'}^{\star }} \), respectively.
\end{enumerate}
\end{lemma}

\begin{proof}
We explicitly consider only the case of isometry, but the same proof works
for isometric embeddability as well. Let $f \colon D \to D'$ be a Polish
ultrametric preserving injection and $D'' \subseteq D'$ be the range of $f$.
We show that $\isom_D$ is classwise Borel isomorphic to $\isom_{D''}$. By
Remark~\ref{rem:2form} we can work with $\isom \restriction F
(\mathbb{U}^\U_D)$ and $\isom \restriction F (\mathbb{U}^\U_{D''})$. Let $C
\in F (\mathbb{U}^\U_{D''})$ be such that there exists a homeomorphism
$\varphi \colon  \mathbb{U}^\U_D \to C$ satisfying $d(\varphi(x), \varphi(y))
= f(d(x,y))$ for every $x,y \in \mathbb{U}^\U_D$ (the existence of $C$
follows from the fact that $f$ is Polish ultrametric preserving). The map
$\varphi$ induces the Borel bijection $\Phi \colon F (\mathbb{U}^\U_D) \to
F(C) \subseteq F (\mathbb{U}^\U_{D''})$. Notice that, by injectivity of $f$,
$\Phi$ reduces $\isom \restriction F (\mathbb{U}^\U_D)$ to $\isom
\restriction F (\mathbb{U}^\U_{D''})$.

Using $f^{-1}$ we analogously define $C' \in F(\mathbb{U}^\U_D)$ and a Borel
$\Psi \colon F (\mathbb{U}^\U_{D''}) \to F(C') \subseteq F
(\mathbb{U}^\U_D)$. By construction $\Psi(\Phi(U)) \isom U$ and
$\Phi(\Psi(V)) \isom V$ for all $U \in F(\mathbb{U}^\U_D)$ and $V \in
F(\mathbb{U}^\U_{D''})$. This shows that the closures under isometries of the
ranges of $\Phi$ and $\Psi$ are $F(\mathbb{U}^\U_{D''})$ and
$F(\mathbb{U}^\U_D)$ respectively, completing the proof that $\isom
\restriction F (\mathbb{U}^\U_D)$ and $\isom \restriction F
(\mathbb{U}^\U_{D''})$ are classwise Borel isomorphic.

This gives (1) and the part of (2) concerning \( \isom_D \) e \( \isom_{D'}
\), since in the latter case $D''=D'$. For the part of (2) concerning \(
\isom^\star_D \) e \( \isom^\star_{D'} \), notice that the restriction of
$\Phi$ to $F^\star (\mathbb{U}^\U_D)$ has range in $F^\star
(\mathbb{U}^\U_{D''}) = F^\star (\mathbb{U}^\U_{D'})$ and similarly for
$\Psi$.
\end{proof}

For some specific \( D \in \mathcal{D} \), the complexity of the relations \(
\isom_D \) and \( \sqsubseteq_D \) (and sometimes also of \( \isom^\star_D \)
and \( \sqsubseteq^\star_D \))  has already been considered in the
literature, so let us end this section by discussing the known results.

Recall from~\cite{Gao2003,GaoShao2011} the following facts concerning the
complexity of  the relation of isometry on ultrametric Polish  spaces.

\begin{proposition} \label{prop:knownisom}
\begin{enumerate}[(1)]
\item Isometry on arbitrary ultrametric Polish spaces is Borel reducible to countable
    graph isomorphism; thus so is each of the relations \( \isom_D \) for
    \( D \in \mathcal{D} \).
\item If \( 0 \) is a limit point of \( D \in \mathcal{D} \), then \(
    \isom_D \) is in fact Borel bireducible with countable graph isomorphism; thus so
    is the relation of isometry on arbitrary ultrametric Polish spaces.
\item Isometry on discrete ultrametric Polish spaces (and thus also any \(
    \isom_D \) with \( D \in \mathcal{D} \) bounded away from \( 0 \)) is Borel reducible
    to isometry on locally compact ultrametric Polish spaces, which in turn
    is Borel reducible to isometry on arbitrary ultrametric Polish spaces.
\item If \( D \in \mathcal{D} \) is finite of cardinality \( n \in \omega
    \), then  \( \isom_D \) is Borel bireducible with isomorphism between
    countable trees of height \( n \), and thus is Borel.
\end{enumerate}
\end{proposition}

The exact complexity of the equivalence relations of isometry on discrete or
locally compact ultrametric Polish spaces considered in (3) remained unknown,
leading to Question~\ref{mainquest}.
Gao and Kechris  provided in \cite[Section 8]{Gao2003} a first lower bound
for such complexity
by showing that isomorphism between trees on $\omega $ with countably many
infinite branches (equivalently, isometry between countable closed subsets of
\( \pre{\omega}{\omega} \)) is Borel reducible to isometry on discrete
ultrametric Polish spaces. Moreover they observed that isomorphism between
well-founded trees on \( \omega \) (equivalently, isometry between discrete
subsets of \( \pre{\omega}{\omega} \)) Borel reduces to the above isomorphism
relation, and thus is a second lower bound for the complexity of isometry
restricted to the classes of ultrametric Polish spaces mentioned above. They
asked whether these lower bounds are sharp. Another somewhat artificial lower
bound (namely: isomorphism between reverse trees, see the end of
Section~\ref{isomwf}) was isolated by Clemens in~\cite{ClemensPreprint}, and
he asked as well whether this other lower bound is sharp. (We will partially
answer these questions on lower bounds after Corollary \ref{cor:mainquest}
and Theorem~\ref{labeltobeadded}.) More precisely, combining the proofs
of~\cite[Theorem 8.10 and Proposition 8.11]{Gao2003} and~\cite[Proposition
16]{ClemensPreprint} with Lemma~\ref{lemma:basic} we actually get that the
above isomorphism relations isolated by Gao-Kechris and Clemens are lower
bounds for the complexity of \( \isom_D \) for, respectively, \( D \in
\mathcal{D} \) ill-founded (and possibly bounded away from \( 0 \)) and \( D
\in \mathcal{D} \) well-ordered and infinite.

For what concerns isometric embeddability, we already know the following facts.

\begin{proposition}\label{propbqo}
If \( D \in \mathcal{D} \) is finite then both \( \sqsubseteq_D \) and \(
\sqsubseteq^\star_D \) are bqo's, and therefore very simple Borel
quasi-orders (such as $(\omega, {=})$ and $(\omega, {\geq})$) are not Borel
reducible to either of them.
\end{proposition}
\begin{proof}[Sketch of the proof]
If \( D \) has cardinality \(n>0 \), then, as noticed in \cite[Theorem
8.4]{GaoShao2011}, \( \sqsubseteq^\star_D \) is Borel bireducible to the
relation of embeddability on subtrees of \( \pre{< \omega}{\omega} \) of
height \( n \). Hence the result follows from a classic result of
Nash-Williams (\cite[Theorem 2]{NashWil1965}).
\end{proof}

The proof of~\cite[Theorem 2]{NashWil1965} makes an essential use of a strong
form of the Axiom of Choice \( \AC \) (see the proof of~\cite[Lemma
38]{NashWil1965}, in which one argues by induction on a well-ordering of all
barriers). However, only very weak versions of \( \AC \) are actually needed
to obtain Proposition~\ref{propbqo}: indeed, in Section~\ref{sect:wf} we will
provide an alternative proof of this result in \( \ZF + \AC_\omega(\RR) \)
alone (where \( \AC_\omega(\RR) \) is the Axiom of Countable Choice over the
Reals).

\begin{proposition}[{\cite[Theorem 8.3]{GaoShao2011}}] \label{prop:louveaurosendal}
If \( D \) contains a decreasing sequence converging to \( 0 \) then \(
\sqsubseteq_D \) is complete for analytic quasi-orders.
\end{proposition}

\begin{proof}
This is essentially~\cite[Proposition 4.2]{louros}, whose proof involves only
spaces in \( \U_D \) for \( D  = \{ 0, 2^{-n} \mid n \in \omega \} \). For
the general case we choose from \( D \) a strictly decreasing sequence \(
(r_n)_{n \in \omega} \) converging to \( 0 \), and then  in the construction
of Louveau and Rosendal we systematically replace the distance \( 2^{-n} \)
with \( r_n \).
\end{proof}

This has been (essentially) strengthened in~\cite{cammarmot} to the following
result.

\begin{theorem}[{\cite[Theorem 5.19]{cammarmot}}] \label{theorconvergingto0}
If \( D \) contains a decreasing sequence converging to \( 0 \) then \(
\sqsubseteq_D \) is invariantly universal.
\end{theorem}
\begin{proof}
For \( D  = \{ 0, 2^{-n} \mid n \in \omega \} \) this follows from the proof
of \cite[Theorem 5.19]{cammarmot}. The general case follows from
Lemma~\ref{lemma:basic}.
\end{proof}

The proofs of Proposition~\ref{prop:louveaurosendal} and
Theorem~\ref{theorconvergingto0} do not say anything about the relation \(
\sqsubseteq^\star_D \): even when \( D  = \{0\} \cup \{2^{-n} \mid n \in
\omega \} \) the spaces involved in those arguments do not realize all
distances from \( D \).

\section{The complexity of isometry}

In this section we fully answer Question~\ref{questgaoshao}, and as a
by-product also Question~\ref{mainquest}.

\subsection{Ill-founded sets of distances} \label{isomif}
We first focus on the study of the complexity of the relations \( \isom_D \)
and \( \isom^\star_D \) for an ill-founded \( D \in \mathcal{D} \). To this
aim, we will use the following combinatorial objects. A \emph{rooted
combinatorial tree} \( G \) is a connected acyclic graph with a distinguished
vertex called the root of \( G \). The collection of all rooted combinatorial
trees with universe $\omega $ forms a Borel subset \( \RCT \) of the Polish
space $\Mod_{ \L }$ of \( \L \)-structures on $\omega $ (for a suitable
language \( \L \)). The relation of isomorphism on $ \RCT $ is easily seen to
be Borel bireducible with isomorphism on trees on $\omega$, and the latter is
Borel bireducible to countable graph isomorphism by \cite{Friedman1989}.

Each \( G \in \RCT \) can be identified in a Borel-in-\( G \) way with a
graph \( G' \) having as domain a subset of  \( \pre{< \omega}{\omega} \)
closed under subsequences and such that its root is \( \emptyset \) and its
edge relation coincide with the successor relation, i.e.\ for \( s, t \in G'
\subseteq \pre{<\omega}{\omega} \) we have \( s \mathrel{G'} t \) if and only
if \( s \restriction \leng(s) -1 = t \) or \( t \restriction \leng(t) -1 = s
\).

Let \( D \in \mathcal{D} \) contain a strictly decreasing sequence \( (r_n)_{n \in \omega} \) with \( r_n
\to r > 0 \). We define a distance on $\pre{< \omega}{\omega}$ with values in
\( \{ 0 \} \cup \{ r_n \mid  n \in \omega \} \subseteq D \) by
\[
d(s,t)= \begin{cases}
0 & \text{if } s = t \\	
r_n & \text{if } s \neq t \text{ and $n$ is greatest such that } s \restriction n = t \restriction n.
\end{cases}
\]
It is easy to check that \( d \) is an ultrametric. The completeness of \( d
\) follows from the fact that \( r > 0 \).\footnote{Had we defined \( d \)
using a sequence \( (r_n)_{n \in \omega} \) converging to \( 0 \), the
resulting metric would have been a non-complete ultrametric, as if \( x \in
\pre{\omega}{\omega} \) then \( ( x \restriction n )_{ n \in \omega } \)
would be a non-converging \( d \)-Cauchy sequence.} Fix an isometric
embedding $\rho$ of $(\pre{< \omega}{\omega},d)$ into the Urysohn space
$\mathbb{U}$. Then $\rho$ induces a Borel map from the subsets of $\pre{<
\omega}{\omega}$ to $\U_D$.

Given \( G \in \RCT \), we construct an ultrametric Polish space \( U_G
\) with distances in \( D \) as
follows. The domain of \( U_G \) is \( G' \subseteq \pre{< \omega}{\omega} \)
and the distance $d_G$ is the restriction of $d$ to $G'$. Notice that \( s
\in G' \) realizes a distance \( r_n \) (i.e.\ $d_G(s,t) = r_n$ for some $t
\in G'$) if and only if either $n <\leng(s)$ or $n = \leng(s)$ and \( s \) is
not a terminal node of \( G' \) (i.e.\ there exists $t \in G'$ with \( s
\subsetneq t \)). Notice also that
\begin{equation}\label{equivalence}
s \subsetneq t \text{ if and only if } d_G(s,t) = r_{\leng(s)}.
\end{equation}

\begin{defin} \label{theta}
Let \( \theta \colon \RCT \to \U_D \) be the composition of the map sending
$G$ to \(U_G\) with the map induced by $\rho$.
\end{defin}

Clearly \( \theta \) is a Borel map.

\begin{theorem} \label{bigtheorem}

The function \( \theta \) simultaneously reduces isomorphism to isometry, and
embeddability to isometric embeddability.
\end{theorem}

\begin{proof}
Fix first $G,H\in \RCT $ and suppose \( \varphi \) embeds \( G \) into \( H
\). This induces an embedding $\varphi' \colon G'\to H'$. Since \(
\varphi'(\emptyset) = \emptyset \) (because \( \emptyset \) is the root of
both \( G' \) and \( H' \)), it is easy to check by induction on \( \leng(s)
\) that \( \leng(\varphi'(s)) = \leng(s) \) for every \( s \in G' \). Arguing
by induction on \( \leng(t) \) and using the previous observation, it follows
that \( s \subseteq t \iff \varphi'(s) \subseteq \varphi'(t) \) for every \(
s,t \in G' \). Therefore \( \varphi' \) is a distance preserving map
witnessing \( U_G \sqsubseteq U_H \), whence $\theta (G)\sqsubseteq\theta
(H)$. If moreover $\varphi $ is assumed to be an isomorphism, then $\varphi'$
is surjective too, whence $\theta (G) \isom \theta (H)$.

Assume now that \( \varphi \) is an isometric embedding between \( U_G \) and
\( U_H \). Fix an enumeration without repetitions \( ( s_n )_{ n \in \omega }
\) of \( G' \) such that \( s_n \subseteq s_m \Rightarrow n \leq m \). We
will recursively construct a sequence of maps \( \varphi_n \colon G' \to H'
\) with the following properties:
\begin{enumerate}[(i)]
\item
\( \varphi_n \) is distance preserving (i.e.\ an isometric embedding);
\item
\( \varphi_{n}(s_i) = \varphi_i(s_i) \) for every \( i < n \);
\item
\( \leng(\varphi_n(s_n)) = \leng(s_n) \).
\end{enumerate}
Given such a sequence \( ( \varphi_n )_{ n \in \omega } \), define \(
\tilde{\varphi} \colon G' \to H' \) by setting \( \tilde{\varphi}(s_n) =
\varphi_n(s_n) \) for every \( n \in \omega \). By (i) and (ii) above, \(
\tilde{\varphi} \) preserves distances between the spaces \( U_G \) and \(
U_H \). We now claim that \( \tilde{\varphi} \) is a graph embedding of \( G'
\) into \( H' \) (so that $G$ embeds into $H$). Let \( s,t \in G' \) be
linked by an edge in \( G' \), and assume without loss of generality that \(
s \subsetneq t \) (so that, in particular, \( \leng(t) =\leng(s) + 1 \)).
Then by \eqref{equivalence} we have \( d_G(s,t) = r_{\leng(s)} \). Since \(
\tilde{\varphi} \) preserves distances, \( d_H (\tilde{\varphi}(s),
\tilde{\varphi}(t)) = r_{\leng(s)} \). By (iii) above, \(
\leng(\tilde{\varphi}(s)) = \leng(s) \), and therefore \( \tilde{\varphi}(t)
\supsetneq \tilde{\varphi}(s) \) by \eqref{equivalence}. Since \(
\leng(\tilde{\varphi}(t)) = \leng(t) = \leng(s) +1 =
\leng(\tilde{\varphi}(s)) + 1 \) by (iii) again, it follows that \(
\tilde{\varphi}(s) \) and \( \tilde{\varphi}(t) \) are linked by an edge in
\( H' \). A similar argument shows that if \( s,t \in G' \) are such that \(
\tilde{\varphi}(s) \) and \( \tilde{\varphi}(t) \) are linked by an edge in
\( H' \) then \( s \) and \( t \) are linked by an edge in \( G' \).

We show next how to recursively construct the sequence \( ( \varphi_n )_{ n \in \omega  }
\) with the desired properties (i)--(iii). Set \( \varphi_{-1} = \varphi \). Let now \(
n \in \omega \) and suppose that \( \varphi_i \) has been defined for every
\( 0 \leq i < n \) and that it satisfies conditions (i)--(iii) above. We
first show that \( \leng(\varphi_{n-1}(s_n)) \geq \leng(s_n) \). This is
clear if \( n = 0 \), as \( s_0 = \emptyset \) by the choice of the
enumeration \( ( s_n )_{ n \in \omega } \). If \( n > 0 \), suppose toward a
contradiction that \( k = \leng(\varphi_{n-1}(s_n)) < \leng(s_n) \), and let
\( i \in \omega \) be such that \( s_i = s_n \restriction k \). Then \( i < n
\) (since \( s_i \subsetneq s_n \)), so that \( \leng(\varphi_{n-1}(s_i)) =
\leng(s_i) = k \) by inductive hypothesis (conditions (ii) and (iii)).
Therefore \( \leng( \varphi_{n-1}(s_i)) = \leng(\varphi_{n-1}(s_n)) = k \),
whence  \( d_H(\varphi_{n-1}(s_i), \varphi_{n-1}(s_n)) > r_k \) by
injectivity of \( \varphi_{n-1} \). Since \( d_G(s_i,s_n) = r_k \) and \(
\varphi_{n-1} \) is distance preserving, this is a contradiction.

We now define \( \varphi_n \) by redefining \( \varphi_{n-1} \) on \( s_n \)
(and possibly on another sequence), and to do this we distinguish various
cases according to the behavior of \( \varphi_{n-1} \). Let \( k = \leng(s_n)
\).

\textbf{Case 1:} \emph{\( \leng(\varphi_{n-1}(s_n)) = k \)}. Then we set \(
\varphi_n = \varphi_{n-1} \): conditions (i)--(iii) are then trivially
satisfied by \( \varphi_n \) by case assumption and inductive hypothesis.

\textbf{Case 2:} \( \leng(\varphi_{n-1}(s_n)) > k \) but \(
\varphi_{n-1}(s_n) \restriction k \) is not in the range of \( \varphi_{n-1}
\). Then we set \( \varphi_n(s_n) = \varphi_{n-1}(s_n) \restriction k \) and
\( \varphi_n(t) = \varphi_{n-1}(t) \) for every \( t \neq s_n \). Then
(ii)--(iii) are automatically satisfied. To check that \( \varphi_n \) is
still distance preserving, for \( t \neq s_n \) let \( i \in \omega \) be
such that \( d_G(t,s_n) = r_i \). Then \( i \leq k \) and \(
d_H(\varphi_{n-1}(t),\varphi_{n-1}(s_n)) = r_i \) (since \( \varphi_{n-1} \)
is distance preserving), whence  by definition of \( d_H \)
\[
d_H(\varphi_n(t), \varphi_n(s_n)) = d_H(\varphi_{n-1}(t),\varphi_{n-1}(s_n) \restriction k)
= d_H(\varphi_{n-1}(t),\varphi_{n-1}(s_n)) = r_i .
\]

\textbf{Case 3:} \( \leng(\varphi_{n-1}(s_n)) > k \) and there is \( s_i \neq
s_n \) such that \( \varphi_{n-1}(s_i) = \varphi_{n-1}(s_n) \restriction k
\). In this case we set \( \varphi_n(s_n) = \varphi_{n-1}(s_i) =
\varphi_{n-1}(s_n) \restriction k \), \( \varphi_n(s_i) = \varphi_{n-1}(s_n)
\), and \( \varphi_n(t) = \varphi_{n-1}(t) \) for every \( t \in G' \setminus
\{ s_n , s_i\} \). It is clear that condition (iii) is satisfied by
definition. Since \( d_H(\varphi_{n-1}(s_n),\varphi_{n-1}(s_i)) = r_k \) and
\( \varphi_{n-1} \) is distance preserving, \eqref{equivalence} implies that
\( s_n \subsetneq s_i \). This implies \( n < i \), so that also (ii) is
satisfied by \( \varphi_n \) (\( \varphi_n \) coincides with \( \varphi_{n-1}
\) on each \( s_j \) with \( j \neq i,n \), hence in particular on every \(
s_j \) with \( j < n \)). It remains only to show that \( \varphi_n \) is
distance preserving. Assume first  towards a contradiction that there is some
\( t \in G' \) such that \( s_n \subsetneq t \subsetneq s_i \). Then \(
d_H(\varphi_{n-1}(t), \varphi_{n-1}(s_i)) = d_G(t,s_i)  = r_{\leng(t)} < r_k
\), while \( \varphi_{n-1}(s_i) \) can realize, besides \( 0 \), only
distances \( \geq r_k \) in \( U_H \) (by  \( \leng(\varphi_{n-1}(s_i)) = k
\) and the definition of \( d_H \)): this gives the desired contradiction.
Similarly, if \( t \in G' \) is such that \( s_i \subsetneq t \), then \(
d_H(\varphi_{n-1}(t), \varphi_{n-1}(s_i)) = d_G(t,s_i)  = r_{\leng(s_i)} <
r_k \), contradicting again the fact that  in \( U_H \) the point \( \varphi_{n-1}(s_i) \) can
realize, besides \( 0 \), only distances \( \geq r_k \).
Therefore \( s_i \) is an immediate successor of \( s_n \) (i.e.\ \(
\leng(s_i) = k +1 \)) and is a terminal node in \( G' \). By definition of \(
d_G \), this implies that \( d_G(s_i,t) = d_G(s_n,t) \) for every \(  t \in
G' \setminus \{ s_n , s_i\} \). Since \( d_H(\varphi_n(s_i),\varphi_n(t)) =
d_H(\varphi_{n-1}(s_n), \varphi_{n-1}(t)) = d_G(s_n,t) \) and \(
d_H(\varphi_n(s_n),\varphi_n(t)) = d_H(\varphi_{n-1}(s_i), \varphi_{n-1}(t))
= d_G(s_i,t) \), we get \( d_H(\varphi_n(s_i),\varphi_n(t)) = d_G(s_i,t) \)
and \( d_H(\varphi_n(s_n),\varphi_n(t)) = d_G(s_n,t) \). Moreover,
\[
d_H(\varphi_n(s_n), \varphi_n(s_i)) = d_H(\varphi_{n-1}(s_i),\varphi_{n-1}(s_n)) = d_G(s_i,s_n)
\]
(because \( \varphi_{n-1} \) is distance preserving). Since \( \varphi_n
\restriction (G \setminus \{ s_n,s_i \})  = \varphi_{n-1} \restriction (G
\setminus \{ s_n,s_i \}) \) and \( \varphi_{n-1} \) satisfies (i) by
inductive hypothesis, we obtain that \( \varphi_n \) is distance preserving
as well, hence we are done.

If moreover \( \varphi \) is an isometry between \( U_G \) and \( U_H \), it
will be proved that \( \tilde{\varphi} \) is surjective: by the previous
arguments, this implies that \( \tilde{\varphi} \) is an isomorphism between
\( G' \) and \( H' \), whence \( G \) and \( H \) are isomorphic. Since \(
\varphi \) is surjective, it is easy to show by induction on \( n \in \omega
\) that so are all the \( \varphi_n \) because Case 2 cannot occur. Moreover,
whenever $\varphi_n\neq\varphi_{n-1}$, which means that $\varphi_n$ has been
defined according to Case 3, one also has $ \leng (\varphi_{n-1}(s_n))= \leng
(s_n)+1$: otherwise $\varphi_{n-1}(s_n)$ would realize in $H'$ the distance
$r_{ \leng (s_n)+1}$, while this does not happen for $s_n$ in $G'$. So
$\varphi_n$ can differ from $\varphi_{n-1}$ only on two sequences (namely, by
the previous argument, $s_n$ and some immediate successor $s_i$ of it) on
which $\varphi_{n-1}$ is not order preserving, but $\varphi_n$ is. This
implies that for every $t\in H'$ the sequence
$(\varphi_n^{-1}(t))_{n\in\omega }$ is eventually constant --- in fact it
takes at most two values. Clearly, its eventual value is $ \tilde
\varphi^{-1}(t)$.
\end{proof}

\begin{corollary} \label{isomuniversal}
Let $D\in \mathcal D $ be ill-founded. Then the relation of isometry on $
\U_D$ is Borel bireducible with countable graph isomorphism.
\end{corollary}

\begin{proof}
The relation \( \isom_D \) Borel reduces to countable graph isomorphism by
Proposition~\ref{prop:knownisom}, so let us show that countable graph isomorphism (or,
equivalently, isomorphism between trees on \( \omega \), or isomorphism on \(
\RCT \)) Borel reduces to \( \isom_D \). If $D$ contains an ill-founded
subset bounded away from $0$, then apply Theorem \ref{bigtheorem}. Otherwise,
$D$ contains a decreasing infinitesimal sequence $r_n$. Then apply the proof
of \cite[Theorem 4.4]{Gao2003} noticing that even if such a proof uses the
sequence $r_n=2^{-n}$, it actually works for any decreasing infinitesimal
sequence (essentially, this amounts to using Lemma \ref{lemma:basic} above).
\end{proof}

We can now answer Question \ref{mainquest}. Notice that both classes of
ultrametric Polish spaces considered in the next corollary are Borel
$\boldsymbol \Pi^1_1$-complete\footnote{A reduction of the set of
well-founded trees to the class of locally compact subsets of Baire space is
obtained by mapping the tree $T \subseteq \pre{<\omega}{\omega}$ to the body
of the pruned tree
\[
\set{(2s) {}^\smallfrown{} n^k \in \pre{<\omega}{\omega}}{s \in T \land n \text{ is odd} \land k \in \omega},
\]
where $2s$ is the sequence obtained by doubling every entry of $s$ and $n^k$
is the sequence of length $k$ with all entries equal to $n$. The same
reduction shows also that the class of discrete ultrametric Polish spaces is
Borel $\boldsymbol \Pi^1_1$-complete. In the computable (lightface) setting
Nies and Solecki recently proved that the class of locally compact
ultrametric Polish spaces is $\Pi^1_1$-complete using a different
construction (\cite{Nies-Solecki}).}.

\begin{corollary} \label{cor:mainquest}
The relations of isometry on discrete ultrametric Polish spaces and on
locally compact ultrametric Polish spaces are both Borel bireducible with
countable graph isomorphism.
\end{corollary}
\begin{proof}
By Proposition~\ref{prop:knownisom} it is enough to show that \( \isom_D \)
is bireducible with countable graph isomorphism  for some \( D \in
\mathcal{D} \) bounded away from \( 0 \). Taking e.g.\ \( D = \{ 0 \} \cup \{
1+2^{-n} \mid n \in \omega \} \) and applying Corollary~\ref{isomuniversal},
we get the desired result.
\end{proof}

Obviously, Corollary~\ref{cor:mainquest} implies that for every class \(
\mathcal{A} \) of ultrametric Polish spaces containing all discrete ones, the
isometry relation on \( \mathcal{A} \) is Borel bireducible with countable
graph isomorphism: this includes the class of \(\sigma\)-compact ultrametric
Polish spaces and the class of countable ultrametric Polish spaces (which are
both Borel $\boldsymbol \Pi^1_1$-complete, see \cite[Proposition
2.5]{lmrnew}).

Observe moreover that a proper analytic equivalence relation on a standard
Borel space (and, consequently, any equivalence relation which is Borel
bireducible with countable graph isomorphism) cannot be Borel reducible to isomorphism
on the collection \( \mathrm{WF} \) of well-founded trees on \(\omega\).
Indeed, \( \mathrm{WF} \) is a \( \boldsymbol{\Pi}^1_1 \)-complete class, and
the map associating to any tree in \( \mathrm{WF} \) its rank is actually a
\( \boldsymbol{\Pi}^1_1 \)-rank. Since the range of any hypothetical Borel
reduction as above would be an analytic subset of \( \mathrm{WF} \),
by~\cite[Theorem 35.23]{Kechris1995} it would be a subset of $ \mathcal
T_{\alpha }$ (see beginning of Section \ref{sect:jump}) for some $\alpha
<\omega_1$. This is impossible since the isomorphism relation on $ \mathcal
T_{\alpha }$ is Borel. So the second lower bound found by Kechris and Gao
(namely, isomorphism on \( \mathrm{WF} \)) is not sharp, even though they
noticed that it is absolutely $\boldsymbol \Delta^1_2$ bireducible with countable graph
isomorphism.

On the other hand we still do not know whether the first lower bound is
sharp, which by our results means whether  isomorphism between trees on
\(\omega\) with countably many infinite branches is Borel bireducible with
countable graph isomorphism (see Question~\ref{questionlowerbound}). This
class too is $ \boldsymbol \Pi^1_1$-complete, and the isomorphism relation on
it is again absolutely $\boldsymbol \Delta^1_2$ bireducible with countable
graph isomorphism. Recall also that, as observed in \cite[Chapter
8]{Gao2003}, this relation is the same as isometry between countable closed
subsets of \( \pre{\omega}{\omega} \).

\medskip

Now we deal with \( \U^\star_D \) for an ill-founded $D\in \mathcal D $. Fix
a decreasing sequence $(r_n)_{n\in\omega }$ in $D$ converging to some $r\geq
0$, and an element $ \bar r \in D$ bigger than every $r_n$. Set $D'=\{ r_n\mid
n\in\omega\}\cup\{ 0\} $. We define a Borel map \( f \colon \U_{D'} \to \U_D
\) as follows. Given \( U \in \U_{D'} \), let \( U^* = (U^*,d_{U^*}) \) be
obtained by gluing \( U \) with the canonical ultrametric Polish space \( U(D
\setminus \{r_0\}) \) (see Definition~\ref{def:U(D)}): this is done by
setting for each \( x \in U \) and \( r' \in D \setminus \{r_0\} \)
\[
d_{U^*}(x,r') = \max \{ \bar{r}, r' \}.
\]
Notice that $U^*$ realizes all distances in $D$, except possibly $r_0$.

We claim that \( U^* \) can be identified in a Borel-in-\( U \) way with an
element \( f(U) \) of  \( \U_D \subseteq F(\mathbb{U}) \). Fix an isometric
embedding $\rho$ of $(\mathbb{U}^\U_D)^*$ in $\mathbb U$. By Lemma
\ref{lemma:2form} we can assume $U \in F(\mathbb{U}^\U_D)$, and hence also \(
U^* \subseteq (\mathbb{U}^\U_D)^* \). Thus we can let $f(U)$ to be the image
of $U^*$ under $\rho$.

\begin{theorem} \label{propertyoff}
The map $f$ reduces $\isom_{D'}$ to $\isom_D$.
\end{theorem}

\begin{proof}
Let $U_0,U_1\in \U_{D'}$. If $\varphi \colon U_0\to U_1$ is an isometry, then
$\protect{{{\varphi} \cup {\id}} \colon U_0^*\to U_1^*}$ is an isometry,
where $ \id $ is the identity on $U(D\setminus\{ r_0\} )$. Conversely, let
$\psi \colon U_0^*\to U_1^*$ be an isometry. Then either $\psi (U_0)\subseteq
U_1$ or $\psi (U_0)\subseteq U(D\setminus\{ r_0\} )$: indeed, any two points
in $U_0$ are at most $r_0$ apart, while any point in $U_1$ has distance at
least $ \bar r $ from any point in $U(D\setminus\{ r_0\} )$. For a similar
reason, in the former case we have $\psi (U(D\setminus\{ r_0\} ))\subseteq
U(D\setminus\{ r_0\} )$, so that actually $\psi (U_0)=U_1$. In the latter,
there are some subcases to consider.

If $U_0$ is not a singleton, then $\psi(U_0) \subseteq  U(D \setminus
\{r_0\}) \cap [0, \bar r) =  U((D \setminus
\{r_0\}) \cap [0, \bar r))$, because \( D(U_0) \subseteq D' \subseteq D \cap [0, \bar r) \) while any \( y \in U(D \setminus
\{r_0\}) \setminus [0, \bar r) \) realizes distances \( \geq \bar r \).
Since any $y \in U((D \setminus \{r_0\}) \cap [0,\bar r)) \setminus
\psi (U_0)$ would be less than $\bar r$ apart from any point of $\psi (U_0)$,
while its preimage has distance $\bar r$ from any point of $U_0$, we conclude
that actually $\psi(U_0) = U((D \setminus \{r_0\}) \cap [0, \bar r)) = U(D'
\setminus \{r_0\})$. Since $\psi$ is a bijection it must then be $\psi (U(D'
\setminus \{r_0\})) = U_1$, so $U_0 \isom U(D' \setminus \{r_0\}) \isom U_1$
as desired. If $U_1$ is not a singleton, we can apply the previous argument
exchanging the roles of $U_0$ and $U_1$ and using $\psi^{-1}$ in place of
$\psi$. The remaining case is when $U_0$ and $U_1$ are both singletons, and
hence clearly isometric.
\end{proof}

Let $\mathcal{A} = \{ U \in \U_{D'} \mid \forall x \in U\, \exists y \in U\,
d_{\mathbb{U}} (x,y)=r_0\}$ and notice that $\mathcal{A}$ is closed under
isometries. Moreover, arguing as in Proposition \ref{prop:standardBorel}, it
is easy to see that $\mathcal{A}$ is Borel.

\begin{lemma} \label{restrictionofisom}
The restriction of $ \isom $ to $\mathcal{A}$ is Borel bireducible with
countable graph isomorphism.
\end{lemma}
\begin{proof}
If $r>0$ observe that in the proof of Theorem \ref{bigtheorem} for $D'$ as
the set of distances, the range of the reduction is contained in
$\mathcal{A}$, since $r_0$ is the distance between the root of the tree and
any other point.

If $r=0$ we are going to apply the proof of \cite[Theorem 4.4]{Gao2003} to
$D'$ as set of distances. This proof shows that the map \( T \mapsto [T] \)
is a Borel reduction between isomorphism on the set $ \mathcal{P T}$ of
nonempty pruned subtrees of $ \pre{<\omega }{\omega } $ and isometry on $
\U_{D'}$. If $ \mathcal{ PT}'$ is the set of pruned trees containing all
finite sequences all of whose entries equal $0$ and no other sequence with a
null entry, then isomorphism on $ \mathcal {PT}'$ is still Borel bireducible
with countable graph isomorphism: this is witnessed by the reduction $T\in \mathcal
{PT}\to T'= \{ s+1 \mid s \in T \} \cup \{ 0^{(n)} \mid n \in \omega \} \in
\mathcal {PT}'$, where \( s+1 = (s(i)+1)_{i < \leng(s)} \) and \( 0^{(n)} \)
is the sequence of length \( n \) constantly equal to \( 0 \). Now notice
that if \( T \in \mathcal{PT}' \) then \( [T] \in \mathcal{A} \).
\end{proof}

\begin{corollary} \label{isomuniversalstar}
Let $D\in \mathcal D $ be ill-founded. Then the relation of isometry on \(
\U_D^\star \) is Borel bireducible with countable graph isomorphism.
\end{corollary}

\begin{proof}
The relation \( \isom_D^\star \) is Borel reducible to countable graph isomorphism by Proposition~\ref{prop:knownisom} and \( \U^\star_D \subseteq \U_D \).
For the other direction,
notice that if $U\in \mathcal{A}$, then $U^*$ realizes all distances in $D$,
including $r_0$, so $U^*\in \U_D^{\star }$. Now apply Lemma
\ref{restrictionofisom} and Theorem \ref{propertyoff}.
\end{proof}

\subsection{Well-founded sets of distances} \label{isomwf}
Let us now consider the case of a well-ordered \( D  \in \mathcal{D}\). By
Lemma~\ref{lemma:basic}, if \( D \) is well-ordered, then up to classwise
Borel isomorphism the relations \( \isom_D \), \( \isom_D^* \),
\({\sqsubseteq_D} \), and \({\sqsubseteq^\star_D} \) do not depend really on
$D$ but only on its order type. Thus for each \( 1 \leq \alpha < \omega_1 \)
we can fix some \( D_\alpha \in \mathcal{D} \) with order type \( \alpha \)
and let \( ( r_\beta )_{ \beta < \alpha } \) be an increasing enumeration of
\( D_\alpha \), so that \( r_0 = 0 \).

To simplify the notation we use $\U_{\alpha}$ and $\U^\star_{\alpha}$ in
place of $\U_{D_\alpha}$ and $\U^\star_{D_\alpha}$. Similarly, the symbols \(
\isom_{\alpha } \), \( \isom_{\alpha }^* \), \( \sqsubseteq_\alpha \), and \(
\sqsubseteq^\star_\alpha \) will abbreviate the corresponding symbols with
subscript $D_{\alpha }$.

Notice that \( D_1 = \{ r_0 \} =  \{ 0 \} \) and every \( X \in \U^\star_1 =
\U_1 \) is a singleton, while if \( \alpha > 1 \) then every space in \(
\U^\star_{\alpha} \) has at least two points because the distance \( r_1>0 \)
must be realized. Moreover, by Lemma~\ref{lemma:basic} we have
\begin{equation}\label{eq:increasing}
{\isom_\alpha} \leq_B {\isom_\beta} \text{ and } {\sqsubseteq_\alpha} \leq_B {\sqsubseteq_\beta} \text{ for every } 1 \leq \alpha \leq \beta < \omega_1.
\end{equation}
The next lemma shows in particular that~\eqref{eq:increasing} remains true if we replace all relations with their counterparts with superscript \( \star \).

\begin{lemma} \label{lemmaequivalence}
Let \( 1\leq\alpha < \omega_1 \). Then ${ \isom_{\alpha }} \sim_B
{\isom_{\alpha }^*}$ and \( {\sqsubseteq_\alpha} \sim_B
{\sqsubseteq^\star_\alpha} \).
\end{lemma}

\begin{proof}
Since \( \U_1= \U^\star_1 \), it can be assumed \( \alpha\geq 2 \). Since \(
\U^\star_{\alpha} \subseteq \U_{\alpha} \), we need only to prove ${
\isom_{\alpha }} \leq_B {\isom_{\alpha }^*}$ and \( {\sqsubseteq_{\alpha }}
\leq_B {\sqsubseteq^\star_{\alpha }} \).

We first prove the assertion for $\alpha =\beta +1$ a successor ordinal. It
can be assumed $D_{\beta +1}=D_{\beta }\cup\{r_{\beta }\} $. Given \( X \in
\U_{\beta+1} \), let \( X'  \in \U^\star_{\beta+1} \) be the space obtained
by taking the disjoint union of \( X \) and the space \( U(D_\beta) \) from
Definition~\ref{def:U(D)}, and setting \( d_{X'}(x,r_\xi) = r_\beta \) for
every  \(x \in X \) and \( \xi < \beta \). Arguing as in the discussion
before Theorem \ref{propertyoff} one can show that the map \( X \mapsto X' \)
is Borel: we claim that it is the desired reduction.

Let \( X,Y \in \U_{\beta +1} \) and \( \varphi \colon X \to Y \) be a witness
of \( X \sqsubseteq_{\beta +1}Y \): then \( {\varphi} \cup {\id} \) witnesses \(
X' \sqsubseteq^\star_{\beta +1}Y' \). If moreover $\varphi$ is an isometry,
then ${\varphi} \cup {\id} $ is an isometry as well.

Conversely, let \( \psi \) be an isometric embedding of \( X' \) into \( Y'
\), and let \( X_0  = \psi^{-1}(U(D_\beta)) \). Notice that \(
d_{X'}(x_0,x_1) < r_\beta \) for every \( x_0,x_1 \in X_0 \) since \(
\psi(x_0),\psi(x_1) \in U(D_\beta) \). Hence, by construction of \( X' \),
either \( X_0 \cap X = \emptyset \) (i.e.\ \( X_0 \subseteq U(D_\beta) \)),
or else \( X_0 \subseteq X \). In the former case \( \psi \restriction X \)
is an isometric embedding of \( X \) into \( Y \). If moreover $\psi$ is
surjective (i.e.\ an isometry), then $\psi \restriction X$ is an isometry onto $Y$: notice indeed
that $\psi(U(D_{\beta }))$ cannot intersect both $Y$ and $U(D_{\beta })$,
since any two points of $U(D_{\beta })$ are less than $r_{\beta }$ apart.

If instead $X_0 \subseteq X$, we claim that \( \varphi =  \psi \restriction
(X \setminus X_0) \cup (\psi \circ \psi) \restriction X_0 \) witnesses \( X
\sqsubseteq_{\beta+1} Y \). Notice that \( \varphi \) is well-defined by the
definition of \( X_0 \) and the fact that \( U(D_\beta) \subseteq X' \).
First we check that the range of \( \varphi \) is contained in \( Y \). If \(
x \in X \setminus X_0 \) then trivially \( \varphi(x) = \psi(x) \in Y \) by
definition of \( X_0 \). If \( x \in X_0 \subseteq X\) then \( \psi(x) \in
U(D_\beta) \subseteq X' \). Since \( x \in X \), we have that \( d_{X'}(x ,
\psi(x))  = r_\beta \), whence also \( d_{Y'}( \psi(x) , \psi(\psi(x)))  =
r_\beta \). Since any two points in \( U(D_\beta) \) are \( < r_\beta \)
apart in $Y'$ and \( \psi(x) \in U(D_\beta) \), it follows that \( \varphi(x)
= \psi(\psi(x)) \in Y \). Finally, we check that \( \varphi \) preserves
distances. It is clearly enough to show that for \( x \in X \setminus X_0 \)
and \( x' \in X_0 \) we have \( d_X(x,x') = d_Y(\varphi(x),\varphi(x')) \).
Since \( \psi(x) \in Y \) and \( \psi(x') \in U(D_\beta) \), we have \(
d_{Y'}(\psi(x),\psi(x')) = r_\beta \), whence \( d_X(x,x') = d_{X'}(x,x') =
d_{Y'}(\psi(x),\psi(x')) = r_\beta \). Since \( \psi(x') \in U(D_\beta)
\subseteq X'\) and \( x \in X \), \( d_{X'}(x, \psi(x')) = r_\beta \),
therefore \( d_{Y'}(\psi(x), \psi(\psi(x'))) = r_\beta \). Since \( \psi(x) =
\varphi(x) \) and \( \psi(\psi(x')) = \varphi(x') \) by definition of \(
\varphi \), we have \( d_Y(\varphi(x),\varphi(x')) =
d_{Y'}(\varphi(x),\varphi(x')) = r_\beta = d_X(x,x') \), as required. If
moreover  $\psi$ is surjective (i.e.\ an isometry), by case hypothesis $\psi(U(D_{\beta
}))\subseteq Y$, and we also have that $\psi(X_0)=U(D_{\beta })$. So if $y\in
Y\setminus \psi(X\setminus X_0)$, which means $y\in \psi(U(D_{\beta }))$, it
follows that $y=\varphi(x)$ for some $x\in X_0$. This implies that $\varphi$
is surjective as well, hence it is an isometry between \( X \) and \( Y \).

Now we prove the result for \( \alpha \geq \omega \). Given \( X  \in
\U_\alpha \), let \( X' \in \U^\star_\alpha \) be the ultrametric space with
domain \( X\times\alpha \) and distance function \( d_{X'} \) defined by
setting, for distinct $(x,\xi), (y,\xi') \in X'$ such that $d_X(x,y)=r_\eta$,
\begin{equation}
d_{X'}((x,\xi ), (y,\xi' )) = \max \{ r_\xi, r_{\xi'}, r_{1+\eta} \}. \label{eq:definitionofdistance}
\end{equation}
(\( 1+ \eta < \alpha \) because \( \alpha \) is infinite). Recalling that
$r_0=0$ we then have in particular that
\begin{itemize}
\item for all distinct \( x,y \in X \), if \( d_X(x,y) = r_\eta \) then \(
    d_{X'}((x,0),(y,0)) = r_{1+\eta} \);
\item for all \( x \in X \) and \( \xi < \alpha \), \( d_{X'}((x,0),(x,\xi
    )) = r_{\xi} \).
\end{itemize}

We claim that the Borel\footnote{The map is Borel again by the argument
before Theorem \ref{propertyoff}. This applies to all reductions in the
subsequent proofs, so from this point on we will not recall it explicitly.}
map \( X \mapsto X' \) reduces $ \isom_{\alpha }$ to $ \isom_{\alpha }^\star$ and
\( \sqsubseteq_\alpha \) to \( \sqsubseteq^\star_\alpha \).

Let \( X,Y \in \U_\alpha \). First assume that \( \varphi \colon X \to Y \)
is an isometric embedding of \( X \) into \( Y \). Define \( \psi \colon X'
\to Y' \) by setting \( \psi(x,\xi ) = (\varphi(x),\xi ) \) for every \( x
\in X \) and \( \xi < \alpha \). Then \( \psi \) is distance preserving.
Indeed, fix \( x,y \in X \). If \( d_X(x,y) = r_\eta \) with \( 0 < \eta <
\alpha \), then \( d_Y(\varphi(x),\varphi(y)) = r_\eta \) and hence \(
d_{X'}((x,0),(y,0)) = r_{1+ \eta} = d_{Y'}((\varphi(x),0),(\varphi(y),0)) \).
Therefore, by definition of $\psi$ and using~\eqref{eq:definitionofdistance}
we get for arbitrary \( \xi,\xi' < \alpha \) such that $(x,\xi) \neq
(y,\xi')$
\[
\begin{split}
d_{X'}((x,\xi ),(y,\xi' )) & = \max \{ r_\xi, r_{\xi'}, d_{X'}((x,0),(y,0))  \} \\
& = \max \{ r_\xi, r_{\xi'}, d_{Y'}((\varphi(x),0),(\varphi(y),0))  \} \\
& = d_{Y'}((\varphi(x),\xi ),(\varphi(y),\xi')) \\
& =d_{Y'}(\psi(x,\xi ),\psi(y,\xi')).
\end{split}
\]
If, in addition, $\varphi$ is surjective (i.e.\ an isometry), then $\psi$ is surjective as well, and hence it witnesses \( X' \isom Y' \).

Assume now that \( \psi \) is an isometric embedding of \( X' \) into \( Y'
\). Notice that for every \( x \in X \) there is \( y \in Y \) such that
either \( \psi(x,0) =(y,0) \) or \( \psi(x,0) =(y,1) \). This is because \(
(x,0) \) realizes all distances in \( X' \) (i.e.\ for every \( \xi < \alpha
\) there is \( x' \in X' \) with \( d_{X'}((x,0),x') = r_\xi \)), hence \(
\psi(x,0) \) must have the same property relative to \( Y' \), which implies
that \( \psi(x,0) \) is of the prescribed form because by~\eqref{eq:definitionofdistance} points of the form \(
(y,\xi ) \)
 realize only the distances \( r_{\xi'} \) with \( \xi' = 0
\) or \( \xi' \geq \xi \). Let \( \varphi(x) \) be such \( y \). Notice that
the function \( \varphi \colon X \to Y \) is injective, as if \( x \neq x'
\in X \) and \( \varphi(x) = \varphi(x') \) then either \( \psi(x,0) =(y,0)
\) and \( \psi(x',0) =(y,1) \), or \( \psi(x,0) =(y,1) \) and \( \psi(x',0)
=(y,0) \) by injectivity of \( \psi \): in both cases \(
d_{Y'}(\psi(x,0),\psi(x',0)) = r_1 < d_{X'}((x,0),(x',0)) \), contradicting
the fact that \( \psi \) is distance preserving. We claim that \( \varphi \)
is also distance preserving (hence an isometric embedding). Let \( x,x' \in X
\) be distinct points. Then \( d_{Y'}((\varphi(x),i),(\varphi(x'),j)) > r_1
\) for every \( i,j \in \{ 0,1\} \) by injectivity of \( \varphi \)
and~\eqref{eq:definitionofdistance}, whence \( d_{Y'}(\psi(x,0),\psi(x',0)) =
d_{Y'}((\varphi(x),0),(\varphi(x'),0)) \) by~\eqref{eq:basicultrametric}.
Since \( \psi \) preserves distances, \( d_{X'}((x,0),(x',0)) =
d_{Y'}((\varphi(x),0),(\varphi(x'),0)) \), and hence \( d_X(x,x') =
d_Y(\varphi(x),\varphi(x')) \) by definition of \( d_{X'} \) and \( d_{Y'} \)
on \( X' \) and \( Y' \), respectively. If $\psi$ is moreover assumed to be
surjective (i.e.\ an isometry), then we can argue as above and show that for every $y\in Y$ there
exist $x\in X$ and $i\in\{ 0,1\}$ such that $\psi(x,i)=(y,0)$. If $i=0$, this
means $\varphi(x)=y$; if $i=1$, then $\psi(x,0)=(y,1)$ because \( (x,0) \) and \( (y,1) \) are the unique points at
distance \( r_1 \) from, respectively, \( (x,1) \) and \( (y,0) \), so again
$\varphi(x)=y$. It follows that $\varphi$ is surjective too, and hence it witnesses \( X \isom Y \).
\end{proof}

\begin{remark}
Although \( {\sqsubseteq_\alpha} \sim_B {\sqsubseteq^\star_\alpha} \) the two
quasi-orders are combinatorially different: $\sqsubseteq_\alpha$ has always a
minimum, while ${\sqsubseteq^{\star }_\alpha}$ has at least two
$\sqsubseteq$-incomparable minimal elements when $\alpha \geq 4$.
\end{remark}

\begin{corollary} \label{cor:isomequivalent}
For every \( D \in \mathcal{D} \), \( { \isom_D} \sim_B
{ \isom^\star_D} \).
\end{corollary}

\begin{proof}
If \( D \) is ill-founded then both \( \isom_D \) and \( \isom^\star_D \) are
Borel bireducible with countable graph isomorphism by
Corollaries~\ref{isomuniversal}~and~\ref{isomuniversalstar}, whence \(
{ \isom_D } \sim_B { \isom^\star_D} \). If instead \( D \) is
well-founded, then the result follows from Lemma~\ref{lemmaequivalence}.
\end{proof}

The following points out that if we restrict ourselves to Borel reducibility
we can weaken  the hypotheses of Lemma \ref{lemma:basic}.

\begin{corollary} \label{improvedembedding}
Let $D,D'\in \mathcal D $. If there is an order-preserving embedding from $D$
into $D'$, then ${ \isom_D} \leq_B {\isom_{D'}}$ and ${ \isom^\star_D} \leq_B
{\isom^\star_{D'}}$.
\end{corollary}

\begin{proof}
If $D'$ is well-ordered, then $D$ is well-ordered as well, and the hypothesis
of Lemma \ref{lemma:basic} is satisfied. Otherwise, $ \isom_{D'}$ and $
\isom^\star_{D'}$ are Borel bireducible with countable graph isomorphism by Corollaries
\ref{isomuniversal} and \ref{isomuniversalstar}, and hence the result follows from
Proposition~\ref{prop:knownisom} and \( \U^\star_D \subseteq \U_D \).
\end{proof}

\begin{lemma} \label{lemmalimit}
Let \( \lambda < \omega_1 \) be limit. If $\isom_{\beta }$ is Borel for every
$\beta <\lambda $, then $ \isom_{\lambda }$ is Borel. If \( \sqsubseteq_\beta
\) is Borel for every \( \beta < \lambda \), then \( \sqsubseteq_\lambda \)
is Borel as well.
\end{lemma}

\begin{proof}
Recall the definition of \( C^x_r \) from Notation~\ref{not:C^x_r}, and
notice in particular that if \( U \in \U_{\lambda} \) then for every \( \beta
< \lambda \) and \( x \in U \)  realizing the distance $r_{\beta }$, we have
\( C^x_{r_\beta}(U) \in \U_{{\beta+1}} \). Given \( U \in \U_{\lambda} \),
one has \( U   =\{ \psi_n(U)\}_{n\in\omega } \) (where the Borel functions \(
\psi_n \) are as in Notation~\ref{notation:psi_n}).

By Lemma~\ref{lem:limit}, it follows that for every \( U, U' \in
\U_{\lambda} \) we have $U \isom_{\lambda } U'$ if and only if
\begin{align*}
\exists n \in \omega \, \forall \beta < \lambda \, (( & C_{r_\beta}^{\psi_0(U)}(U)\neq\emptyset \iff C_{r_\beta}^{\psi_n(U')}(U')\neq\emptyset)\wedge \\
 & (C_{r_\beta}^{\psi_0(U)}(U)\neq\emptyset\implies
C^{\psi_0(U)}_{r_\beta}(U) \isom_{\beta+1} C^{\psi_n(U')}_{r_\beta}(U'))).
\end{align*}
Similarly, we have $U \sqsubseteq_\lambda U'$ if and only if
\[
\exists n \in \omega \, \forall \beta < \lambda \,
(C_{r_\beta}^{\psi_0(U)}(U)\neq\emptyset\implies
C^{\psi_0(U)}_{r_\beta}(U) \sqsubseteq_{\beta+1} C^{\psi_n(U')}_{r_\beta}(U')).
\]
Since each $ \isom_{\beta+1}$ (respectively, \( \sqsubseteq_{\beta+1} \)) is
assumed to be Borel it follows that $ \isom_{\lambda }$ (respectively, \(
\sqsubseteq_\lambda \)) is Borel as well.
\end{proof}

\begin{theorem} \label{lem:successor}
For all \( 1<\alpha < \omega_1 \), \( {\isom_{\alpha+1}} \sim_B
E_{{\isom_\alpha}^\inj} \) and \( {\sqsubseteq_{\alpha+1}} \sim_B
{{\sqsubseteq_\alpha}^{\inj}} \).
\end{theorem}

\begin{proof}
First we show ${{\sqsubseteq_\alpha}^{\inj}} \leq_B {\sqsubseteq_{\alpha
+1}}$ and $E_{{\isom_\alpha}^\inj}\leq_B { \isom_{\alpha+1}}$. To each
sequence of spaces $\vec{X} = (X_n)_{n \in \omega} \in
\pre{\omega}{(\U_{\alpha })}$, associate the disjoint union $\Phi
(\vec{X})=\bigcup_{n\in\omega }X_n$, where $d_{\Phi(\vec{X})}(x,x')=r_{\alpha
}$ whenever $x\in X_i$ and $x'\in X_j$ with $i\neq j$.

If $\vec{X} = (X_n)_{n \in \omega}$ and \( \vec{Y} = (Y_n)_{n \in \omega} \)
are such that \( \vec{X} \mathrel{{\sqsubseteq_\alpha}^{\inj}}  \vec{Y}$, let
$f \colon \omega \to \omega $ be an injection with $\varphi_n \colon X_n\to
Y_{f(n)}$ an isometric embedding for all $n\in\omega $. Then
$\bigcup_{n\in\omega }\varphi_n \colon \Phi (\vec{X}) \to \Phi (\vec{Y})$ is
an isometric embedding. If moreover $ \vec X \mathrel{E_{{ \isom_{\alpha }}^{
\inj}}} \vec Y $, then by Lemma \ref{ebij} this is witnessed by a bijection
$f \colon \omega\to\omega $ such that for every $n\in\omega $ there is an isometry
$\varphi_n \colon X_n\to Y_{f(n)}$. Thus $\bigcup_{n\in\omega }\varphi_n \colon \Phi ( \vec
X )\to\Phi ( \vec Y )$ is an isometry.

Conversely, let $\varphi \colon \Phi (\vec{X})\to\Phi (\vec{Y})$ be an
isometric embedding. Then for each $n\in\omega $ there is $f(n)\in\omega $
such that $\varphi (X_n)\subseteq Y_{f(n)}$: indeed two points in $X_n$ are
closer to each other than $r_{\alpha }$, which is the distance between any
two points belonging to distinct $Y_m$. For a similar reason, the function $f
\colon \omega \to \omega$ is injective, as points $x\in X_n$, $x'\in X_{n'}$
for $n\neq n'$ have distance $r_{\alpha }$, which is bigger than any distance
between points from a single $Y_m$. The restriction of $\varphi $ to each
$X_n$ is then an isometric embedding into $Y_{f(n)}$, so $\vec{X}
\mathrel{{\sqsubseteq_\alpha}^{\inj}} \vec{Y}$. If moreover $\varphi $ is
an isometry onto $\Phi ( \vec Y )$, then $f$ is surjective too and \( \varphi \restriction X_n \colon X_n \to Y_{f(n)} \) is an isometry.
Hence \( f \)
witnesses $ \vec X \mathrel{E_{{ \isom_{\alpha }}^{\inj }}} \vec Y $.

We now show that ${\sqsubseteq_{\alpha +1}} \leq_B {{\sqsubseteq_{\alpha
}}^{\inj}}$ and ${ \isom_{\alpha +1}}\leq_BE_{{ \isom_{\alpha }}^{ \inj }}$.
For each $X\in \U_{{\alpha +1}}$ and $x,x'\in X$, define $x \mathrel{E_X}
x'\Leftrightarrow d_X(x,x')<r_{\alpha }$. This is an equivalence relation on
$X$ which, by countability of $X$, has countably many equivalence classes.
Let $( X_n )_{ n \in \omega }$ be an enumeration of the \( E_X \)-equivalence
classes where, if there are finitely many of them, say $m$, then
$X_n=\emptyset $ for $n\geq m$.

Suppose first that $\alpha =\beta +1$ is a successor ordinal. Let
$X'_n=X_n\cup\{*_n\}$, where $*_n$ are new elements and extend the distance
on $X_n$ by setting $d_{X'_n}(x,*_n)=r_{\beta }$ for all $x\in X_n$ (here we
use \(\alpha >1\)). We consider the map $X \mapsto ( X'_n )_{ n \in \omega
}$.

If $\varphi \colon X\to Y$ is an isometric embedding, then it injectively maps
$E_X$-classes into $E_Y$-classes,
thus inducing a well-defined \emph{partial}
injective map \( f \)  satisfying $\varphi (X_n)\subseteq Y_{f(n)}$ for every \( n
\) such that \( X_n \neq \emptyset \). By its definition, if the domain of
$f$ is not already all of $ \omega $, such domain is finite and thus $f$ can
be arbitrarily extended to an injection of $ \omega$ into itself which will
still be denoted by $f$. Finally define the isometric embedding $\varphi_n
\colon X'_n\to Y'_{f(n)}$ by extending $\varphi \restriction {X_n}$ with the
addition of the condition $\varphi_n(*_n)=*_{f(n)}$:
this shows that \( f \) witnesses \( (X'_n)_{n \in \omega} \mathrel{{\sqsubseteq_\alpha}^\inj} (Y'_n)_{n \in \omega} \).
If $\varphi $  is in
addition an isometry onto $Y$, then the number of $E_X$-classes equals the
number of $E_Y$-classes, say $m \leq \omega$, the function $f$ is a permutation of $m$, and
every $\varphi \restriction X_n$ is an isometry onto $Y_{f(n)}$. Extending $f$ to a
bijection $f \colon \omega\to\omega $, the functions $\varphi_n$ defined above turn
out to be isometries as well, so that \( f \) witnesses $(X'_n)_{n\in\omega }E_{{ \isom_{\alpha
}}^{ \inj }}(Y'_n)_{n\in\omega }$.

Conversely, assume $(X'_n)_{n \in \omega} \mathrel{{\sqsubseteq_{\alpha
}}^{\inj}} (Y'_n)_{n\in \omega}$ so that there are isometric embeddings
$\varphi_n \colon X'_n\to Y'_{f(n)}$ for some injective $f \colon \omega \to
\omega $. If $\varphi_n(X_n)\subseteq Y_{f(n)}$, then let
$\varphi'_n=\varphi_n \restriction{X_n}$. Otherwise, let $a\in X_n$ with
$\varphi_n(a)=*_{f(n)}$. In this case, for $x\in X_n$ define
\[
\varphi'_n(x)=
\begin{cases}
\varphi_n(x) & \text{if } x \neq a \\
\varphi_n(*_n) & \text{if } x=a.
\end{cases}
\]
It is not hard to check that, whenever $X_n\neq\emptyset $, \( \varphi'_n
\colon X_n \to Y_{f(n)} \) is still an isometric embedding. So $\varphi =
\bigcup_n\varphi'_n$ is an isometric embedding $X\to Y$. If the stronger
condition $(X'_n)_{n\in\omega } \mathrel{E_{{ \isom_{\alpha }}^{ \inj }}}
(Y'_n)_{n\in\omega }$ holds then, by Lemma \ref{ebij}, $f$ can be assumed to
be a bijection with each $\varphi_n$ an isometry onto $Y'_{f(n)}$. So the
functions $\varphi'_n$, as well as $\varphi $, turn out to be isometries.

Let now $\alpha $ be infinite. Define a new distance $d'_{X_n}$ on $X_n$ by
letting $d'_{X_n}(x,x')=r_{1+ \beta}$ for distinct \( x,x' \), where
$r_{\beta} = d(x,x')$. Let $X'_n=X_n\cup\{ x^* \mid x\in X_n \}$ be a
disjoint union of two copies of $X_n$, where the ultrametric $d'_{X_n}$ is
extended by declaring $d_{X'_n}(x,x^*)=r_1$ for all $x\in X_n$. (Thus
$d_{X'_n}(x^*,y) = d_{X'_n} (x,y^*) = d'_{X_n} (x,y)$ when $x \neq y$.) Finally, let $X''_n$ be
$X'_n$ if this is nonempty, and consist of exactly one point if
$X_n=X'_n=\emptyset $. Notice that \( X''_n \) is a singleton if and only if
\( X_n = \emptyset \), and that if \( X_n \neq \emptyset \) then for every \(
x,y \in X''_n \) at distance \( r_1 \) we have that either \( x = y^* \) or
\( y = x^* \). We claim that the map \( X \mapsto (X''_n)_{n \in \omega} \)
is the desired reduction.

Assume $\varphi \colon X\to Y$ is an isometric embedding. Again, this induces
a partial injection $f$ such that $\varphi (X_n)\subseteq Y_{f(n)}$ for all
\( n \in \omega \) for which \( X_n \neq \emptyset \),  which can then be
extended arbitrarily to an injection $f \colon \omega\to\omega $. If
$X_n\neq\emptyset $, then $\varphi_n = \varphi \restriction X_n$ can be
extended to an isometric embedding $\varphi'_n \colon X''_n\to Y''_{f(n)}$ by
letting $\varphi'_n(x^*)=(\varphi_n(x))^*$ for all \( x\in X_n \). If instead
$X_n=\emptyset $, then any function $\varphi'_n \colon  X''_n\to Y''_{f(n)}$
is an isometric embedding. Thus the injection \( f \) and the isometric
embeddings $\varphi'_n$ witness $(X''_n)_{n \in \omega}
\mathrel{{\sqsubseteq_{\alpha }}^{\inj}}(Y''_n)_{n \in \omega}$. In case $f$
is an isometry, then exactly as before one gets a bijection $f$ and
isometries $\varphi' _n$ witnessing $(X''_n)_{n\in\omega } \mathrel{E_{{
\isom_{\alpha }}^{ \inj }}} (Y''_n)_{n\in\omega }$.

Conversely, let $f \colon \omega\to\omega$ be an injection and $\varphi_n
\colon X''_n\to Y''_{f(n)}$ be isometric embeddings witnessing $(X''_n)_{n
\in \omega} \mathrel{{\sqsubseteq_{\alpha }}^{\inj}} (Y''_n)_{n \in \omega}$.
This implies $Y_{f(n)} \neq \emptyset$ whenever $X_n \neq \emptyset$.
Moreover, for such an $n$ and any $x\in X_n$, we have
$\{\varphi_n(x),\varphi_n(x^*)\} =\{y,y^*\}$ for some $y\in Y_{f(n)}$. This
defines an isometric embedding $\psi_n \colon X_n\to Y_{f(n)} \colon x\mapsto
y$. Then $\bigcup_n\psi_n \colon X\to Y$ is an isometric embedding. As
before, the stronger hypothesis $(X''_n)_{n\in\omega } \mathrel{E_{{ \isom_{
\alpha }}^{ \inj }}} (Y''_n)_{n\in\omega }$  allows us to
assume that $f$ is bijective (use Lemma \ref{ebij}) with each $\varphi_n$ an isometry. Consequently,
every $\psi_n$ (for \( n \) such that \( X_n \neq \emptyset \)) is an isometry, implying that $\bigcup_n\psi_n$ is an isometry
as well.
\end{proof}

\begin{remark}
Although ${\sqsubseteq_{\alpha+1}}$ and ${\sqsubseteq_\alpha}^\inj$ are Borel
bireducible, their quotient orders are not isomorphic. They both have a bottom
element $\bot$: however, in the former every immediate successor $x$ of $\bot$ has a
unique immediate successor whose predecessors are exactly $\bot$ and $x$,
while in the latter this property fails.
\end{remark}

\begin{theorem} \label{labeltobeadded}
The relations \( \isom_\alpha \), for $1\leq\alpha <\omega_1$, form a
strictly increasing chain of Borel equivalence relations which is cofinal
among Borel equivalence relations classifiable by countable structures.
\end{theorem}

\begin{proof}
Inductively, all relations $ \isom_{\alpha }$ are Borel by applying Theorem
\ref{lem:successor} and Proposition~\ref{friedmaninj} in the successor step,
and Lemma \ref{lemmalimit} in the limit step. That $ \isom_{\alpha }$ is a
strictly increasing sequence follows from~\eqref{eq:increasing} and Theorem \ref{lem:successor}
together with Proposition \ref{friedmaninj}.

To verify cofinality, by Proposition \ref{folklore} it is enough to show that
the isomorphism relation on $ \mathcal T_{\alpha }$ Borel reduces to
${\isom_{1+\alpha }}$. Given any well-founded tree \( T \), let \( \rk_T \)
be the function assigning to each node of $T$ its rank in $T$, which is an
ordinal smaller than the rank of \( T \). Given $T \in \mathcal{T}_\alpha$,
let
\[
T'=T\cup\{ s {}^\smallfrown{}  0\mid s \text{ is a terminal node in } T\} .
\]
Then for every \( t \in T \) we have \( \rk_{T'}(t) = 1 + \rk_T(t) < 1+ \alpha \).
Now define an ultrametric
 $d_{T'}$ on $T'$ by letting, for distinct $s$ and $t$, $d_{T'}(s,t) =
r_{\rk_{T'}(u)}$, where $u$ is the longest common initial
segment of $s$ and $t$. Notice that
\begin{equation} \label{terminalnode}
d_{T'}(t,u)=r_1 \iff \exists s \text{ terminal node of } T \text{ such that } \{ t,u\} =\{ s,s {}^\smallfrown{}  0\} .
\end{equation}
On the other hand, if $s$ is not a terminal node of $T$, then the least
non-null distance realized by $s$ is $r_{\rk_{T'}(s)} = r_{1+\rk_T(s)}>r_1$.

By construction, if $T_0$ and $T_1$ are isomorphic trees then
$(T_0',d_{T_0'})$ and $(T_1',d_{T_1'})$ are isometric.

Conversely, let $\varphi $ be an isometry between the spaces $(T_0', d_{T'_0})$ and
$(T_1', d_{T'_1})$. Let $\psi \colon T_0\to T_1$ be defined by letting $\psi (t)= \varphi (t)$
if $t$ is not terminal in $T_0$, and
\[
\psi (s)= \text{the unique terminal node }
s' \text{ of } T_1 \text{ such that } \varphi (s)\in\{ s',s' {}^\smallfrown{}  0\}
\]
if $s$ is terminal in $T_0$ (the existence of such an \( s' \) is guaranteed
by~\eqref{terminalnode}).  Notice that the map \( \psi \) still preserves
distances, and that it is a bijection between \( T_0 \) and \( T_1 \). We
claim that $\psi $ is an isomorphism, i.e.\ that it preserves the
tree-ordering relation (namely, inclusion) on \( T_0 \) and \( T_1 \). First
notice that $\rk_{T'_0}(s)=\rk_{T'_1}(\psi (s))$ for all $s\in T_0$: for
terminal nodes this is obvious from the definition of \( \psi \), while if
$s$ is not terminal it follows from the above observation about the smallest
non-null distance realized by \( s \). Let now \( s \subsetneq t \in T_0 \).
Then \( d_{T'_0}(s,t) = \rk_{T'_0}(s) = \rk_{T'_1}(\psi(s)) \). Let \( u \)
be the longest common subsequence of \( \psi(s) \) and \( \psi(t) \), so that
by definition \( d_{T'_1}(\psi(s), \psi(t)) = \rk_{T'_1}(u) \). Then by the
above computations we get \( \rk_{T'_1}(u) = \rk_{T'_1}(\psi(s)) \) because
\( \psi \) preserves distances. But since \( u \subseteq \psi(s) \), this
implies \( u = \psi(s) \), whence \( \psi(s) = u \subsetneq \psi(t) \). A
similar argument shows that if \( \psi(s) \subsetneq \psi(t) \) then \( s
\subsetneq t \), hence we are done.
\end{proof}

In \cite[Section 6]{GaoShao2011}, Gao and Shao provide a faithful translation
of ultrametric Polish spaces as a certain kind of combinatorial objects. More
precisely, given a countable \( D \in \mathcal{D} \) they define the class \(
\mathcal{T}_D \) of \emph{\( D \)-trees} and show that there are two maps
$\Phi \colon \U_D \to \mathcal{T}_D$ and $\Psi \colon \mathcal{T}_D \to \U_D$ such that the
following holds:
\begin{itemize}
  \item for all $U \in \U_D$ and $T \in \mathcal{T}_D$, $\Psi(\Phi(U))$ is
      isometric to $U$ and $\Phi(\Psi(T))$ is isomorphic to $T$ (as
      structures in the appropriate language);
  \item $\Phi$ reduces isometry (respectively, isometric embeddability) on $\U_D$
      to isomorphism (respectively, embeddability) on $\mathcal{T}_D$;
  \item $\Psi$ reduces isomorphism (respectively, embeddability) on
      $\mathcal{T}_D$ to isometry (respectively, isometric embeddability) on
      $\U_D$.
\end{itemize}
Therefore our results on $\cong_D$ (in particular Corollary
\ref{isomuniversal} and Theorem \ref{labeltobeadded}) can be translated as
results about isomorphism between $D$-trees, and an analogous observation will
apply to our results on $\sqsubseteq_D$ (in particular Theorems
\ref{theorillfounded1} and \ref{th:sumup}). Notice also that when $D$ has
order type $\omega$ the class $\mathcal{T}_D$ essentially coincides with the
class of \emph{reverse trees} considered by Clemens in \cite{ClemensPreprint}, where
he observes that isomorphism between reverse trees is another lower bound for
the complexity of isometry on discrete ultrametric Polish spaces. Since
Theorem \ref{labeltobeadded} shows in particular that such a relation is
Borel, also this lower bound is not sharp.

\section{The complexity of isometric embeddability}

This section is devoted to Question~\ref{questemb}, answering it in most cases.

\subsection{Ill-founded sets of distances} \label{embif}

In this subsection we complete the study, started in Section~\ref{isomif}, of
the complexity with respect to Borel reducibility of \( \sqsubseteq_D \) and
\( \sqsubseteq^\star_D \) when \( D \in \mathcal{D} \) is not well-founded.
In fact we will show that they are both invariantly universal (hence also
complete for analytic quasi-orders), improving in this way
Proposition~\ref{prop:louveaurosendal} and Theorem~\ref{theorconvergingto0}.

Building on previous work in~\cite{frimot}, in~\cite[Section 3]{cammarmot}
the authors constructed a class \( \mathbb{G} \) of countable rooted
combinatorial trees with the following properties (see~\cite[Corollaries 3.2
and 3.4]{cammarmot}):

\begin{fact} \label{fact:G}
\begin{enumerate}[(1)]
\item \( \mathbb{G} \) is a Borel subset of $\Mod_{ \L } = \pre{\omega
    \times \omega}{2}$ (the Polish space of countable \( \L \)-structures,
    where  \( \L \) is the language consisting of one binary relation
    symbol), so it is a standard Borel space;
\item on $\mathbb{G}$ equality and isomorphism coincide;
\item each \( G \in \mathbb{G} \) is a graph which is infinite and rigid,
    i.e.\ its unique automorphism is the identity function.
\item every \( G \in \mathbb{G} \) has a distinguished vertex \( r(G) \)
    (the \emph{root} of \( G \)) such that every embedding of \( G \) into
    \( H \in \mathbb{G} \) must send \( r(G) \) to \( r(H) \).
\end{enumerate}
\end{fact}

Thus $ \mathbb G $ can be actually considered as a Borel subset of $ \RCT $
(see the beginning of section \ref{isomif}), so that each \( G \in \mathbb{G}
\) can be identified in a Borel-in-\( G \) way with a graph \( G' \) on  \(
\pre{< \omega}{\omega} \) with \( r(G) \) identified with the empty sequence
\( \emptyset \in G' \). The class of graphs \( \mathbb{G} \) was the key tool
for the method developed in~\cite{cammarmot} for proving invariant
universality of a given pair \( (S,E) \); this is summarized in the following
theorem.

\begin{theorem}\label{theorsaturation}
Let $S$ be an analytic quasi-order on a standard Borel space $Z$, and $E
\subseteq S$  be an analytic equivalence relation on the same space. Denote
by \( \sqsubseteq_\mathbb{G} \) and \( \cong_\mathbb{G} \) the restrictions
to \( \mathbb{G} \) of the embeddability and isomorphism relations,
respectively. Suppose there exists a Borel function $f \colon \mathbb{G} \to
Z$ which simultaneously witnesses ${\sqsubseteq_\mathbb{G}} \leq_B S$ and
${=_\mathbb{G}} \leq_B E$ (which is the same as ${\cong_\mathbb{G}} \leq_B
E$). Furthermore, let $Y$ be a Polish group, $a \colon Y \times W \to W$ a
Borel action of $Y$ on a standard Borel space $W$, and $g \colon Z \to W$
witness $E \leq_B E_a$. Consider the map $\Sigma \colon \mathbb{G} \to F(Y)$
which assigns to $G \in \mathbb{G}$ the stabilizer of $(g \circ f)(G)$ with
respect to $a$, i.e.\
\[
\Sigma(G) = \{y \in Y \mid a(y, (g \circ f)(G)) = (g \circ f)(G)\}.
\]

If $\Sigma$ is Borel, then the pair $(S,E)$ is invariantly universal.
\end{theorem}

\begin{theorem} \label{theorillfounded1}
If \( D \) is ill-founded, then \( \sqsubseteq_D \) is invariantly
universal (hence also complete for analytic quasi-orders).
\end{theorem}

\begin{proof}
By Theorem~\ref{theorconvergingto0}, we can assume that $0$ is not a limit
point of $D$ and that \( D \) contains a strictly decreasing sequence of
distances \( ( r_n )_{ n \in \omega } \) converging to some \( r \neq 0 \).
In particular, the spaces in $\U_D$ are discrete. To prove the theorem, we
will apply Theorem~\ref{theorsaturation} with \( S = {\sqsubseteq_D}
\restriction Z \) and \( E = {\isom_D} \restriction Z  \) for \( Z \) a
suitable Borel subset of \( \U_D \) closed under isometry; then the invariant
universality of \( (S,E) \) implies that \( ({\sqsubseteq_D}, {\isom_D}) \)
is invariantly universal as well, as desired.

Set
\[
Z = \{ U \in \U_D \mid U \text{ is infinite} \},
\]
so that \( Z \)  satisfies the required conditions. Recall the definition of
the Borel function $\theta \colon \RCT \to \U_D$ from Definition \ref{theta}, and
notice that \( \theta ( \mathbb G )\subseteq Z \) because each \( G \in
\mathbb{G} \) is infinite. By Theorem \ref{bigtheorem}, the map  \( \theta \)
simultaneously reduces  \( =_ \mathbb{G} \) to \( \isom_D \) and \(
\sqsubseteq_\mathbb{G} \) to \( \sqsubseteq_D \). In particular, this already
shows that \( \sqsubseteq_D \) is complete.

The next step to apply Theorem~\ref{theorsaturation} is to reduce the
relation \( E = {\isom_D} \restriction Z \) to a Borel group action. Consider
the countable language \( \Lambda = \{ R_q  \mid q \in D \} \), where each \(
R_q \) is a binary predicate. To each \( U \in Z \) associate the \( \Lambda
\)-structure \( S(U) \) on \(\omega\) by letting
\begin{equation*}
R^{S(U)}_q(i,j) \iff d_U(\psi_i(U),\psi_j(U)) = q,
 \end{equation*}
where the functions \( \psi_n \) are as in Notation~\ref{notation:psi_n}.
Recall also that since every \( U \in Z \) is discrete and infinite, \(
(\psi_n(U))_{n \in \omega} \) is an enumeration without repetitions of all
points in \( U \). It is clear that if \( U_0,U_1 \in Z \) and \( \varphi \)
is an isometry between them, then the map \( f_\varphi \colon \omega \to
\omega \) sending \( i \) to the unique \( j \) such that \(
\varphi(\psi_i(U_0)) = \psi_j(U_1) \) is an isomorphism between \( S(U_0) \)
and \( S(U_1) \). Conversely, if \( f \) is an isomorphism between \( S(U_0)
\) and \( S(U_1) \) then the map \( \varphi_f \colon U_0 \to U_1 \) sending
\( \psi_i(U_0) \) to \( \psi_{f(i)}(U_1) \) is an isometry. Therefore
\[
g \colon Z \to \Mod_\Lambda \colon U \mapsto S(U)
\]
is a Borel map reducing the relation of isometry on \( Z \) to the
isomorphism relation on \( \Mod_\Lambda \), which is the orbit equivalence
induced by the continuous logic action \( j_\Lambda \) of the Polish group
\( S_\infty \) of all permutations of \(\omega\) on the Polish space \(
\Mod_\Lambda \) of \( \Lambda \)-models with domain \(\omega\).

Therefore, to apply Theorem~\ref{theorsaturation} it remains only to show
that the map \( \Sigma \) assigning to each \( G \in \mathbb{G} \) the
stabilizer of \( S(\theta(G)) \) with respect to \( j_\Lambda \) is a Borel
map.
To see this, notice that by the above discussion every automorphism \( h
\) of \( S(\theta(G)) \) can be identified with an isometry of $\theta(G)$
and hence with an isometry \( \varphi_h \) of \( U_G \) into itself.
Conversely, every isometry \( \varphi \) of \( U_G \) into itself can be
identified with an automorphism \( h_\varphi \) of \( S(\theta(G)) \).
Since
each \( G \in \mathbb{G} \) is rigid, then there are no distinct terminal
nodes in \( G' \) sharing the same immediate predecessor.
From this fact and
the analysis performed in the proof of Theorem~\ref{bigtheorem}, it follows that if \( \varphi \) is an isometry
of \( U_G \) into itself and \( s \in G' \) is such that \( \varphi(s) \neq s
\), then one of \( s, \varphi(s) \) is a terminal node of \( G' \) and the
other is its immediate predecessor, and moreover \( \varphi(\varphi(s)) = s
\).
Therefore \( \varphi \) can differ from the identity function only in
that it may switch some terminal nodes of \( G' \) with their immediate
predecessor.
Recall the definition of $\rho $ from the beginning of Section \ref{isomif} and notice that the maps sending \( G \in \mathbb{G} \) to,
respectively,
\[
T_G = \{ n \in \omega \mid (\rho^{-1} \circ \psi_n \circ \theta)(G) \text{ is terminal in }G' \}
\]
and
\begin{align*}
  P_G = \{ (n,m) \in \omega \times \omega \mid & (\rho^{-1} \circ \psi_n \circ \theta)(G) \text{ is an immediate predecessor} \\
  & \quad \text{of } (\rho^{-1} \circ \psi_m \circ \theta)(G) \text{ in } G' \}
\end{align*}
are Borel.
We then get for \( h \in S_\infty \) that \( h \in \Sigma(G) \)
if and only if for all $n \in \omega$ such that $h(n) \neq n$
\[
n \in T_G \land (h(n),n) \in P_G \quad \text{or} \quad h(n) \in T_G \land (n,h(n)) \in P_G.
 \]
Thus \( \Sigma \) is Borel and therefore \( (S,E) \) (hence also \( (\sqsubseteq_D, \isom_D ) \))  is invariantly
universal by Theorem~\ref{theorsaturation}.
\end{proof}

Theorem~\ref{theorillfounded1} can be further improved to the following.

\begin{theorem} \label{theorillfounded2}
If \( D \) is ill-founded, then \( \sqsubseteq^\star_D \) is invariantly
universal (hence also complete for analytic quasi-orders).
\end{theorem}

\begin{proof}
As done before Theorem \ref{propertyoff}, fix a decreasing sequence
$(r_n)_{n\in\omega }$ in $D$ converging to some $r\geq 0$, and an element $
\bar r \in D$ bigger than every $r_n$. Set $D'=\{ r_n\mid n\in\omega\}\cup\{
0\} $. Set again, as before Lemma \ref{restrictionofisom}, $ \mathcal A = \{
U \in \U_{D'} \mid \forall x \in U\, \exists y \in U\, d_{\mathbb{U}}
(x,y)=r_0\}$ and recall that it is Borel and closed under isometries. Notice
that the proofs of Theorem~\ref{theorconvergingto0} (if \( r = 0 \)) and
Theorem~\ref{theorillfounded1} (if \( r > 0 \)) actually show that \(
{\sqsubseteq} \restriction \mathcal A \) is invariantly universal (in the
latter case because the range of $\theta$ is contained in $\mathcal A$).
Thus, by the observation after Definition \ref{def:cB}, it is enough to show
that ${\sqsubseteq} \restriction \mathcal A $ classwise Borel embeds into \(
\sqsubseteq^\star_D \), where both quasi-orders are paired with isometry.
This amounts to show the existence of a class \( \mathcal B \subseteq
\U^\star_D \) and two maps \( f \colon \mathcal A \to \mathcal B \) and \( g
\colon \mathcal B \to \mathcal A \) such that:
\begin{enumerate}[(i)]
\item \( \mathcal B \) is a Borel subset of \( \U^\star_D \) (hence also of \(
    F(\mathbb{U}) \)) closed under isometries;
\item \( f \) and \( g \) are both Borel;
\item for each \( U \in \mathcal A \), \( g(f(U)) \) is isometric to \( U \);
\item for each \( F \in \mathcal B \), \( f(g(F)) \) is isometric to \( F \);
\item \( f \) simultaneously reduces \( {\isom} \restriction \mathcal{A} \) to \( {\isom} \restriction \mathcal{B} \) and
 \( {\sqsubseteq} \restriction \mathcal{A} \) to \( {\sqsubseteq} \restriction \mathcal{B} \).
\end{enumerate}
Notice that from (iv) and (v) it also follows that \( g \) is a reduction of
\( {\isom} \restriction \mathcal{B} \) to \( {\isom} \restriction \mathcal{A} \) and of
 \( {\sqsubseteq} \restriction \mathcal{B} \) to \( {\sqsubseteq} \restriction \mathcal{A} \).

A function $f$ reducing  $ \isom_{D'}$ to  $ \isom_D$ has been defined before Theorem \ref{propertyoff}.
As our $f$ here we take the restriction of that one to $ \mathcal A $, recalling that $f( \mathcal A )\subseteq \U_D^{\star }$.
We next define \( \mathcal B \) as the closure (in \( F(\mathbb{U}) \)) under isometry of \( f( \mathcal A) \): since \( \U^\star_D \) itself is closed under
isometry, \( \mathcal B \subseteq \U^\star_D \). To see that \( \mathcal B \) is Borel, first
notice that, by definition of $ \mathcal A $, the elements in \( U \subseteq U^* \) can
be characterized by the fact that they realize the distance \( r_0 \).
Moreover, the clopen ball \( B(x,\bar{r}) \) of \( U^* \) centered in \( x
\in U \) always coincide with \( U \).
These observations lead to the
following Borel description of \( \mathcal B \). A space \( F \in \U^\star_D \) is in
\( \mathcal B \) if and only if
\begin{enumerate}[$\bullet$]
\item for some/every \( n \) in the set
\[
C_F = \{ n \in \omega \mid \exists k \in \omega\, d_{\mathbb{U}}(\psi_k(F),\psi_n(F)) = r_0 \} ,
\]
the subspace \( B(\psi_n(F), \bar{r}) \cap F \) of \( F \) is in \( \mathcal A \);
\item for some/any \( n \in C_F \), the subspace \( F \setminus
    B(\psi_n(F), \bar{r}) \) of \( F \) is isometric to \( U(D \setminus
    \{r_0\}) \);\footnote{This is a Borel condition because the isometry
    relation on \(F(\mathbb{U}) \), being reducible to an orbit equivalence
    relation, has only Borel classes (see e.g.~\cite[Theorem
    15.14]{Kechris1995}).}
\item for all \( n \in C_F \) and all \( m \in \omega \setminus C_F \), \(
    d_{\mathbb{U}}(\psi_n(F),\psi_m(F)) \geq \bar{r} \);
\item for all $r \in D$ with \( r > \bar{r} \), all \( n \in C_F \) and all
    \( m \in \omega \setminus C_F \),
\begin{align*}
d_{\mathbb{U}}(\psi_n(F),\psi_m(F)) = r \iff & \forall k \in \omega \setminus (C_F \cup \{ m \})\,
 (d_{\mathbb{U}}(\psi_k(F),\psi_m(F)) \geq r) \\
& \land \exists k \in \omega \setminus C_F \, (d_{\mathbb{U}}(\psi_k(F),\psi_m(F)) = r).
\end{align*}
\end{enumerate}
Finally, we define \( g \colon \mathcal B \to \mathcal A \) by setting \( g(F) = B(\psi_n(F),
\bar{r}) \cap F \) for some/any \( n \in C_F \), so that \( g \) is clearly a
Borel function.

We already showed that (i) and (ii) are satisfied.
Conditions (iii) and (iv)
easily follow from the definition of \( \mathcal B \) and the observations preceding
it (together with the definitions of \( f \) and \( g \), of course).
As for (v), it has already been proved in Theorem \ref{propertyoff} that $U_0 \isom U_1$ if and only if $U^*_0 \isom U^*_1$.
So we
only need to prove that for all \( U_0, U_1 \in \mathcal A \)
\[
U_0 \sqsubseteq U_1 \iff U^*_0 \sqsubseteq U^*_1.
 \]
Since, as already observed, \( U_i \subseteq U^*_i  \) (for \( i = 0,1 \))
can be characterized as the collection of all points in \( U^*_i \) which
realize \( r_0 \), any isometric embedding \( \varphi^* \) of \( U^*_0 \)
into \( U^*_1 \) must map points in \( U_0 \) to points in \( U_1 \).
Hence if \( \varphi^* \colon U^*_0
\to U^*_1 \) is an isometric embedding, then so
is \( \varphi = \varphi^* \restriction U_0 \colon U_0 \to U_1\). Conversely,
notice that the subspace \( U^*_i \setminus U_i = U(D \setminus \{ r_0 \}) \)
of \( U^*_i \) (for \( i = 0,1 \)) does not depend on \( U_i \), and that the
same is true for each distance between a point in \( U^*_i \setminus U_i \)
and an arbitrary point in \( U^*_i \). Therefore, if \( \varphi \colon U_0
\to U_1 \) is an isometric embedding, then so is
the map \( \varphi^* = \varphi \cup (\id \restriction (D \setminus \{ r_0
\}))\colon U^*_0 \to U^*_1 \). This concludes our proof.
\end{proof}

\begin{corollary} \label{cor:equivalent}
For every \( D \in \mathcal{D} \), \( {\sqsubseteq_D} \sim_B
{\sqsubseteq^\star_D} \).
\end{corollary}

\begin{proof}
If \( D \) is ill-founded then both \( \sqsubseteq_D \) and \(
\sqsubseteq^\star_D \) are complete by Theorems~\ref{theorillfounded1} and
\ref{theorillfounded2}, whence \( {\sqsubseteq_D } \sim_B
{\sqsubseteq^\star_D} \). If instead \( D \) is well-founded, then the result
is contained in Lemma~\ref{lemmaequivalence}.
\end{proof}

\subsection{Well-founded sets of distances} \label{sect:wf}

In this section we investigate more closely the relations of isometric
embeddability between Polish ultrametric spaces using a well-ordered set of
distances. By~\eqref{eq:increasing} we have \( {\sqsubseteq_\alpha}
\leq_B {\sqsubseteq_\beta} \) for every \( 1 \leq \alpha \leq \beta < \omega_1 \),
whence also \( {\sqsubseteq^\star_\alpha} \leq_B {\sqsubseteq^\star_\beta} \)
by Lemma \ref{lemmaequivalence}. Moreover, by Proposition \ref{propbqo} both
\( \sqsubseteq_n \) and \( \sqsubseteq^\star_n \) are bqo's. We are now ready
to provide an alternative proof of this fact which uses only \(
\AC_\omega(\RR) \). The unique result we will need from~\cite{NashWil1965} is
Corollary 28A,  whose proof does not use any strong form of \( \AC \). If \(
S \) is a quasi-order on \( X \), we let \( S^\# \) be the quasi-order on \(
\pow(X) \) obtained by setting for every \( A, B \subseteq X \)
\[
A \mathrel{S}^\# B \iff \exists f \colon A \to B \text{ injective such that } \forall x \in A\, ( x \mathrel{S} f(x)).
 \]

\begin{proof}[Proof of Proposition~\ref{propbqo}]
Notice first that if a quasi-order (Borel) embeds in a bqo then it is a bqo.
Then, Lemmas~\ref{lemma:basic} and~\ref{lemmaequivalence} imply that it is
enough to show by induction on \( n \geq 1 \) that the relation \(
\sqsubseteq_n \) is a bqo. The case \( n = 1 \) is clear. Assume now that \(
\sqsubseteq_n \) is a bqo: since by~\cite[Corollary 28A]{NashWil1965} the
quasi-order \( {\sqsubseteq_n}^\# \) on \( \pow(\U_{n}) \) is a bqo as well,
it is enough to prove that \( \sqsubseteq_{n+1} \) is reducible to \(
{\sqsubseteq_n}^\# \), i.e.\ that there is a function \( f \colon \U_{{n+1}}
\to \pow(\U_{n}) \) such that \( X \sqsubseteq_{n+1} Y \iff f(X)
\mathrel{{\sqsubseteq_n}^\#} f(Y) \) for every \( X,Y \in \U_{n+1} \). Since by
Theorem~\ref{lem:successor} we already know that \( \sqsubseteq_{n+1} \) is
(Borel) reducible to \( {\sqsubseteq_n}^{\inj} \), we just need to show that
\( {\sqsubseteq_n}^{\inj} \) reduces to \( {\sqsubseteq_n}^\# \). To see
this, consider the map \( g \) replacing each sequence \( ( X_i )_{ i \in
\omega } \) of spaces in \( \U_{n} \) with a set \( \{ X'_i \mid i \in \omega
\} \subseteq \U_{n} \) such that each \( X'_i \) is isometric to \( X_i \)
and \( X'_i \neq X'_j \) for distinct \( i,j \in \omega \): then \( g \)
clearly reduces \( {\sqsubseteq_n}^{\inj} \) to \( {\sqsubseteq_n}^\# \), as
required.
\end{proof}

We now show that if
instead \(\alpha\) is infinite, the situation is quite different.

\begin{proposition} \label{prop:alphainfinite=>nonwqo}
Let \( D \in \mathcal{D} \). Then the partial order \( (\pow(D \setminus \{ 0
\}), {\subseteq}) \) embeds into \( \sqsubseteq_D \) (hence also into \(
\sqsubseteq^\star_D \) by Corollary~\ref{cor:equivalent}).
\end{proposition}

\begin{proof}
The map \( i \) sending \( X \subseteq D \setminus \{ 0 \} \) into \( U(X
\cup \{ 0 \}) \) is the desired embedding.
Indeed, if \( X \subseteq Y \) then the identity function witnesses \( i(X)
\sqsubseteq i(Y) \). Conversely, assume that there is an isometric embedding
\( f \) of \( i(X) \) into \( i(Y) \). Let \( r \in X \), so that in
particular \( r > 0 \). Then the distance $r$ is realized in $i(X)$, and
therefore in $i(Y)$ too; so $r \in Y$ by definition of $i(Y)$. It follows \(
X \subseteq Y \), as desired.
\end{proof}

\begin{corollary} \label{cor:alphainfinite=>nonwqo}
Let \( \omega \leq \alpha < \omega_1 \). Then \( (\pow(\omega), \subseteq)
\leq_B {\sqsubseteq_\alpha} \), and hence also \( (\pow(\omega), \subseteq)
\leq_B {\sqsubseteq^\star_\alpha} \). In particular, both \(
\sqsubseteq_{\alpha } \) and \( \sqsubseteq^\star_{\alpha } \) contain
infinite antichains and infinite decreasing chains.
\end{corollary}

\begin{proof}
Identify \( D_\alpha \setminus \{ 0 \} \) with \(\omega\), so that \( \pow
(D_\alpha \setminus \{ 0 \}) \) is naturally identified with \( \pow(\omega)
\). Then the map \( i \) defined in the proof of
Proposition~\ref{prop:alphainfinite=>nonwqo} is Borel, and hence witnesses \(
(\pow(\omega), \subseteq) \leq_B {\sqsubseteq_\alpha} \).
\end{proof}

\begin{theorem} \label{corincreasing}
For every \( 1 \leq \alpha < \omega_1 \), if \( \sqsubseteq_\alpha \) is
Borel then \( {\sqsubseteq_\alpha} <_B {\sqsubseteq_{\alpha+1}} \).
\end{theorem}

\begin{proof}
We already observed that \( {\sqsubseteq_\alpha} \leq_B {\sqsubseteq_\beta}
\) for every \( 1 \leq \alpha \leq \beta < \omega_1 \). If \( \alpha = 1 \),
then the quotient order of \( \sqsubseteq_1 \) consists of one point, while
if \( \alpha = 2 \), then the quotient order of \( \sqsubseteq_2 \) is a
linear order of order type \( \omega+1 \) (the different \( E_{\sqsubseteq_2}
\) classes correspond to the possible cardinalities of the spaces in \(
\U_{2} \)). Therefore \( {\sqsubseteq_2} \nleq_B {\sqsubseteq_1} \), and, by
Corollary~\ref{inj_jumps} and Theorem~\ref{lem:successor}, \(
{\sqsubseteq_2} <_B {\sqsubseteq_2}^{\inj} \sim_B
 {\sqsubseteq_3}   \).

For the case \( 3 \leq \alpha < \omega_1 \), by Theorem~\ref{lem:successor} it
is enough to show that \( {\sqsubseteq_\alpha} <_B
{{\sqsubseteq_\alpha}^{\inj}} \) (assuming that \( \sqsubseteq_\alpha \) is
Borel). To see this, consider the spaces \( U_0 , U_1 \in \U_\alpha \) with
domain \( \{ x,y \} \) and distances defined by \( d_{U_0}(x,y) = r_1 \) and
\( d_{U_1}(x,y) = r_2 \). Then clearly \( U_0 \not \sqsubseteq U_1 \), \( U_1
\not \sqsubseteq U_0 \), and \( \{ U \in \U_\alpha \mid U \sqsubseteq U_0,U_1
\} \) is  the \( E_{\sqsubseteq_\alpha} \)-equivalence class of the space
consisting of just one point. The desired result then follows from
Corollary~\ref{inj_jumps}.
\end{proof}

\begin{lemma} \label{lem:productlimit}
Let $(S_n)_{n \in \omega}$ be a sequence of analytic quasi-orders and
$\lambda $ an infinite countable ordinal. If for every \( n \in \omega \)
there exists \( \beta_n < \lambda \) such that \( S_n \leq_B {
\sqsubseteq_{\beta_n} } \), then \( \prod_{n \in \omega} S_n \leq_B
{\sqsubseteq_\lambda} \).
\end{lemma}

\begin{proof}
Using~\eqref{eq:increasing} and Lemma~\ref{lemmaequivalence}, it is enough to
show $\prod_{n\in\omega }\sqsubseteq^{\star
}_{\beta_n}{\leq_B}\sqsubseteq_{\lambda }$ with \( \beta_n > 1 \) for all \(
n \in \omega \). We distinguish two cases, depending on whether \( \lambda \)
is a successor or not.

Assume first that \( \lambda = \lambda' + 1 \), so that, in particular, \(
\lambda' \geq \omega \). Fix a bijection $\langle \cdot,\cdot \rangle
\colon\omega\times (\omega\setminus\{ 0\} )\to\omega\setminus\{ 0\} $,
increasing with respect to its second argument. Given a sequence of spaces \(
(X_n)_{n \in \omega}  \in \prod_{n \in \omega} \U_{{\beta_n}}^{\star } \),
let $X'$ be the space in \(\U_{{\lambda}} \) whose domain is the disjoint
union of the \( X_n \) and whose distance \( d_{X'} \) is defined by letting,
for distinct $x,y\in X'$,
\[
d_{X'}(x,y)=
\begin{cases}
r_{\lambda'} & \text{if $x \in X_n$ and $y\in X_m$ with $n \neq m$;} \\
r_\xi & \text{if $x,y \in X_n$ and $d_{X_n}(x,y) = r_\xi$ with $\xi \geq \omega$}; \\
r_{\langle n,i \rangle} & \text{if $x,y \in X_n$ and $d_{X_n}(x,y) = r_i$ with $i \in \omega$}.
\end{cases}
\]
Pick two sequences \( ( X_n )_{ n \in \omega } , ( Y_n )_{ n \in \omega } \in
\prod_{n \in \omega} \U_{{\beta_n}}^{\star } \). If $X_n \sqsubseteq Y_n$ for
every $n$ then, gluing together the embeddings, we get an isometric embedding
of $X'$ into $Y'$. Conversely, assume $\varphi \colon X'\to Y'$ is an
isometric embedding. Since \( \beta_n < \lambda = \lambda' + 1 \), the
distance \( r_{\lambda'} \) cannot be realized in \( X_n \) (as a subspace of
\( X' \)), and hence for every \( n \in \omega \) there is \( k(n) \in
\omega\) such that $\varphi (X_n)\subseteq Y_{k(n)}$. Moreover, since the
distances with finite positive index used in $X_n$ (as a subspace of $X'$)
are not used in any $Y_m$ (as a subspace of \( Y' \)) with $m\neq n$, we
easily get \( k(n) = n \), so that \( \varphi(X_n) \subseteq Y_n \) for every
\( n \in \omega \) (here we are using $X_n \in \U^\star_{{\beta_n}}$ and \(
\beta_n > 1 \)). It then easily follows that for all \( n \in \omega \) the
restriction of  \( \varphi \) to \( X_n \) witnesses $X_n \sqsubseteq^{\star
}_{\beta_n} Y_n $.

For the case \( \lambda \) limit the argument is similar. For each \( n \in
\omega \), let \( \lambda_n = \max \{ \beta_i \mid i \leq n \} < \lambda \).
Given \( (X_n)_{n \in \omega}  \in \prod_{n \in \omega} \U_{{\beta_n}}^{\star
} \), let $X'$ be the space in \(\U_{{\lambda}} \) whose domain is the
disjoint union of the \( X_n \) and whose distance \( d_{X'} \) is defined
by letting, for distinct $x,y\in X'$,
\[
d_{X'}(x,y)=
\begin{cases}
r_{2 \cdot \lambda_{\max \{ n,m \}}+1} & \text{if $x \in X_n$ and $y\in X_m$ with $n \neq m$}; \\
r_{2 \cdot \xi} & \text{if $x,y \in X_n$ and $d_{X_n}(x,y) = r_\xi$ with $\xi \geq \omega$};\\
r_{2 \cdot \langle n,i \rangle} & \text{if $x,y \in X_n$ and $d_{X_n}(x,y) = r_i$ with $i \in \omega$}.
\end{cases}
\]
Let \( ( X_n )_{ n \in \omega } , ( Y_n )_{ n \in \omega } \in \prod_{n \in
\omega} \U_{{\beta_n}}^{\star } \). If each $X_n$ embeds isometrically into
$Y_n$, then gluing together the embeddings we get an isometric embedding of
$X'$ into $Y'$. Conversely, assume $\varphi \colon X'\to Y'$ is an isometric
embedding. Since no distance with odd index can be realized within a single
\( X_n \) (as a subspace of \( X' \)), for every \( n \in \omega \) there is
\( k(n) \in \omega\) such that $\varphi (X_n)\subseteq Y_{k(n)}$, and arguing
as above we easily get \( k(n) = n \). Thus \( \varphi(X_n) \subseteq Y_n \)
for every \( n \in \omega \), whence $X_n  \sqsubseteq^{\star }_{\beta_n} Y_n
$ for every \( n \in \omega \).
\end{proof}

Recall the definition of the Rosendal's sequence \( P_\alpha \), \( \alpha <
\omega_1 \), given in Definition~\ref{defin:rosendalsequence}.

\begin{corollary}
For all \( \alpha < \omega_1 \),  \( P_{1+\alpha} \leq_B
{\sqsubseteq_{\omega+\alpha}} \).
\end{corollary}

\begin{proof}
By induction on \( \alpha < \omega_1 \), using
Corollary~\ref{cor:alphainfinite=>nonwqo} (together with the fact that \( P_1
\sim_B (\pow(\omega), \subseteq) \) by Proposition~\ref{prop:simple}(1)),
Theorem~\ref{lem:successor} (together with the fact that \( P_{1+\alpha+1} = P_{1+\alpha}^\cf \leq_B P_{1+\alpha}^\inj \)),
and Lemma~\ref{lem:productlimit}.
\end{proof}

\begin{theorem} \label{th:sumup}
Let \( 1 \leq \alpha < \omega_1 \). Then
\begin{enumerate}[(1)]
\item if \( \alpha \leq \omega \), then \( \sqsubseteq_\alpha \) is Borel;
\item if $\alpha > \omega$, \( \sqsubseteq_\alpha \) contains both upper
    and lower cones that are \( \boldsymbol{\Sigma}^1_1 \)-complete, and
    hence $\sqsubseteq_\alpha$ is analytic non-Borel;
\item all \( E_{\sqsubseteq_{\omega+1}} \)-equivalence classes are Borel,
    hence \( \sqsubseteq_{\omega+1} \) is not complete for analytic
    quasi-orders;
\item for all \( \alpha < \beta \leq \omega+2 \), \( {\sqsubseteq_\alpha}
    <_B {\sqsubseteq_\beta} \).
\end{enumerate}
\end{theorem}

\begin{proof}

(1) By Lemma~\ref{lemmalimit} it is enough to prove the result for \( \alpha
< \omega \), and this will be shown by induction on \( \alpha\geq 1 \). The
case \( \alpha = 1 \) is clear and the induction step is immediate using
Proposition~\ref{propbqo}, Corollary \ref{cor:injonwqo} and
Theorem~\ref{lem:successor}.

(2) By Corollary~\ref{cor:alphainfinite=>nonwqo} and Theorem~\ref{lem:successor}, we get \( {(\pow(\omega), \subseteq)^{\inj}}
\leq_B {\sqsubseteq_{\omega+1}} \), so that the claim is true for \( \alpha =
\omega+1 \) by Proposition~\ref{prop:P(omega)}. For an arbitrary \(\alpha >
\omega+1 \), use the fact that \( {\sqsubseteq_{\omega+1}} \leq_B
{\sqsubseteq_\alpha} \) by~\eqref{eq:increasing}.

(3) Since \( {\sqsubseteq_{\omega+1}} \sim_B {{\sqsubseteq_\omega}^{\inj}} \)
by Theorem~\ref{lem:successor} and \( \sqsubseteq_\omega \) is Borel by
part~(1), the result follows from Theorem~\ref{th:S^inj not complete}.

(4) Clearly \( {\sqsubseteq_\alpha} \leq_B
{\sqsubseteq_\beta} \) by~\eqref{eq:increasing}, so only the inequality \( {\sqsubseteq_\beta} \nleq_B
{\sqsubseteq_\alpha} \) needs to be proved. For \( \alpha \leq \omega \) this
follows from Theorem~\ref{corincreasing} and (1). For \( \alpha = \omega+1 \)
and \( \beta = \omega+2 \), first observe that if \( {\sqsubseteq_{\omega+2}}
\leq_B {\sqsubseteq_{\omega+1}} \), then any witness to this would also
witness \({ E_{\sqsubseteq_{\omega+2}}} \leq_B {E_{\sqsubseteq_{\omega+1}}}
\). But this is impossible because \( E_{\sqsubseteq_{\omega+1}} \) has only
Borel equivalence classes by part~(3), while \( E_{\sqsubseteq_{\omega+2}} \)
has a \( \boldsymbol{\Sigma}^1_1 \)-complete class by part~(2) applied with \( \alpha = \omega+1 \),
Lemma~\ref{lem:upperlowercones}(2), and the fact that \(
{\sqsubseteq_{\omega+2}} \sim_B {{\sqsubseteq_{\omega+1}}^{\inj}} \) by
Theorem~\ref{lem:successor}.
\end{proof}

\section{Open problems}

\subsection{Isometry}
In this paper we have given a fairly complete treatement of the relation of isometry on ultrametric Polish spaces; these are, in particular, zero-dimensional spaces.
However, one of the main questions asked by \cite{Gao2003} and still unanswered is the following:

\begin{question}
What is the complexity of isometry between zero-dimensional Polish metric
spaces?
\end{question}

Clemens \cite{clemensisometry} showed that this relation is strictly above countable
graph isomorphism. It is conjectured in \cite[Chapter 10]{Gao2003} that it is
Borel bireducible with any complete orbit equivalence relation.

The same question might be asked for any other topological dimension. For
infinite-dimensional spaces the isometry relation is Borel bireducible with
any complete orbit equivalence relation by the proofs of \cite[Theorem
1]{Gao2003} and \cite[Theorem 7]{clemensisometry}. For other dimensions the
problem is open, but it is easy to observe that if $\alpha \leq
\alpha'<\omega_1$ then the relation of isometry on spaces of dimension
$\alpha$ is Borel reducible to the same relation on spaces of dimension
$\alpha'$.

Corollary \ref{cor:mainquest} answers what appears to be the main question
from \cite[Chapter 8]{Gao2003}: the relations of isomorphism on discrete
Polish ultrametric spaces and on locally compact Polish ultrametric spaces
are both Borel bireducible to countable graph isomorphism. As discussed after Corollary
\ref{cor:mainquest}, among the lower bounds proposed in the literature for
the complexity of these relations the only one that could still be sharp is
isomorphism between trees on \(\omega\) with countably many infinite branches
or, equivalently,  isometry on countable closed subspaces of $ \pre{\omega
}{\omega } $.

\begin{question} \label{questionlowerbound}
Is the relation of isomorphism between trees on \(\omega\) with countably
many infinite branches Borel bireducible with countable graph isomorphism?
\end{question}

\subsection{Isometric embeddability}

The main problem left open by Theorem \ref{th:sumup} is the following:

\begin{question}\label{q:main}
Does there exist $\alpha \geq \omega+2$ such that $\sqsubseteq_{\alpha}$ is
complete/invariantly universal for analytic quasi-orders? In particular, what
about $\sqsubseteq_{\omega+2}$?
\end{question}

In the previous question one can equivalently replace $\sqsubseteq_{\alpha}$
with $\sqsubseteq^\star_{\alpha}$. For completeness this is immediate from
Lemma \ref{lemmaequivalence}. For invariant universality this follows from
the observation after Definition \ref{def:cB} and the fact that one can
sharpen the argument of the proof of Lemma \ref{lemmaequivalence} to show
that each of $\sqsubseteq_{\alpha}$ and $\sqsubseteq^\star_{\alpha}$
classwise Borel embeds into the other one.

To the best of our knowledge, the techniques to prove that an analytic
quasi-order $S$ is not complete are the following: show that $S$ is
combinatorially simple (e.g.\ a wqo); show that $S$ is topologically simple
(i.e.\ Borel); show that the equivalence relation $E_S$ is simple (e.g.\
almost all its equivalence classes are Borel). None of these apply to
$\sqsubseteq_{\omega+2}$. In fact, combinatorially already \(
\sqsubseteq_{\omega+1} \) (which is not complete for analytic quasi-orders)
is very complicated because it embeds all partial orders \( P \) of  size \(
\omega_1 \) (and thus, assuming the Continuum Hypothesis, the quotient order
of any quasi-order on a standard Borel space). To see this, use the fact that
by~\cite{Parovicenko1963} any such \( P \) can be embedded in the relation \(
(\pow(\omega), \subseteq^*) \) of inclusion modulo finite sets, together with
the chain of reducibilities \( {(\pow(\omega), \subseteq^*)} \leq_B
{(\pow(\omega), \subseteq)^{\cf}} \leq_B {\sqsubseteq_{\omega+1}} \). The
same argument also shows that it is not possible to answer negatively the
subsequent Question~\ref{q:powominjinj} using only combinatorial arguments.
Moreover, \( \sqsubseteq_{\omega+2} \) is not Borel by Theorem
\ref{th:sumup}(2), and its associated equivalence relation contains (many)
non-Borel classes, as shown in the proof of Theorem \ref{th:sumup}(4).

A possible way to obtain a positive answer to Question \ref{q:main} is to
answer positively the following question:

\begin{question}\label{q:powominjinj}
Is \( ((\pow(\omega), {\subseteq})^{\inj})^{\inj} \)  complete/invariantly
universal for analytic quasi-orders?
\end{question}

Recall that \((\pow(\omega), {\subseteq}) \leq_B {\sqsubseteq_\omega}\) by
Corollary \ref{cor:alphainfinite=>nonwqo}, and hence we have also \(
(\pow(\omega), {\subseteq})^{\inj} \leq_B {\sqsubseteq_{\omega+1}} \) and \(
((\pow(\omega), {\subseteq})^{\inj})^{\inj} \leq_B {\sqsubseteq_{\omega+2}}
\) by Theorem \ref{lem:successor}.

\begin{question}\label{q:powomega}
Is it true that \( {\sqsubseteq_{\omega+1}} \sim_B (\pow(\omega),
{\subseteq})^{\inj} \)? What about \( {\sqsubseteq_\omega} \sim_B
(\pow(\omega), {\subseteq})\)?
\end{question}

Notice that a positive answer to the second part of Question \ref{q:powomega}
implies a positive answer to the first part by Theorem \ref{lem:successor}. By
the same theorem, a positive answer to the latter would imply that Question
\ref{q:powominjinj} is equivalent to the completeness part of Question \ref{q:main} for $\alpha=
\omega+2$.

Notice also that an upper bound for the complexity of \( \sqsubseteq_\omega
\) is \( (\pow(\omega), {\subseteq})^{\cf} \) (whence also \(
{\sqsubseteq_{\omega+1}} \leq_B (( \pow(\omega), \subseteq)^{\cf})^{\inj}
\)). To see this, first observe that \( {\sqsubseteq_n} <_B (\pow(\omega),
{\subseteq}) \) (because \( \sqsubseteq_n \) is a bqo with only countably many \( E_{\sqsubseteq_n} \)-classes by
Corollary~\ref{corwqo=countable} and Proposition~\ref{propbqo}), and then use Lemma~\ref{lem:limit} to get \(
{\sqsubseteq_\omega} \leq_B (\pow(\omega), {\subseteq})^{\cf} \). However,
since \( (\pow(\omega), \subseteq) <_B (\pow(\omega), \subseteq)^{\cf} \) by
Lemma~\ref{lemmacf}, the exact complexity of \( \sqsubseteq_\omega \) remains
undetermined.

\subsection{Classwise Borel embeddability}

The only known tool for showing that a pair $(S,E)$ is invariantly universal
is Theorem~\ref{theorsaturation}. Its proof in \cite{cammarmot} actually
shows that the conclusion can be strengthened to: $(R,{=}) \sqsubseteq_{cB}
(S,E)$ for every analytic quasi-order $R$. This naturally leads to consider
the following notion: a pair $(S,E)$ as in Definition \ref{def:cB} is
\emph{$\sqsubseteq_{cB}$-complete} if $(R,F) \sqsubseteq_{cB} (S,E)$ for
every pair $(R,F)$. Notice that if a pair \( (S,E) \) is
$\sqsubseteq_{cB}$-complete then, in particular, both \( S \) and \( E \) are
\( \leq_B \)-complete for analytic quasi-orders and analytic equivalence
relations, respectively. Thus none of the invariantly universal pairs \( (S,E
) \) considered in this paper and in \cite{cammarmot} is
$\sqsubseteq_{cB}$-complete, because to apply our
Theorem~\ref{theorsaturation} we need that \( E \) be  Borel reducible to an
orbit equivalence relation (whence \( E \) cannot be \( \leq_B \)-complete).
However, using standard arguments it is not hard to construct a
$\sqsubseteq_{cB}$-complete pair. It is thus natural to ask the following:

\begin{question}
Do there exist \lq\lq natural\rq\rq\ examples of $\sqsubseteq_{cB}$-complete
pairs? In particular, for which of the invariantly universal quasi-orders \(
S \) considered in this paper and in \cite{cammarmot} we have that \( (S,E_S)
\) is also $\sqsubseteq_{cB}$-complete?
\end{question}

Of course this question is strongly related to~\cite[Question
6.4]{cammarmot}, which still remains wide open.


\begin{thebibliography}{CMMR13}

\bibitem[Bur79]{Burgess1979}
John~P. Burgess.
\newblock A reflection phenomenon in descriptive set theory.
\newblock {\em Fund. Math.}, 104(2):127--139, 1979.

\bibitem[CGK01]{Clemens2001}
John~D. Clemens, Su~Gao, and Alexander~S. Kechris.
\newblock Polish metric spaces: their classification and isometry groups.
\newblock {\em Bull. Symbolic Logic}, 7(3):361--375, 2001.

\bibitem[Cle07]{ClemensPreprint}
John~D. Clemens.
\newblock Isometry of {P}olish metric spaces with a fixed set of distances.
\newblock preprint downloaded on Feb. 4, 2015 from
  \url{http://wwwmath.uni-muenster.de/u/john.clemens/publications.html}, 2007.

\bibitem[Cle12]{clemensisometry}
John~D. Clemens.
\newblock Isometry of {P}olish metric spaces.
\newblock {\em Ann. Pure Appl. Logic}, 163(9):1196--1209, 2012.

\bibitem[CM07]{Camerlo2007}
Riccardo Camerlo and Alberto Marcone.
\newblock Coloring linear orders with {R}ado's partial order.
\newblock {\em MLQ Math. Log. Q.}, 53(3):301--305, 2007.

\bibitem[CMMR13]{cammarmot}
Riccardo Camerlo, Alberto Marcone, and Luca Motto~Ros.
\newblock Invariantly universal analytic quasi-orders.
\newblock {\em Trans. Amer. Math. Soc.}, 365(4):1901--1931, 2013.

\bibitem[FMR11]{frimot}
Sy-David Friedman and Luca Motto~Ros.
\newblock Analytic equivalence relations and bi-embeddability.
\newblock {\em J. Symbolic Logic}, 76(1):243--266, 2011.

\bibitem[Fri00]{friedman2000}
Harvey~M. Friedman.
\newblock Borel and {B}aire reducibility.
\newblock {\em Fund. Math.}, 164(1):61--69, 2000.

\bibitem[FS89]{Friedman1989}
Harvey Friedman and Lee Stanley.
\newblock A {B}orel reducibility theory for classes of countable structures.
\newblock {\em J. Symbolic Logic}, 54(3):894--914, 1989.

\bibitem[Gao09]{gaobook}
Su~Gao.
\newblock {\em Invariant descriptive set theory}, volume 293 of {\em Pure and
  Applied Mathematics (Boca Raton)}.
\newblock CRC Press, Boca Raton, FL, 2009.

\bibitem[GK03]{Gao2003}
Su~Gao and Alexander~S. Kechris.
\newblock On the classification of {P}olish metric spaces up to isometry.
\newblock {\em Mem. Amer. Math. Soc.}, 161(766):viii+78, 2003.

\bibitem[Gro99]{gromov1999}
Misha Gromov.
\newblock {\em Metric structures for {R}iemannian and non-{R}iemannian spaces},
  volume 152 of {\em Progress in Mathematics}.
\newblock Birkh\"auser Boston, Inc., Boston, MA, 1999.
\newblock Based on the 1981 French original [ MR0682063 (85e:53051)], With
  appendices by M. Katz, P. Pansu and S. Semmes, Translated from the French by
  Sean Michael Bates.

\bibitem[GS11]{GaoShao2011}
Su~Gao and Chuang Shao.
\newblock Polish ultrametric {U}rysohn spaces and their isometry groups.
\newblock {\em Topology Appl.}, 158(3):492--508, 2011.

\bibitem[HKL90]{HKL}
L.~A. Harrington, A.~S. Kechris, and A.~Louveau.
\newblock A {G}limm-{E}ffros dichotomy for {B}orel equivalence relations.
\newblock {\em J. Amer. Math. Soc.}, 3(4):903--928, 1990.

\bibitem[Kec95]{Kechris1995}
Alexander~S. Kechris.
\newblock {\em Classical descriptive set theory}, volume 156 of {\em Graduate
  Texts in Mathematics}.
\newblock Springer-Verlag, New York, 1995.

\bibitem[Lav71]{Laver1971}
Richard Laver.
\newblock On {F}ra\"\i ss\'e's order type conjecture.
\newblock {\em Ann. of Math. (2)}, 93:89--111, 1971.

\bibitem[LR05]{louros}
Alain Louveau and Christian Rosendal.
\newblock Complete analytic equivalence relations.
\newblock {\em Trans. Amer. Math. Soc.}, 357(12):4839--4866 (electronic), 2005.

\bibitem[MR12]{MR12}
Luca Motto~Ros.
\newblock On the complexity of the relations of isomorphism and
  bi-embeddability.
\newblock {\em Proc. Amer. Math. Soc.}, 140(1):309--323, 2012.

\bibitem[MR17]{lmrnew}
Luca Motto~Ros.
\newblock Can we classify complete metric spaces up to isometry?
\newblock {\em Boll. Unione Mat. Ital.}, 10(3):369--410, 2017.

\bibitem[NS15]{Nies-Solecki}
Andr{\'e} Nies and S{\l}awomir Solecki.
\newblock Local compactness for computable {P}olish metric spaces is
  {$\Pi^1_1$}-complete.
\newblock In Arnold Beckmann, Victor Mitrana, and Mariya Soskova, editors, {\em
  Evolving Computability}, volume 9136 of {\em Lecture Notes in Comput. Sci.},
  pages 286--290. Springer, Heidelberg, 2015.

\bibitem[NW65]{NashWil1965}
C.~St. J.~A. Nash-Williams.
\newblock On well-quasi-ordering infinite trees.
\newblock {\em Proc. Cambridge Philos. Soc.}, 61:697--720, 1965.

\bibitem[NW68]{NW67}
C.~St. J.~A. Nash-Williams.
\newblock On better-quasi-ordering transfinite sequences.
\newblock {\em Proc. Cambridge Philos. Soc.}, 64:273--290, 1968.

\bibitem[Par63]{Parovicenko1963}
I.~I. Parovi{\v{c}}enko.
\newblock On a universal bicompactum of weight {$\aleph $}.
\newblock {\em Dokl. Akad. Nauk SSSR}, 150:36--39, 1963.

\bibitem[PT14]{PT}
Prapanpong Pongsriiam and Imchit Termwuttipong.
\newblock Remarks on ultrametrics and metric-preserving functions.
\newblock {\em Abstr. Appl. Anal.}, pages Art. ID 163258, 9, 2014.

\bibitem[Ros05]{Rosend2005}
Christian Rosendal.
\newblock Cofinal families of {B}orel equivalence relations and quasiorders.
\newblock {\em J. Symbolic Logic}, 70(4):1325--1340, 2005.

\bibitem[Sil80]{Silver1980}
Jack~H. Silver.
\newblock Counting the number of equivalence classes of {B}orel and coanalytic
  equivalence relations.
\newblock {\em Ann. Math. Logic}, 18(1):1--28, 1980.

\bibitem[Sta85]{Stanley1985}
Lee~J. Stanley.
\newblock Borel diagonalization and abstract set theory: recent results of
  {H}arvey {F}riedman.
\newblock In {\em Harvey {F}riedman's research on the foundations of
  mathematics}, volume 117 of {\em Stud. Logic Found. Math.}, pages 11--86.
  North-Holland, Amsterdam, 1985.

\bibitem[Ver98]{vershik1998}
A.~M. Vershik.
\newblock The universal {U}ryson space, {G}romov's metric triples, and random
  metrics on the series of natural numbers.
\newblock {\em Uspekhi Mat. Nauk}, 53(5(323)):57--64, 1998.

\bibitem[Wol67]{Wolk}
E.~S. Wolk.
\newblock Partially well ordered sets and partial ordinals.
\newblock {\em Fund. Math.}, 60:175--186, 1967.

\end{thebibliography}
\end{document}